\documentclass{article}
\usepackage{amsmath,amsfonts,amsthm,epsfig,amssymb,color}
\graphicspath{{figures/}}
\newtheorem{theorem}{Theorem}[section]
\newtheorem{lemma}[theorem]{Lemma}

\newtheorem{conjecture}[theorem]{Conjecture}
\newtheorem*{claim*}{Claim}
\newtheorem{corollary}[theorem]{Corollary}
\numberwithin{figure}{section}

\theoremstyle{definition}
\newtheorem{remark}[theorem]{Remark}
\newtheorem{definition}[theorem]{Definition}
\newtheorem{example}[theorem]{Example}

\renewcommand\Pr{{\mathop{\mathbb P{}}\nolimits}}
\renewcommand\phi{\varphi}

\newcommand\noproof{\hfill$\Box$}
\newcommand\webcite[1]{\texttt{\def~{\~{}}#1}}
\newcommand\arxiv[1]{\\\webcite{http://arxiv.org/abs/#1}}

\newcommand\orlab{labelled}
\newcommand\orlabation{labelling}
\newcommand\orlabing{labelling}
\newcommand\plpa{non-crossing}
\newcommand\plp{non-crossing partition}

\newcommand\tR{{\widetilde R}}
\newcommand\eps{\varepsilon}
\newcommand\nrl{\ell}
\newcommand\pg{\succ}
\newcommand\pl{\prec}
\newcommand\floor[1]{\lfloor #1 \rfloor}
\newcommand\ceil[1]{\lceil #1 \rceil}

\newcommand\la{\lambda}
\newcommand\cc{{\mathrm{c}}}
\newcommand\E{\operatorname{\mathbb E{}}}

\newcommand\area{{\operatorname{\mathrm{area}}}}
\newcommand\Z{{\mathbb Z}}
\newcommand\RR{{\mathbb R}}
\newcommand\TT{{\mathbb T}}

\newcommand\bb[1]{\bigl(#1\bigr)}

\newcommand\dual{*}
\newcommand\PP{{\mathcal P}}
\newcommand\cC{{\mathcal C}}
\newcommand\cF{{\mathcal F}}
\newcommand\cH{{\mathcal H}}
\newcommand\cL{{\mathcal L}}
\newcommand\cLl{{\cL}_{m,n}}
\newcommand\cD{{\mathcal D}}
\newcommand\cU{{\mathcal U}}
\newcommand\cA{{\mathcal A}}
\newcommand\cHd{{\mathcal H}^*}
\newcommand\veca{{\bf a}}
\newcommand\vecb{{\bf b}}
\newcommand\vecp{{\bf p}}
\newcommand\vecq{{\bf q}}

\newcommand\vecx{{\bf x}}
\newcommand\vecy{{\bf y}}
\newcommand\ddh{\frac{{\mathrm d}}{{\mathrm d}h}}
\newcommand\dH{d_{\mathrm H}}
\newcommand\Sb{O_{\mathrm b}}
\newcommand\Hb{H_{\mathrm b}}
\newcommand\Hw{H_{\mathrm w}}
\newcommand\Vb{V_{\mathrm b}}
\newcommand\Vw{V_{\mathrm w}}
\newcommand\ALG{{\mathbb A}}
\newcommand\diam{\mathrm{diam}}
\newcommand\w{{\mathrm w}}
\newcommand\blk{{\mathrm b}}
\newcommand\rX{{X}}
\newcommand\Ed{{E_{\mathrm v}}}
\newcommand\strip{{T}}
\newcommand\ww{width}
\newcommand\hh{height}
\newcommand\xx{m}
\newcommand\yy{n}
\newcommand\off{w}
\newcommand\td{{\theta^*}}

\begin{document}
\title{Percolation on self-dual polygon configurations}

\date{May 11, 2010}

\author{B\'ela Bollob\'as%
\thanks{Department of Pure Mathematics and Mathematical Statistics,
Wilberforce Road, Cambridge CB3 0WB, UK}
\thanks{Department of Mathematical Sciences,
University of Memphis, Memphis TN 38152, USA}
\thanks{Research supported in part by NSF grants DMS-0906634,
 CNS-0721983 and CCF-0728928, and ARO grant W911NF-06-1-0076}
\and Oliver Riordan%
\thanks{Mathematical Institute, University of Oxford, 24--29 St Giles', Oxford OX1 3LB, UK}
}

\maketitle

\begin{abstract}
Recently, Scullard and Ziff noticed that a broad class of planar percolation models
are self-dual under a simple condition that, in a parametrized version of such a model,
reduces to a single equation. They state that the solution of the resulting
equation gives the critical point.
However, just as in the classical case of bond percolation on the square lattice,
self-duality is simply the starting point: the mathematical difficulty
is precisely showing that self-duality implies criticality.
Here we do so for a generalization of the models considered by Scullard and Ziff.
In these models, the states of the bonds need not be independent;
furthermore, increasing events need not be positively correlated,
so new techniques are needed in the analysis.
The main new ingredients are a generalization of Harris's Lemma
to products of partially ordered sets, and a new proof of a type of
Russo--Seymour--Welsh Lemma with minimal symmetry assumptions.
\end{abstract}

\tableofcontents

\section{Introduction}

In 1963 Sykes and Essam~\cite{SykesEssamPRL} noticed that, in independent bond percolation, a star
with bond probabilities $p_1$, $p_2$ and $p_3$ may be replaced by a triangle with
bond probabilities $r_1$, $r_2$ and $r_3$, provided the $p_i$ and $r_i$ satisfy
certain equations; in particular, a star in which
each bond has probability $p_0=1-2 \sin (\pi/18)$ of being open may be replaced by a triangle
with bond probabilities $1-p_0$. Sykes and Essam went on to use this
star-triangle transformation to predict that $p_0$ and $1-p_0$ should be the critical probabilities
for bond percolation on the hexagonal and triangular lattices, respectively.
In 1981 Wierman~\cite{WiermanHT} gave a rigorous
proof of this result with the aid of a Russo--Seymour--Welsh-type theorem.
In 1982 Kesten~\cite{KestenBook} extended Wierman's theorem to describe
the `critical surface' of weighted bond percolation on the
triangular lattice, although the details were worked out only in 1999
by Grimmett~\cite{GG99}. Later, in 2008, it was shown~\cite{kfold} that the
sum of the critical probabilities of a centrally symmetric planar lattice and its
dual is 1; more generally, an analogous statement allowing for percolation
with different probabilities for different bonds was proved,
giving the Kesten--Grimmett theorem as an easy corollary.

In 1984 Wierman~\cite{Wierman84} used the general star-triangle transformation to determine
the exact critical probability for bond percolation (with equal bond probabilities)
on a lattice obtained from the square lattice by adding some diagonals. Using
a simpler transformation, Suding and Ziff~\cite{SudingZiff}
deduced the critical probability for site percolation
on the extended Kagom\'e lattice from Wierman's result for the hexagonal lattice.

Recently, extending work of Scullard~\cite{Scullard}
and Chayes and Lei~\cite{CL},
Ziff~\cite{Ziff_cdc} and Ziff and Scullard~\cite{ZS_exactbond}
proposed a simple criterion
predicting the value of the critical probability for a wide
variety of percolation models in the plane.
In addition to the usual independent site and/or bond percolation
models on a number of lattices, these models
include cases which can be seen as bond percolation with
local dependencies between the states of certain bonds.

The predictions of Scullard and Ziff are similar in nature
to those of Sykes and Essam~\cite{SykesEssamPRL} mentioned above:
having shown that for a certain
probability $p_0$ the percolation model is `self-dual', they
state that this probability $p_0$ is therefore critical.
Mathematically, there is a folklore `conjecture' (with, as far as we are
aware, no precise formulation) stating that any `reasonable' 
self-dual planar model is critical. This conjecture is still wide open.
It is well known to hold in certain special cases,
in particular for site or bond percolation on lattices with certain
symmetries, such as reflection in a line (see Kesten~\cite{KestenBook}),
or rotational symmetry of any order, as shown in~\cite{kfold}. More generally,
as remarked in~\cite{kfold}, it can be shown for site or bond percolation on any lattice by
combining results of Sheffield~\cite[Corollary 9.4.6]{Sheffield},
Aizenman, Kesten and Newman~\cite{AKNunique} and Menshikov~\cite{Menshikov}.

Chayes and Lei~\cite{CL} independently described a special
case of the Scullard--Ziff criterion (as well as a generalization
to random cluster models), and gave a sketch proof
of criticality under an extra assumption.
In a recent preprint, Wierman and Ziff~\cite{WZ}
proved criticality in certain special cases,
using known results on self-dual planar lattices.

In this paper we shall prove that the Scullard--Ziff criterion does indeed give the 
critical point for a wide variety of planar percolation models.
In the original papers in the physics literature, the exact
scope of applicability of the criterion is not entirely clear. In this paper
we shall define \emph{precisely} a general class of models that
are self-dual in the appropriate sense, and
use new methods to show that the self-dual point is indeed
critical in \emph{all} cases.

Although full definitions will be given only in the next section, let us illustrate
some simple special cases of our main result, starting
with one very concrete (but rather specific) example, and then turning to a more general family.

\begin{figure}[htb]
 \centering
 \input{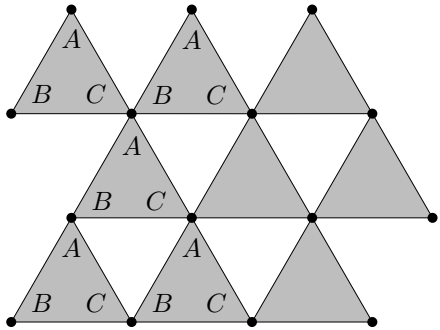}
 \caption{The triangular lattice, with alternate triangles shaded.}\label{fig_simpletri}
\end{figure}

\begin{example}\label{ex1}
Consider the usual triangular lattice shown in Figure~\ref{fig_simpletri}.
Given a parameter $0<p<1$, we initially select each bond (i.e., edge)
independently with probability
$p$. But then bonds within a shaded triangle `compete': if there are two bonds,
the first in the clockwise order `wins', and the other is deleted. However,
if all three bonds are present there is a standoff, and all three remain.
This results in a configuration of bonds such that, within each shaded triangle,
with probability $(1-p)^3$ no bonds are present, with probability $p^3$
all three are present, and otherwise exactly one bond is present, with
each of these cases having probability $\frac{1}{3}(1-p^3-(1-p)^3)=p(1-p)$.
The configurations in different triangles are of course independent.

Applying our main results to this particular model we shall
see that (with probability 1) the
remaining bonds form a graph containing an infinite component
if and only if $p>1/2$; in other words, the model \emph{percolates}
if and only if $p>1/2$. Note for later that considering a single shaded
triangle, the probability that the vertices $A$ and $B$ are connected
within this triangle is $p(1-p)+p^3$, as is the probability that
$B$ and $C$ are connected. The probability that both events
hold is $p^3$, since this happens if and only if all three bonds are selected.
\end{example}

\begin{figure}[htb]
 \centering
 \input{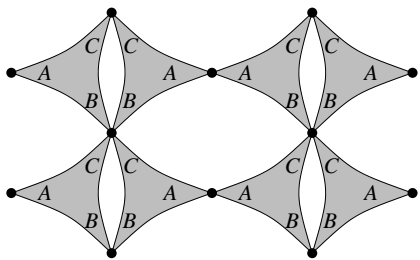}
 \caption{A lattice of labelled triangles with no axes of symmetry, and no rotational symmetries.}\label{fig_reflsym_1}
\end{figure}

\begin{example}\label{ex2}
More generally, consider, for example, either the usual triangular lattice as above (viewed
as an arrangement of shaded triangles), or the lattice of triangles a small
part of which is illustrated in Figure~\ref{fig_reflsym_1}. These are
both examples of \emph{self-dual hyperlattices}, to be defined 
in the next section. Suppose that
each shaded triangle contains some mechanism that connects certain subsets
of its vertices, with these processes independent in different triangles;
we have given one example above, but the mechanism is irrelevant, only
the final connection probabilities matter.
Suppose that in each triangle the probability that all vertices are connected
(inside the triangle)
is $p_{ABC}$, that none are connected is $p_\emptyset$, that $A$ and $B$
are connected to each other but not to $C$ is $p_{AB}$, and so on, with these probabilities the same
for all triangles. Then, except possibly in the degenerate case $p_{ABC}=p_\emptyset=0$,
there is (with probability 1) an infinite connected component
if and only if $p_{ABC}>p_\emptyset$.
\end{example}

Even the very special cases described above are outside the scope of existing results,
for several reasons. Firstly, in general they do
not correspond to independent bond or site percolation on any lattice
(as required in~\cite{WZ}). This is certainly the case when increasing
events are not positively correlated; see the discussion in Section~\ref{sec_mod}. 
Such correlation is absent in Example~\ref{ex1} when $p$ is equal to or close to $1/2$,
since $p^3 < (p(1-p)+p^3)^2$ when $p=1/2$.
Positive correlation is also required for the arguments in~\cite{CL}.
Secondly, in Example~\ref{ex2} there is no rotational or reflectional symmetry, as required
in~\cite{WZ} and in~\cite{CL}, so existing Russo--Seymour--Welsh-type results do not
apply. 

Our results show that self-duality implies criticality in a context
that is much broader than the Scullard--Ziff criterion; an example
is given in Figure~\ref{fig_square}.

Mathematically, the main interest of this paper
is perhaps in the development of new tools needed to analyze the general
model, including new proofs of analogues of the Russo--Seymour--Welsh
Lemma, and a generalization of Harris's Lemma. The rest of the paper is organized
as follows. In Section~\ref{sec_mod} we describe the model we shall study, and
state our main results.
In Section~\ref{sec_Harris} we present a generalization of Harris's Lemma to
products of posets. In Section~\ref{sec_cc} we prove various technical
results about the model. The heart of the paper is Section~\ref{sec_RSW},
where we prove a new RSW-type lemma; the proof is rather involved,
so we first illustrate the ideas in a simpler setting (bond percolation
on $\Z^2$) in Subsection~\ref{ss_z2}. In Section~\ref{sec_RSWapp} we show
how to apply this lemma using only the minimal symmetry guaranteed by self-duality.
Finally, in Section~\ref{sec_deduce} we show that (as in~\cite{ourKesten}), using
a suitable sharp-threshold result, it is but a small step from rectangle
crossings to the final results. In Section~\ref{sec_crit} we conclude with a brief discussion
of the behaviour of the model in the critical case.

\section{The model and results}\label{sec_mod}

The starting point of our investigation is
an embedding of a suitable hypergraph in the plane,
as described implicitly by Ziff and Scullard~\cite{ZS_exactbond} and explicitly
by Wierman and Ziff~\cite{WZ}.
In these papers the hypergraphs considered are
3-uniform, but much of the time there is no need for this restriction.
Since the concept of a plane hypergraph seems fundamental, we give several
equivalent definitions. To avoid irrelevant difficulties we always
assume piecewise linearity of all curves.

By a \emph{polygon} we mean a piecewise-linear closed curve $P$
in the plane that, if it touches itself at all, does so only externally
at some vertices. (To be pedantic, $P$ is the image of some regular $n$-gon $P'$
under a piecewise linear map defined on the closed domain bounded by $P'$
that is injective except possibly at the vertices of $P'$.)
Such a polygon surrounds (more precisely winds around) a simply connected
open set that we call its \emph{interior}.

By a \emph{plane hypergraph} $\cH$ we mean a set of points of $\RR^2$, the \emph{vertices},
together with a set of polygons, the \emph{hyperedges},
with the following properties:

\begin{enumerate}
\renewcommand{\theenumi}{\textup{(\roman{enumi})}}
\renewcommand{\labelenumi}{\theenumi}
\item\label{p1} any bounded subset of the plane
contains only finitely many vertices, and meets only finitely many hyperedges,

\item the interiors of the hyperedges are disjoint,

\item each hyperedge is incident with at least one vertex, and

\item hyperedges meet themselves or each other only at vertices.
\end{enumerate}
Note that we allow a hyperedge to meet the same vertex several times,
as in Figure~\ref{fig_hypergraph}, simply because there turns out to be no reason not to.
Property \ref{p1} ensures that each vertex meets only finitely many hyperedges,
and vice versa. In this paper, all plane hypergraphs we consider are connected,
in the natural sense.
\begin{figure}[htb]
 \centering
\[ \epsfig{file=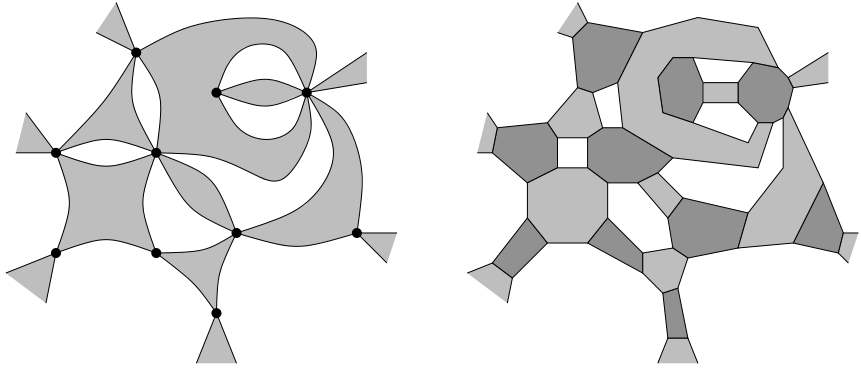}  \]                                             
 \caption{On the left is part of a plane hypergraph $\cH$: the shaded regions are the
(interiors of the) polygons
corresponding to hyperedges. The same picture may be seen as a proper 2-colouring
of the faces of a plane graph. On the right is the corresponding 3-coloured cubic map.
Note that not all edges are drawn as straight lines (see the top
right of the figure).}\label{fig_hypergraph}
\end{figure}

Two plane hypergraphs are \emph{isomorphic} if there is a homeomorphism from
the plane to itself mapping one into the other, in the obvious sense.
Of course, given a plane hypergraph (or indeed an isomorphism class
of plane hypergraphs) there is a corresponding abstract hypergraph;
more precisely, a
(multi-)hyper(-multi-)graph, where we have included `multi' twice
to indicate that two or more hyperedges may be incident with the same set of vertices, and a single
hyperedge may be incident with a vertex more than once.  However, we shall
work directly with the drawings throughout.

Plane hypergraphs are natural generalizations of plane (multi)graphs.
Indeed, we may think of a plane graph as a representation
of an abstract graph, with the vertices represented by points, and
the edges by connected sets meeting only at the vertices.
The edges of a graph are naturally represented by minimal sets
connecting the corresponding vertices, i.e., simple curves, but in the hypergraph case polygons are more
natural, so we use them even for hyperedges with only two vertices (or indeed,
one vertex). When we come to percolation in a moment, the idea is that
instead of each edge in a graph randomly
either connecting its vertices or not, independently of the other edges, each
hyperedge will randomly connect some subsets of its vertices, independently
of the other hyperedges.

It is easy to check that (connected, as always in this paper)
plane hypergraphs $\cH$ correspond exactly to
\emph{shaded} locally finite connected plane multigraphs $G$, where the faces of $G$ are
properly coloured grey and white, so that every edge borders faces of
different colours, with any unbounded faces white. (At this stage
there may be one or more unbounded faces; in the bulk of the paper
all faces will be bounded.)
Indeed, vertices correspond to vertices, and the hyperedges of $\cH$ are
simply the grey faces of $G$, as on the left in Figure~\ref{fig_hypergraph}.
Of course, a plane graph $G$ with at most one unbounded face has an appropriate
shading if and only if every degree is even, and then it has either
one or two shadings, depending on whether or not it has an unbounded face.

Note that if $e$ is a hyperedge incident with $|e|$ vertices (counting
multiplicity), then, as a polygon, $e$ is made up of $|e|$ segments joining vertices,
corresponding to the edges of $G$.
We cannot in general draw these segments as straight lines. For example,
$\cH$ may contain triples $uvw$ and $uvw'$, say; furthermore,
there may be further hyperedges inside the region bounded by the two curves
joining $u$ and $v$ associated to these triples.

By a \emph{face} of a plane hypergraph $\cH$ (defined as above)
we mean a component of what is left of the plane after removing all hyperedges
and their interiors, i.e., a white face of the corresponding graph $G$.
For our next few definitions (in particular that of duality) to make
sense, it is convenient to insist that each face of $\cH$ has finitely
many edges in its boundary. This is equivalent to imposing the condition
that if $\cH$ is infinite, then it has no unbounded faces.
(Of course, if $\cH$ is finite, then it necessarily has exactly one
unbounded face.) When it comes to percolation,
we naturally consider only infinite $\cH$.

There is yet another way of defining plane hypergraphs, which will turn out to be much
more convenient to work with, but is at first sight perhaps less natural (at
least for percolation).
Given a (connected, as usual) plane hypergraph $\cH$ and the corresponding
graph $G$ as above,
replace each vertex $v$ of $\cH$ by a \emph{black} $2d$-gon, where
$d=d(v)$ is the degree of $v$, each face by a \emph{white} $2d$-gon, where
$d$ is the number of edges (of $G$) bounding the face,
and each hyperedge $e$ incident with $d$ vertices (counted
with multiplicity) by a \emph{grey} $2d$-gon. In this way we obtain
a cubic planar map $M$ in which the faces are properly coloured black, white
and grey; see Figure~\ref{fig_hypergraph}. (As before, we cannot necessarily
draw the edges of the polygons as straight lines.)

The reverse transformation is even simpler: starting from a cubic map $M$
(i.e., a connected locally finite 3-regular plane graph in which
each face has finitely many edges in its boundary)
in which the faces are properly coloured black, white and grey,
we simply contract each black face to a point to form a vertex
of $\cH$, and take the grey faces to form the hyperedges.
In what follows we shall refer to such a coloured map $M$ as a \emph{map hypergraph}
(or simply a hypergraph, when there is no danger of confusion),
and denote it also by $\cH$.

Note that when $\cH$ is a graph (i.e., $\cH$ is 2-uniform),
then the corresponding
map is the one considered in Chapter 3 of~\cite{BRbook} (see Figures
2 and 4, for example). There, the 4-gon corresponding to an edge
is coloured black or white according to whether the edge
is open or closed; here, the $2|e|$-gon corresponding to 
a hyperedge $e$
is grey for now, but will
be coloured with a mixture of black and white later.

By a \emph{plane hyperlattice} $\cH$ we shall mean an infinite connected
plane hypergraph (defined in any of the three ways above)
with a lattice $\cL$ of translational symmetries, i.e., such that there are
linearly independent vectors $\veca$ and $\vecb$ with the property
that translation of the plane through either vector maps the drawing into
itself in the obvious sense, corresponding to an isomorphism of the underlying
hypergraph.
Throughout, we view $\cL=\{m\veca+n\vecb:m,n\in\Z\}$ as a subset of $\RR^2$.
More formally, we define a \emph{plane hyperlattice} to be a pair
$(\cH,\cL)$ as above, since in what follows $\cL$ need not be the full lattice
of translational symmetries of $\cH$; in spite of this, we usually omit
$\cL$ from the notation.
Naturally, when we consider isomorphisms of plane hyperlattices,
these are required to preserve the corresponding lattices of symmetries.
More precisely, a homeomorphism $S:\RR^2\to \RR^2$ is an isomorphism
from the plane hyperlattice $(\cH,\cL)$ to $(\cH',\cL')$ if it 
corresponds to a plane hypergraph isomorphism and satisfies
$S(x+\ell)=S(x)+T(\ell)$ for all $x\in \RR^2$ and $\ell\in \cL$,
where $T$ is a linear map with $T(\cL)=\cL'$.

In the context of percolation, the natural notion of the dual of a plane
hypergraph $\cH$ turns out to be the plane hypergraph $\cHd$ defined as follows.
Take a vertex of $\cHd$ inside each face of $\cH$.
To obtain the hyperedges of $\cHd$, replace each hyperedge $e$ of $\cH$
by the \emph{dual hyperedge} $e^*$ joining the vertices corresponding to the faces
that $e$ meets, as in Figure~\ref{fig_dual}.
\begin{figure}[htb]
 \centering
 \input{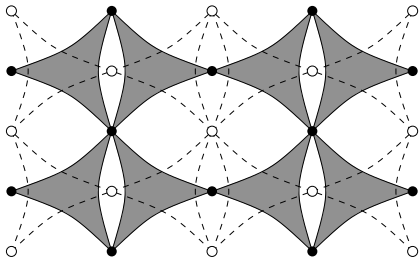}
 \caption{Part of a 3-uniform plane hyperlattice (filled circles and shaded triangles)
and its dual (open circles and dashed triangles). The hyperedges of the dual
are the concave dashed triangles.}\label{fig_dual}
\end{figure}
In the special case of 2-uniform hypergraphs, i.e., graphs, this
is the usual notion of planar duality.
In the 3-uniform case considered (with mild additional restrictions)
by Wierman and Ziff~\cite{WZ},
the notion of duality just defined coincides with theirs;
the description of Ziff and Scullard~\cite{ZS_exactbond} seems to be ambiguous.

In terms of the shaded graph $G$, the operation of taking the dual is rather complicated --
it is \emph{not} simply reversing the shading. However, in the 3-coloured map formulation,
it is very simple to construct the dual: simply exchange black and white. Indeed, one can think
of a hyperedge and its dual together as corresponding to a grey face
of the map $M$ (for example, in Figure~\ref{fig_dual} one can take the intersection
of $e$ and $e^*$ as the grey face); the vertices of $\cH$ correspond
to the black faces of $M$, and the vertices of $\cHd$ to the white faces of $M$.

Of course, choosing the drawing appropriately, we may take the dual
$\cHd$ of a plane hyperlattice $\cH$ to be a plane hyperlattice, and the dual of $\cHd$
to be $\cH$. A plane hyperlattice $\cH$ is \emph{self-dual} if $\cHd$ is isomorphic
to $\cH$; examples are shown in Figures~\ref{fig_dual},~\ref{fig_tri} and~\ref{fig_square}.

If $P$ is a polygon in the plane,
then by a \emph{\plp}
of its vertex set $V(P)$  we mean
a partition $\pi$ such that no two distinct parts of $\pi$ contain
interlaced pairs of vertices: if $x$, $y$, $z$, $w$ are four distinct
vertices appearing in this cyclic order around $P$,
a \plp\ $\pi$ cannot contain two parts one
of which includes $x$ and $z$, and the other
$y$ and $w$. Equivalently, a partition $\pi$ of $V(P)$ is a \plp\ if
and only if it may be realized by constructing disjoint (path-)connected
subsets $S_i$ of $P$ (which we take to include its interior)
so that each part of $\pi$ is the intersection
of some $S_i$ with $V(P)$.

The \emph{dual} $\pi^*$ of a \plp\ $\pi$ is the \plp\ of
the \emph{edges} of $P$ in which two edges $e$ and $f$ are in the same
part if and only if they are not interlaced with two vertices
$x$ and $y$ in a common part of $\pi$. Here interlaced means
that the edges and vertices occur in the cyclic order $e$, $x$, $f$, $y$
or its reverse.

Turning finally to percolation, the \emph{state} of a hyperedge $e$
will simply be a \plp\ of its vertices. (More precisely,
of the incidences of $e$ with its vertex set, so if $e$ touches
itself, the relevant vertex appears multiple times in the groundset
of the partition.) A \emph{configuration} $\omega$ is an assignment
of a state to each hyperedge of the hypergraph $\cH$ under consideration.
We think of the state of a hyperedge $e$ as describing connections
within $e$. In particular, by an \emph{open path}
in a configuration $\omega$ we mean a sequence $v_0e_1v_1e_2\ldots e_\ell v_\ell$
such that, for each $i$, the partition 
of the vertices of $e_i$ has a part containing both $v_{i-1}$ and $v_i$.
Two vertices are \emph{connected} in $\omega$ if they are joined by an open
path, and the \emph{open clusters} of $\omega$ are the maximal connected
sets of vertices.

Finally, a \emph{hyperlattice percolation model} consists
of a plane hyperlattice $(\cH,\cL)$ together with a probability measure
on configurations on $\cH$ such that the states of different
hyperedges are independent, and the measure is preserved by
the action of $\cL$. In other words, if $e'$ is a translate
of $e$ under an element of $\cL$, then corresponding states
in $e'$ and $e$ have the same probabilities. Note that for
a single hyperedge $e$, all probability distributions on the set of \plp s
associated to $e$ are allowed.

As usual, the sigma-field of measurable events is the one generated by
cylindrical sets, i.e., by events depending only on the states of a finite
set of hyperedges. In fact, except when defining percolation, throughout
this paper we can
work with large enough finite regions of $\cH$,
so there are no issues of measurability.

The \emph{dual} $\omega^\dual$ of a configuration $\omega$ on $\cH$
is the configuration on $\cHd$ in which the state of $e^*$
is the dual of the state of $e$ (noting that vertices
of $e^*$ correspond to edges of the polygon $e$). It is not hard
to check that finite open clusters in $\omega$ are
surrounded by open cycles in $\omega^*$
and \emph{vice versa}; this is most easily
seen in the colouring formulation described at
the start of Section~\ref{sec_cc}.

Given a plane hyperlattice $(\cH,\cL)$, suppressing
$\cL$ in the notation as usual,
the hyperlattice percolation models on $\cH$ may be
parametrized as follows. First pick
one representative $e_i$ of each orbit of the action of $\cL$
on the hyperedges. Then for each \plp\ $\pi$ of
the vertices of $e_i$, choose a probability $p_{i,\pi}$,
subject only to $\sum_\pi p_{i,\pi}=1$ for each $i$.
We call such a vector $\vecp=(p_{i,\pi})_{i,\pi}$ a \emph{probability vector} (for $\cH$),
and write $\cH(\vecp)$ for the corresponding percolation model.
The \emph{dual vector} $\vecp^*$ assigns the probability $p_{i,\pi}$
to the partition $\pi^*$ of $e_i^*$, so $\cHd(\vecp^*)$ is a hyperlattice
percolation model on $\cHd$. 

The hyperlattice percolation model $\cH(\vecp)$ is \emph{self-dual}
if $\cH(\vecp)$ and $\cHd(\vecp^*)$ are isomorphic, i.e., if there
is an isomorphism from $\cH$ to $\cHd$ such if $e\in E(\cH)$ and $f\in E(\cHd)$
correspond under the isomorphism, then each partition $\pi$ of $f$
has the same probability in $\cHd(\vecp^*)$ as the corresponding partition
of $e$ does in $\cH(\vecp)$. Our aim is to show that self-dual hyperlattice
models are `critical', but first we must define what critical means.

The set of partitions of a (here finite) set $S$ forms
a poset $\PP$ in a natural way: we have $\pi\preccurlyeq \pi'$ if
any two elements in the same part of $\pi$ are in the same part
of $\pi'$, i.e., the parts of $\pi'$ are unions of those of $\pi$,
i.e., if $\pi$ \emph{refines} $\pi'$.

Given a hyperedge $e_i$ as above, let $\PP=\PP_{e_i}$ be the poset
formed by the \emph{\plpa} partitions
of the vertices of $e_i$.
An \emph{upset} $\cU$ in $\PP$ is a subset of $\PP$ such that if $\pi\in \cU$
and $\pi\prec \pi'$ then $\pi'\in \cU$.
Given an upset $\cU$ in $\PP_{e_i}$ and a probability vector $\vecp$,
let $p_i(\cU)=\sum_{\pi\in \cU}p_{i,\pi}$ denote the probability that
the state of $e_i$ is in $\cU$.
Given two probability vectors $\vecp$ and $\vecq$, we say that
$\vecq$ \emph{dominates} $\vecp$
if $q_i(\cU)\ge p_i(\cU)$ for each $i$ and each upset $\cU\subset \PP_{e_i}$.
We say that $\vecq$ \emph{strictly dominates} $\vecp$, and write $\vecq\pg\vecp$,
if $q_i(\cU)> p_i(\cU)$ for each $i$ and each non-trivial upset $\cU\subset \PP_{e_i}$,
i.e., for all upsets apart from
$\cU=\emptyset$ and $\cU=\PP_{e_i}$. Note that we can have $\vecp\ne\vecq$ such
that $\vecq$ dominates $\vecp$ but does not strictly dominate it.

Hall's theorem implies that $\vecq$ dominates
$\vecp$ if and only if $\vecq$ can be obtained from $\vecp$ by moving
`probability mass' from elements $p_{i,\pi}$ to elements $p_{i,\pi'}$ with $\pi\prec \pi'$.
In the case
of strict domination, we can assume that, for each $i$, a non-zero mass is moved from
each $\pi$ to each $\pi'\succ \pi$.

A percolation model $\cH(\vecp)$ \emph{percolates} if the probability that the
open cluster containing any given vertex is infinite is positive. As usual,
this is equivalent to the existence with probability 1 of an infinite open cluster.
The model $\cH(\vecp)$ is \emph{critical} if two conditions hold:
for any $\vecq\pg\vecp$ the model $\cH(\vecq)$ percolates, and for any
$\vecq\pl\vecp$ the model $\cH(\vecq)$ does not percolate.

We say that the model $\cH(\vecp)$ exhibits \emph{exponential decay} (of the volume)
if there is a constant $\alpha>0$ such that for any fixed vertex $v$
the probability that the open cluster containing $v$ contains at least
$n$ vertices is at most $e^{-\alpha n}$ for all $n\ge 2$. Our main result is the following.

\begin{theorem}\label{th1}
Let $\cH(\vecp)$ be a self-dual hyperlattice percolation model.
Then for any $\vecq\pg\vecp$ the model $\cH(\vecq)$ percolates,
and for any $\vecq\pl\vecp$ the model $\cH(\vecq)$ exhibits exponential decay.
In particular, $\cH(\vecp)$ is critical.
\end{theorem}

Although this is far from the main point, a very special case is that the
critical probability for bond percolation on any self-dual planar lattice
is $1/2$. Note that the condition $\vecq\pg\vecp$ is stronger than what one
could hope for, namely a similar result with this condition
replaced by $\vecq$ dominating $\vecp$ and $\vecq\ne \vecp$. However, one
would then need to rule out degenerate special cases (corresponding
to increasing the probability of bonds that are `dead ends' in a bond percolation model,
for example).
Also, Theorem~\ref{th1} with the present conditions
is strong enough for the main application, Corollary~\ref{c1} below.

It turns out that in proving Theorem~\ref{th1}, we do not require
an exact isomorphism between $\cH(\vecp)$ and $\cHd(\vecp^\dual)$.
We call two hyperlattice percolation models $\cH_1(\vecp_1)$ and $\cH_2(\vecp_2)$
\emph{equivalent} if they can be coupled so that, for some constant $C$,
for every open path $P$ in either model there is an open path $P'$
in the other model at Hausdorff distance at most $C$ from $P$.
Roughly speaking, the typical reason for two models to be equivalent is
that they can be viewed as different
ways of realizing a single underlying model.

For example, consider a plane triangulation $G$ with a lattice $\cL$ of translational symmetries,
such as the triangular lattice. Then there are two natural ways to form a hyperlattice
from $G$, illustrated (for a more complicated lattice)
in Figure~\ref{fig_tri_rel}.
\begin{figure}[htb]
\[
 \epsfig{file=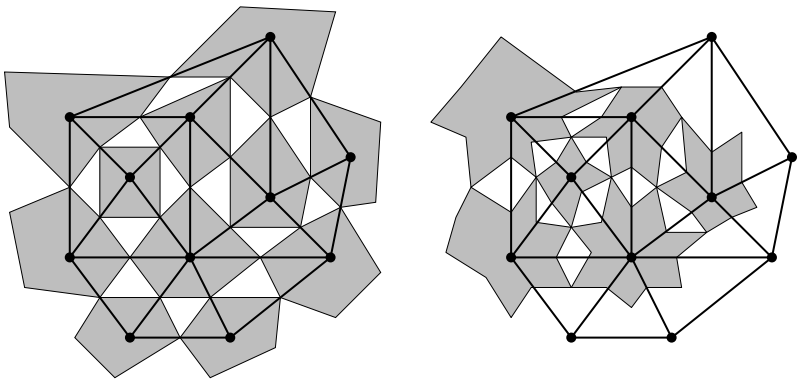}
\]
\caption{The filled circles and thick lines show part of a planar triangulation $G$,
which is assumed to have a lattice of symmetries not visible on this scale. On the left is part
of the hyperlattice $\cH$ formed from $G$ by taking a vertex for each edge,
and a hyperedge for each vertex. On the right is part of the hyperlattice $\cH'$
formed by taking a vertex for each face of $G$, and a hyperedge for each vertex. Note
that $\cH$ and $\cH'$ are dual as hyperlattices.}\label{fig_tri_rel}
\end{figure}
In the first, there is a vertex for each
edge of $G$, in the second, a vertex for each face of $G$. In either case, there is a hyperedge
for each vertex $v$ of $G$; this hyperedge is incident to all vertices
corresponding to edges or faces of $G$ that $v$ is incident to.
Within each hyperedge, assign probability $p$ to the partition in which all vertices
are in a single part, and $1-p$ to that in which every vertex is in a separate part.
Let us call a hyperedge \emph{open} if we select the partition into one part,
and \emph{closed} otherwise.
Then the resulting models $\cH(p)$ and $\cH'(p)$ are equivalent using the natural coupling,
i.e., the coupling in which the hyperedges in the two models corresponding to a vertex $v$ of $G$ have the same state,
open or closed.
Indeed, in either model an open path of length more than 1 consists of a sequence of open hyperedges
with consecutive ones sharing (hyperlattice) vertices. But two hyperedges share a vertex
in $\cH$ if and only if the corresponding vertices of $G$ are joined by an edge,
and share a vertex in $\cH'$ if and only if the corresponding vertices of $G$ are in a common
face. In a triangulation, two vertices are in a common face if and only if they are joined by an edge.

Emphasizing the lattice in the notation, for a change,
we say that a hyperlattice percolation model $(\cH(\vecp),\cL)$ is \emph{approximately
self-dual} if there is a model $(\cH'(\vecp'),\cL)$ equivalent to $(\cHd(\vecp^*),\cL)$
such that $(\cH'(\vecp'),\cL)$ and $(\cH(\vecp),\cL)$ are isomorphic as plane hyperlattices,
with the corresponding linear map $T$ an isometry of the plane.
The last restriction is a technicality: in the case of (exact) self-duality
we did not impose it, but as we shall see in Lemma~\ref{lsymclass}, any isomorphism
witnessing self-duality has this property (after a suitable
affine transformation). This is presumably also true for approximate
self-duality, but as the condition will (we believe) self-evidently hold in any applications,
we do not bother checking this. For our proofs, approximate self-duality
is (apart from one technicality)
just as good as self-duality, so we obtain the following strengthening
of Theorem~\ref{th1}. In this result `malleability' is a technical
condition we shall introduce later (see Definitions~\ref{maldef} and~\ref{maldef2}); in the 3-uniform case, all
hyperlattice percolation models are malleable.
Also, any `site percolation' model, where only the partition
into singletons and that into a single part occur, is malleable.

\begin{theorem}\label{th2}
Let $\cH(\vecp)$ be a malleable approximately self-dual hyperlattice percolation model.
Then for any $\vecq\pg\vecp$ the model $\cH(\vecq)$ percolates,
and for any $\vecq\pl\vecp$ the model $\cH(\vecq)$ exhibits exponential decay.
In particular, $\cH(\vecp)$ is critical.
\end{theorem}

To indicate that this extension may be useful, consider a plane triangulation $G$
with a lattice of translational symmetries, and a real number $0<p<1$.
Consider the two hyperlattice percolation models $\cH(p)$ and $\cH'(p)$
corresponding to site percolation on $G$, defined 
as above. Then $\cH$ and $\cH'$ are dual to each other,
so $\cH(p)$ and $\cH'(1-p)$ are dual as hyperlattice percolation models. Since, as noted
above, $\cH(p)$ and $\cH'(p)$ are equivalent, we see that $\cH(1/2)$
is approximately self-dual, so Theorem~\ref{th2} implies that the critical
probability for site percolation on $G$ is $1/2$. In itself this is not new (see
the discussion in the introduction),
but it indicates that the models to which Theorem~\ref{th2} applies
include ones with a site percolation `flavour'.

We have just seen that site percolation on a triangulation with a
lattice of symmetries, which for $p=1/2$ is easily seen to be
self-dual in an appropriate sense, may be transformed
to a hyperlattice percolation model that is only approximately self-dual.
In this case, the duality is clearer in the site percolation formulation
than the hyperlattice one.
Unsurprisingly, there are also cases where the reverse holds.
Indeed, for any plane hyperlattice $\cH$,
selecting only the partitions in which all vertices in a given
edge are connected or none are, we obtain a site percolation model on
an (in general non-planar) graph. For example, taking $\cH$ as in Figure~\ref{fig_reflsym_1},
one obtains the non-planar graph in Figure~\ref{fig_nonpl}.
\begin{figure}[htb]
\centering
\[ \epsfig{file=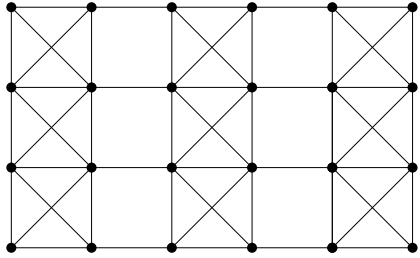} \]
\caption{A non-planar lattice which is `self-dual' for site percolation.}\label{fig_nonpl}
\end{figure}
Theorem~\ref{th1}
shows that if $\cH$ is self-dual, then the critical probability for site percolation on
the resulting graph is $1/2$.

Returning to the original motivation,
the key observation of Ziff~\cite{Ziff_cdc} and Ziff and Scullard~\cite{ZS_exactbond}
(present also in the original papers of Scullard~\cite{Scullard} and
Chayes and Lei~\cite{CL}
in the special case where $\cH$ is the `triangular hyperlattice' $T$
shown in Figure~\ref{fig_tri})
is that if $\cH$ is 3-uniform, then the condition for self-duality becomes
very simple, at least if one takes the same connection probabilities
in all triangles. Unfortunately, this involves some further definitions, 
to allow for non-symmetric cases.

By a \emph{\orlab\ plane hyperlattice} we mean a plane hyperlattice
in which the vertices around each hyperedge are labelled $1,2,\ldots,k$
in a way that is consistent with either the clockwise or anti-clockwise cyclic order
within each hyperedge, and is globally consistent with the lattice $\cL$ of translational
symmetries.
When $k=3$ or $k=4$ we often use letters $A$, $B,\ldots,$ to denote the labels
rather than numbers.
In the 3-uniform case, assigning labels as above amounts to designating the vertices
of a hyperedge $e$ as its $A$-, $B$- and $C$-vertex in any order,
as in~\cite{WZ}.
The simplest example of a \orlab\ plane hyperlattice
is the \emph{\orlab\ triangular hyperlattice} shown in Figure~\ref{fig_tri}.
Another example is illustrated in Figure~\ref{fig_reflsym}.
\begin{figure}[htb]
 \centering
 \input{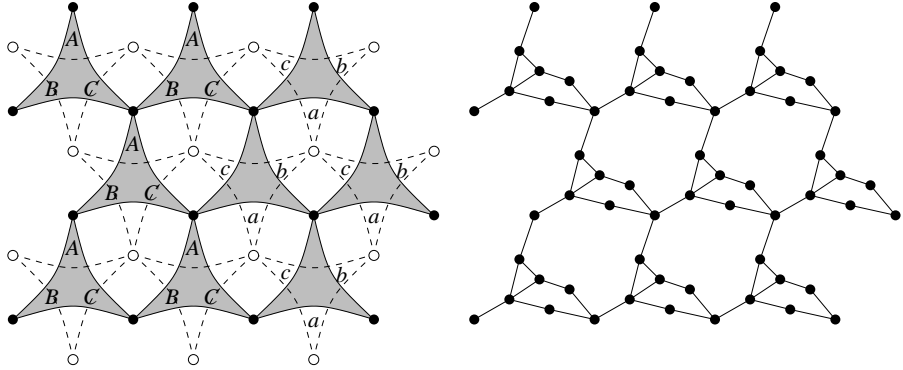}
 \caption{The shaded triangles depict the
\emph{\orlab\ triangular hyperlattice} $T$, i.e., the \orlab\ 3-uniform hyperlattice
obtained by taking alternate triangles in the triangular lattice, and \orlabing{}
them consistently; capital letters denote the \orlabation{} of $T$.
The dashed lines and lower case letters depict the dual $T^*$. 
Note that $T$ is self-dual: there is a rotation through $\pi$
mapping $T$ to $T^*$.
The figure on the right shows a lattice $L$ obtained by substituting
a `generator' into each hyperedge of $T$.}\label{fig_tri}
\end{figure}

\begin{figure}[htb]
 \centering
 \input{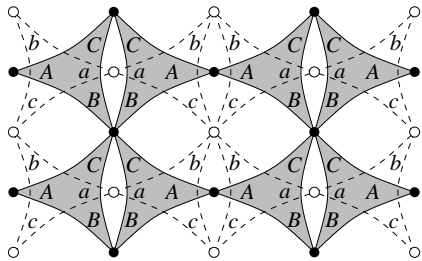}
 \caption{A \orlab\ 3-uniform hyperlattice with only translational symmetries
that is isomorphic to its dual by reflection in a horizontal line. For visual clarity,
the \orlabation{} of the dual is indicated with lower case letters.}\label{fig_reflsym}
\end{figure}

Note that the $A$-vertex of one hyperedge may be the $B$-vertex
of another hyperedge (or indeed, of the same hyperedge if it touches itself).
Given a $k$-uniform \orlab\ plane hyperlattice and a probability
vector $\vecp$ consisting of probabilities $p_{\pi}$, $\pi\in \PP_{\{1,2,\ldots,k\}}$,
that sum to $1$, there is a corresponding hyperlattice percolation
model $\cH(\vecp)$: for every hyperedge $e$, we assign $p_{\pi}$
as the probability that the vertices of $e$ are partitioned according
to $\pi$, with the groundset $\{1,2,\ldots,k\}$ of $\pi$ corresponding
to the vertices of $e$ in a manner indicated by the labelling.

Specializing to the 3-uniform case, the dual $e^*$ of an edge
$e$ inherits a labelling from $e$: take the first vertex
to be the one opposite the first vertex of $e$, and so on,
so a \orlab\ 3-uniform plane hyperlattice $\cH$ has a dual $\cHd$
that is again a \orlab\ 3-uniform plane hyperlattice.

If $AcBaCb$ is a triangle with vertices $A$, $B$, $C$ and edges
$a$, $b$ and $c$, then there are five possible partitions
of the vertex set, all of which are \plpa. These are represented by
Wierman and Ziff~\cite{WZ} as $A|B|C$, $AB|C$, $AC|B$, $BC|A$, and $ABC$;
we use the more compact notation $\emptyset$, $AB$, $AC$, $BC$ and $ABC$.
The duals of these partitions are, respectively, $abc$, $ab$, $ac$, $bc$ and $\emptyset$.
In this setting, a probability vector $\vecp$ is simply a vector
$\vecp=(p_\emptyset,p_{AB},p_{AC},p_{BC},p_{ABC})$ of non-negative reals summing to $1$,
and the dual of $\cH(\vecp)$ is simply $\cHd(\vecp^*)$,
where $\vecp^*$ is formed from $\vecp$ by interchanging $p_\emptyset$ and $p_{ABC}$.

The observation of Scullard and Ziff mentioned earlier may be formulated
as follows: if the \orlab\ plane
hyperlattice $\cH$ is self-dual, then the model $\cH(\vecp)$
is self-dual if and only if $p_\emptyset=p_{ABC}$.
The key point is that the other three partitions are all self-dual,
as long as the dual triangle is labelled in the appropriate way,
so only two entries in the probability vector, namely $p_\emptyset$ and $p_{ABC}$,
change when we pass to the dual.
[In the papers~\cite{Ziff_cdc,ZS_exactbond}, the formulation 
of duality is not totally clear. Wierman and Ziff~\cite{WZ}
clearly formulate the notion of self-duality for
un\orlab\ plane hyperlattices, and state that one
can consider any \orlabation{} with lattice structure,
but this seems to be an oversight; one needs the \orlab\
hyperlattice to be self-dual \emph{as a \orlab\ hyperlattice},
which is not always the case.]

If the probabilities in $\vecp$ are appropriate
functions of a single parameter $p$, the condition $p_\emptyset=p_{ABC}$
allows one to determine the critical point
of the model; in general it gives the critical surface.
As remarked earlier, Scullard and Ziff do not discuss whether
self-duality in fact implies criticality; that it does
is shown by Theorem~\ref{th1}.

\begin{corollary}\label{c1}
Let $\cH$ be a \orlab\ 3-uniform plane hyperlattice that
is isomorphic (as a \orlab\ plane hyperlattice) to its dual,
and let $\vecp$ be a probability vector $(p_\emptyset,p_{AB},p_{AC},p_{BC},p_{ABC})$.
Then $\cH(\vecp)$ percolates if $p_{ABC}>p_\emptyset$, and exhibits
exponential decay if $p_{ABC}<p_\emptyset$.
\end{corollary}
\begin{proof}
Let $\vecq$ be the probability
vector $(p,p_{AB},p_{AC},p_{BC},p)$, where $p=(p_\emptyset+p_{ABC})/2$,
and note that $\vecq^*=\vecq$.
Then the dual of the hyperlattice percolation model $\cH(\vecq)$ is 
$\cHd(\vecq^*)=\cHd(\vecq)$, which is isomorphic to $\cH(\vecq)$ by
the assumption on $\cH$.
Thus $\cH(\vecq)$ is a self-dual hyperlattice percolation model,
and so is critical by Theorem~\ref{th1}.
If $p_{ABC}>p_\emptyset$, then $\vecp\pg\vecq$, while if $p_{ABC}<p_\emptyset$,
then $\vecp\pl\vecq$, so the result follows from Theorem~\ref{th1}.
\end{proof}
When $p_{ABC}=p_\emptyset$, the model $\cH(\vecp)$ may or may not percolate.
A (degenerate) example that percolates is given by taking $p_{AB}=1$ and all other
probabilities zero in the triangular hyperlattice shown in Figure~\ref{fig_tri}.
An example that does not is given by taking connection probabilities in the
same hyperlattice corresponding to critical bond percolation on the triangular lattice.
As we shall see in Section~\ref{sec_crit}, in non-degenerate models
there is no percolation at the self-dual point.

As we have seen, Corollary~\ref{c1} follows from Theorem~\ref{th1} simply by restricting
the parametrization of the percolation model, using
the same partition probabilities
for all triangles, rather than allowing different ones for each
orbit under the action of the lattice of symmetries. In other words,
we took all triangles to be of the same \emph{type}. Of course, one
can restrict the model in other ways, considering two or more types of triangle,
or one type of triangle and one type of $4$-gon, etc. Since any results
obtained in this way are simply special cases of Theorem~\ref{th1} we
omit the details; the case of a single type of triangle is of special
importance, since (for self-dual $\cH)$, self-duality reduces
to a single equation, so one obtains the entire critical surface,
rather than a lower dimensional subset of it.

For example, consider the self-dual plane hyperlattice $\cH$ shown in Figure~\ref{fig_square}.
\begin{figure}[htb]
 \centering
 \input{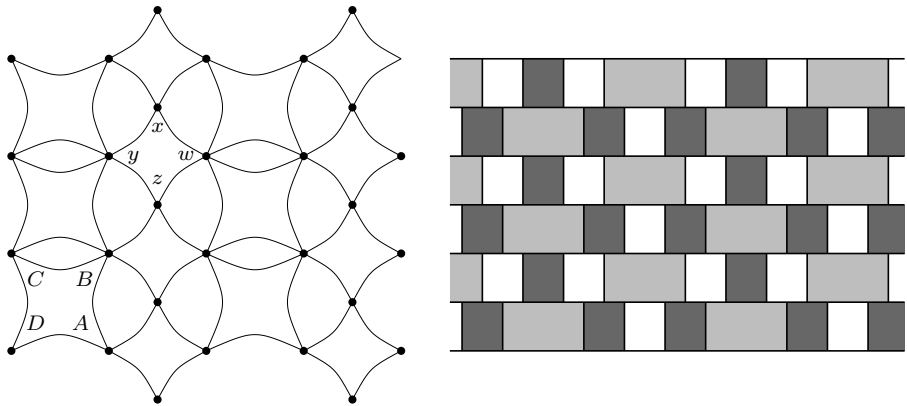}
 \caption{Part of a self-dual 4-uniform hyperlattice $\cH$, shown both as a plane hypergraph
and as the corresponding 3-coloured cubic map.}\label{fig_square}
\end{figure}
Depending on the parameters we choose, $\cH(\vecp)$ can be self-dual
via several different maps $S$ from the plane to itself.
Letting $e_1$ denote the hyperedge $ABCD$ and $e_2$ $xyzw$, there is a translation
mapping $e_1$ into $e_2^*$ and $e_2^*$ into a hyperedge congruent to $e_1$. The model
$\cH(\vecp)$ is self-dual under this translation if and only if
the following equations hold: $p_{1,\emptyset}=p_{2,xyzw}$,
$p_{1,AB}=p_{2,xyz}$, $p_{1,ABC}=p_{2,yz}$, $p_{1,AC}=p_{2,xw|yz}$, $p_{1,AB|CD}=p_{2,xz}$,
and $p_{1,ABCD}=p_{2,\emptyset}$, together with the images of these equations
under rotation. One natural way to satisfy these equations is simply to take
the same probability $p=1/14$ for each of the $14$ \plp s in each square;
taking probability $1/14$ for all partitions other than that into a single
part or into four singletons, and probabilities $p$ and $1/7-p$
for these two partitions, it follows that the percolation threshold in this model
is at $p=1/14$.

Before turning to the proof of Theorem~\ref{th1}, let us make some remarks.
Firstly, an important special class of hyperlattice percolation models
consists of those corresponding to bond or site percolation on (planar) lattices.
By a \emph{$k$-generator} we mean a finite graph $G$ with $k$ distinguished
vertices;
when $k=3$, we denote these vertices by $A$, $B$ and $C$. In a \emph{planar generator} we
insist that $G$ is planar, and that the distinguished vertices lie
in a common face, which we may take to be the outer face.
Suppose that each bond (edge) of $G$ is assigned
a probability. Then, taking the bonds open independently with these probabilities,
for each partition $\pi$ of the distinguished vertices, let $p_\pi$ be the probability
that precisely those vertices in the same part of $\pi$ are connected by open
paths in $G$. Replacing each hyperedge in a \orlab\ $k$-uniform plane hyperlattice $\cH$
by a copy of $G$ one obtains a (planar, if $G$ is planar) lattice $L$.
Taking all bonds open independently with the appropriate
probabilities, the resulting bond percolation model is equivalent to $\cH(\vecp)$ 
in an obvious sense. In this case we say that $\cH(\vecp)$ is \emph{bond realizable}.
We may define \emph{site realizability} analogously; this time, in each generator
we insist that the distinguished vertices are always open.
The definitions extend to general plane hyperlattices using one generator for each
equivalence class of hyperedges.

The key observation of Ziff~\cite{Ziff_cdc} and Ziff and Scullard~\cite{ZS_exactbond}
is that, for any self-dual \orlab\ $3$-uniform plane hyperlattice $\cH$
and any generator $G$,
no matter how complicated, the equation $\vecp=\vecp^\dual$ reduces to a \emph{single}
polynomial equation in the probabilities associated to the bonds or sites of $G$,
so, assuming self-duality implies criticality, this equation gives the entire
critical surface in the inhomogeneous case. Moreover, if we take
the same probability for each bond or site, this polynomial equation gives
the critical probability for bond or site percolation on the lattice generated.
Indeed, this was how Scullard~\cite{Scullard} found the critical point for the `martini
lattice'.

As an example, this method predicts the critical point for bond percolation
on the lattice $L$
shown on the right in Figure~\ref{fig_tri},  as the root of a certain
polynomial of degree~9. Corollary~\ref{c1} proves that this predicted value
is indeed critical. To see the power of the method,
consider the original proof
of Wierman~\cite{WiermanHT} that the critical
probability for bond percolation on the triangular lattice is the
root $2\sin(\pi/18)$ of the equation $p^3-3p+1=0$. This proof relied
on the star-triangle transformation, and an apparent coincidence between
the various connection probabilities associated to a star and to a triangle.
Scullard and Ziff's method gives an argument (which becomes a proof
using the results of Wierman and Ziff~\cite{WZ} or Corollary~\ref{c1})
that does not rely on this -- one simply considers a triangle
and writes down the equation that the probability $p^3+3p^2(1-p)$
that all vertices are connected is equal to the probability $(1-p)^3$
that none are. The fact that one need not consider the dual
lattice is key for examples such as
the lattice shown on the right in Figure~\ref{fig_tri}, which is not
simply related to its dual.

We should emphasize that having observed that a certain percolation model is self-dual,
one is very far from proving criticality. The classical example
is bond percolation on the square lattice with $p=1/2$. This model is obviously
self-dual; proving the conjectured criticality was one of the key open
problems in the early days of percolation theory, finally settled
after 20 years by Kesten~\cite{Kesten1/2}.
Since then, criticality at the self-dual point has been proved
for a number of other models, but these remain the exceptions.
In the hyperlattice context, Wierman and Ziff~\cite{WZ} proved
a result of this type using existing results
on planar lattices; for this reason they consider
only a subclass of bond-realizable models, with planar
generators and certain symmetries. This does
not include examples such as that in Figure~\ref{fig_tri},
which has no symmetries other than translations. However, as outlined
earlier, one can easily adapt their method to such lattices,
as long as the generator is planar. 
Chayes and Lei~\cite{CL} proved a result of this class for
triangular hyperlattices;
they considered only the symmetric case, and imposed an additional condition
on the parameters to ensure positive correlation (see below).
They sketched an argument using standard techniques
for independent percolation on planar lattices that they
claimed extends to this case.

The real significance of Theorem~\ref{th1} is that it applies
to models that are not bond or site realizable. In such cases,
standard results (such as Menshikov's
Theorem~\cite{Menshikov}, for example) do not apply,
and there seems to be no simple way to adapt the arguments
of Chayes and Lei or Wierman and Ziff. Indeed, considerable work will be
needed to prove lemmas corresponding to, for example, Harris's Lemma
and the Russo--Seymour--Welsh Lemma.
Let us note that there \emph{are} hyperlattice percolation models
that are not bond or site realizable. Indeed, considering a single
generator $G$ and the corresponding probabilities $\vecp$,
the event that two given vertices are connected in $G$ is an upset
in terms of the states of the individual bonds or sites.
Thus Harris's Lemma implies that if the vector $\vecp$
is realizable, then for any two upsets $\cU_1$, $\cU_2$ in $\PP$ we have
$\Pr(\cU_1\cap \cU_2)\ge \Pr(\cU_1)\Pr(\cU_2)$.
For example, we must have $p_{ABC}\ge (p_{ABC}+p_{AB})(p_{ABC}+p_{BC})$.
(In the 3-uniform case, Chayes and Lei~\cite{CL} show
that a necessary and sufficient condition for such positive
correlation is that $p_{ABC}p_{\emptyset}\ge p_{AB}(p_{BC}+p_{AC})$,
and the equations obtained from this by permuting $A$, $B$ and $C$, all hold.)
Of course it is trivial to find a probability vector $\vecp$ for which this does
not hold; an example is given in the introduction.

The applicability of Corollary~\ref{c1} is also not limited to models
that are bond or site realizable.
Indeed Scullard and Ziff~\cite{Scullard,Ziff_cdc,ZS_exactbond}
noted that their duality observation does not require the model
to be bond or site realizable.
As in the original paper of Scullard~\cite{Scullard}, one may think of any model $\cH(\vecp)$
as bond percolation on a suitable (planar) lattice, where the states
of the bonds within a generator may be dependent, although those in different
generators must be independent. For example, we may take each 3-generator to be a
triangle, and declare that with probability $p_{i,\emptyset}$
none of the edges are open, with probability $p_{i,ABC}$ all three are,
with probability $p_{i,AB}$ the edge $AB$ is open and the other edges
are closed, and so on. Alternatively, we can take the generator
to be a star.
However, there is no need to think of bonds at all; for percolation,
the only relevant property of the configuration within a triangle
is which of the vertices $A$, $B$ and $C$ the configuration connects
to which others, so it is natural to take this (random) partition
of the vertices as the fundamental object of study. For the mathematical
work (deducing criticality from self-duality), the details of the hyperlattice
turn out to be mostly irrelevant; this is why we consider
general plane hyperlattices in the rest of the paper.

Finally, let us note that extensions of the Scullard--Ziff criterion
to the random cluster model and Potts model have
been described by Chayes and Lei~\cite{CL} and Wu~\cite{Wu};
establishing criticality at the self-dual point remains an open
problem in these cases. Returning to hyperlattice percolation,
in a few very special cases results have been proved that go further
than determining the critical point. For example,
Sedlock and Wierman~\cite{SedlockWierman} established
the equality of the critical exponents between certain
pairs of models, and Chayes and Lei~\cite{CL_cardy} extended
Smirnov's conformal invariance result~\cite{Smirnov}
to what is essentially a very restricted case of the present model.
In this paper we shall not consider such extensions; rather we
shall prove that self-duality does imply criticality,
in the full generality of hyperlattice percolation.

\section{A generalization of Harris's Lemma}\label{sec_Harris}

Given posets $\PP_1,\ldots,\PP_n$, their \emph{product} is the poset $\PP_1\times\cdots\times \PP_n$
where each element $\vecx$ is a list $(x_1,\ldots,x_n)$ with $x_i$ an element
of $\PP_i$, with $\vecx\preccurlyeq\vecy$ if and only if $x_i\preccurlyeq y_i$
for $i=1,\ldots,n$.
If $\Pr$ is a product probability measure on a product
of posets, then with a slight abuse of notation we write $\Pr$
for any of the corresponding marginal measures.

In the later sections of this paper we shall make repeated use of the following generalization
of Harris's Lemma~\cite{Harris} to products of posets. We only need
the case where all $\PP_i$ are equal and finite, but since the proof
gives a little more, we state the result more generally. As usual, a \emph{greatest element}
in a poset $\PP$ means an element $y$ such that $x\preccurlyeq y$ for all $x\in \PP$.
Of course, if a greatest element exists, then it is unique.

\begin{lemma}\label{l_ourHarris}
Let $p>0$.
There is a constant $C=C(p)>0$
such that if $\Pr$ is a product measure on a product
$\PP=\PP_1\times\cdots\times \PP_n$ in which each factor $\PP_i$ is a poset
with a greatest element
whose probability is at least $p$, then for any two upsets
$A$ and $B$ in $\PP$ we have
\[
 \Pr(A\cap B)\ge (\Pr(A)\Pr(B))^C.
\]
\end{lemma}
\begin{proof}
We shall prove the result with $C=\ceil{2/p}$.

As usual, we use induction on $n$. When $n=0$, the
set $\PP$ contains only a single element, and the inequality
is trivial. (It is also not hard to verify directly for $n=1$.)
Suppose then that $n\ge 1$, 
and that the result holds for smaller $n$.
Suppose for notational convenience that $\PP_n$ is finite, and list
its elements as $x_0,x_1,\ldots,x_k$, with $x_0$ the greatest element.
Let $p_i$ denote the probability of element $i$ in $\PP_n$, so 
our assumption is that $p_0\ge p$.

Given a set $S\subset \PP$, let
\[
 S_i=\{y\in \PP_1\times\cdots\times\PP_{n-1}: (y,x_i)\in S\} \subset
 \PP_1\times\cdots\times\PP_{n-1}
\]
denote the $i$th \emph{slice} of $S$ (with respect
to the last factor in the product). Clearly, if $S$ is an upset, then so is $S_i$,
so the induction
hypothesis gives $\Pr((A\cap B)_i)\ge (\Pr(A_i)\Pr(B_i))^C$ for each $i$.
Also, the upset conditions give $A_i\subset A_0$ and $B_i\subset B_0$ for $i>0$.
Since $\Pr$ is a product measure, we have
$\Pr(S)=\sum_i p_i\Pr(S_i)$ for any $S$.
Using these observations, it suffices to show that
\begin{equation}\label{abineq}
 \sum_{i=0}^k p_i (a_ib_i)^C \ge (ab)^C
\end{equation}
holds whenever the non-negative real numbers $p_i$, $a_i$ and $b_i$
satisfy the following conditions:
$p_0,\ldots,p_k$ sum to 1, $p_0\ge p$, 
$a_0=\max_i a_i$, $b_0=\max_i b_i$,
$a=\sum p_ia_i$, and $b=\sum p_ib_i$.

In proving \eqref{abineq} we may assume that $a$, $b>0$. Dividing the $a_i$ through
by $a$ and the $b_i$ by $b$, we may assume that $a=b=1$.
Let $\alpha_i=a_i-1$ and $\beta_i=b_i-1$. Since $\sum_i p_i\alpha_i=0$,
we have $\alpha_0=\max_i \alpha_i\ge 0$.
Also,
\begin{equation}\label{ems}
 \sum_{i\,:\,\alpha_i<0} -p_i\alpha_i   = \sum_{i\,:\,\alpha_i>0} p_i\alpha_i
  \le \sum_i p_i\alpha_0 = \alpha_0,
\end{equation}
and similarly for the $\beta_i$.
Our aim is to prove that $\sum_i p_i(1+\alpha_i)^C(1+\beta_i)^C \ge 1$.
Recalling that $\sum_i p_i(\alpha_i+\beta_i)=0$, this is equivalent
to showing that
\begin{equation}\label{sum}
 \Delta = \sum_i p_i\bigl( (1+\alpha_i)^C(1+\beta_i)^C - 1-C(\alpha_i+\beta_i)\bigr) \ge 0.
\end{equation}
Since $\alpha_0$, $\beta_0\ge 0$, we have
$(1+\alpha_0)^C(1+\beta_0)^C\ge 1+C\alpha_0+C\beta_0+C^2\alpha_0\beta_0$, so 
the contribution to the sum $\Delta$ from the $i=0$ term
is at least $p_0C^2\alpha_0\beta_0$.

Turning to the remaining terms, since $(1+x)^n\ge 1+nx$ if $x\ge -1$ and
$n$ is a positive integer, we have
\[
 (1+\alpha_i)^C(1+\beta_i)^C = (1+\alpha_i+\beta_i+\alpha_i\beta_i)^C
 \ge 1+C(\alpha_i+\beta_i+\alpha_i\beta_i).
\]
If $\alpha_i$ and $\beta_i$ have the same sign, then the contribution
of the $i$th summand to \eqref{sum}
is nonnegative.
If $\alpha_i<0$ and $\beta_i>0$, then the negative of the contribution
of the $i$th summand to \eqref{sum}
is at most $p_iC|\alpha_i|\beta_i \le p_iC|\alpha_i|\beta_0$.
By \eqref{ems}, the negative of the sum of the contribution of all such terms
is at most
\[
 C\beta_0\sum_{i\,:\,\alpha_i<0} p_i(-\alpha_i) \le C\beta_0\alpha_0.
\]
The same bound holds for terms with $\alpha_i>0$ and $\beta_i<0$, so we conclude
that the sum in \eqref{sum} satisfies
\[
 \Delta\ge p_0C^2\alpha_0\beta_0-2C\alpha_0\beta_0 = (p_0C-2)C\alpha_0\beta_0.
\]
Since $p_0C\ge 2$ by our choice of $C$, this establishes the inequality, and hence the lemma.
\end{proof}
Note that we have not attempted to optimize the value of $C$ above. Indeed, for $p_0$ small,
the proof above goes through with $C$ only slightly larger
than $1/p_0$, noting that the contribution
to $\Delta$ from $i=0$ is at least $C(C-1)(\alpha_0^2+\beta_0^2)/2+C^2\alpha_0\beta_0
\ge (C(C-1)+C^2)\alpha_0\beta_0$.

Of course, Harris's Lemma itself does not apply in this setting, i.e.,
one cannot simply take $C=1$.
Indeed, considering the upsets $\{x_0,x_1\}$ and $\{x_0,x_2\}$ in the poset
on $\{x_0,x_1,x_2\}$ in which $x_0$ is greatest and $x_1$ and $x_2$ are incomparable,
with $\Pr(x_0)=p_0$ and $\Pr(x_1)=\Pr(x_2)=(1-p_0)/2$,
we may have $\Pr(A)=\Pr(B)=(1+p_0)/2$ and $\Pr(A\cap B)=p_0$.
For $p_0$ small, this shows that we need the exponent $C$ to be at least a constant times
$\log(1/p_0)$.

Since the form of the bound obtained will be irrelevant in our remaining
arguments, let us state
as a corollary a weaker, more abstract version of the result.

\begin{lemma}\label{ourHarris2}
Let $\PP$ be a finite poset with a greatest element $x_0$ and
let $\Pr$ be a probability measure on $\PP$ with $\Pr(x_0)>0$.
There is a function $F=F_{\PP,\Pr}$ from $(0,1]^2$ to $(0,1]$
that is strictly increasing in each argument
such that, for any $n\ge 1$ and any upsets $A$ and $B$ in $\PP^n$
with $\Pr(A)$, $\Pr(B)>0$, we have
\[
 \Pr(A\cap B)\ge F(\Pr(A),\Pr(B)).
\]
\end{lemma}
\begin{proof}
Immediate from Lemma~\ref{l_ourHarris}.
\end{proof}
As usual, the extension to infinite products is immediate by approximating
with the finite case.

Note that while the form of the function $F$ is irrelevant, it is natural
to look for an $F$ of the form $F(a,b)=(ab)^C$. Indeed, given upsets
$A_i$ and $B_i$ in $\PP^{n_i}$ then, considering the product upsets $A_1\times A_2$
and $B_1\times B_2$ in $\PP^{n_1+n_2}$, one sees that the optimal $F$ satisfies
$F(a_1a_2,b_1b_2)\le F(a_1,b_1)F(a_2,b_2)$. Of course the optimal
$F$ cannot be precisely $F(a,b)=(ab)^C$, since we certainly need $C>1$,
and then the bound is not tight if $a=1$, for example.

\subsection{High probability unions of upsets}

In many applications of Harris's Lemma
in percolation, the exact form of the bound is not important,
so the weaker conclusion of the more generally applicable
Lemma~\ref{ourHarris2} may be used instead of Harris's bound. We give
one example that we shall use later: a form of the `square-root' trick,
showing that if the union of a fixed number of upsets has high
enough probability, then one of the upsets has high probability.

\begin{corollary}\label{csr}
Let $\PP$ be a finite poset with a least element $x_0$ and
let $\Pr$ be a probability measure on $\PP$ with $\Pr(x_0)>0$.
Given $\eps>0$ and a positive integer $k$ there
is a $\delta=\delta(\PP,\Pr,k,\eps)>0$ such that, for any $n\ge 1$,
if $A_1,\ldots,A_k$ are upsets in $\PP^n$
with $\Pr(\bigcup A_i)\ge 1-\delta$, then
$\Pr(A_i)\ge 1-\eps$ for some $i$.
\end{corollary}
Note that $x_0$ is a \emph{least} element here, not a greatest one.
\begin{proof}
Set $\delta_1=\eps$. For $j\ge 2$ let $\delta_j=F(\delta_{j-1},\eps)$,
where $F$ is the function given by Lemma~\ref{ourHarris2}
applied to the reverse of $\PP$,  and set $\delta=\delta_k/2$.

If $\Pr(A_i)<1-\eps$ for each $i$, then the downsets $A_i^\cc$
each have probability at least $\eps$. Viewing these downsets
as upsets in the reversed poset, it follows by Lemma~\ref{ourHarris2}
and induction
on $j$ that $\Pr(A_1^\cc\cap \cdots\cap A_j^\cc)\ge \delta_j$.
Thus $\Pr(\bigcup_{i=1}^k A_i)\le 1-\delta_k<1-\delta$, a contradiction.
\end{proof}

We shall also need a related result,
stating that if we have a union of upsets which is extremely likely
to hold, then it is very likely that \emph{many} of the individual upsets hold,
as long as we rule out the trivial case that the union is extremely
likely because one of the individual upsets is itself extremely likely.

\begin{lemma}\label{manyhold}
Let $\PP$ be a poset with a least element $x_0$, and let $\Pr$
be a probability measure on $\PP$ with $\Pr(x_0)>0$.
Given an integer $N>0$ and a real number $\eps>0$, there exists
a $\delta=\delta(\Pr(x_0),N,\eps)>0$ such that, for any $n$ and any collection $A_1,\ldots,A_m$
of upsets in $\PP^n$ with $\Pr(A_i)\le 1-\eps$ for all $i$
and $\Pr(\bigcup A_i)\ge 1-\delta$, the probability
that at least $N$ of the events $A_i$ hold is at least $1-\eps$.
\end{lemma}

\begin{proof}
Let $F$ be the function appearing in Lemma~\ref{ourHarris2} applied
to the reverse of the poset $\PP$.
Set $\eps'=F(\eps/N,\eps)$ and $\delta_1=\eps/N$.
Inductively define $\delta_k$ by $\delta_k=F(\eps',\delta_{k-1})$
for $k\ge 2$.

We claim that, for any $k\ge 1$, if $A_1,\ldots,A_m$ is any collection
of upsets in any power $\PP^n$ of $\PP$ with  $\Pr(\bigcup A_i)\ge 1-\delta_k$,
then we can  find disjoint index sets $I_1,I_2,\ldots,I_k$
such that for each $1\le j\le k$ we have $\Pr(\bigcup_{i\in I_j} A_i)\ge 1-\eps/N$.
The result then follows by setting $\delta=\delta_N$ and $k=N$: the claim
tells us that
with probability at least $1-N\eps/N=1-\eps$, for every $j$
at least one of the events $\{A_i:i\in I_j\}$ holds, so at least $N$
of the $A_i$ hold.

For $k=1$ the claim is trivial, taking $I_1=\{1,2,\ldots,m\}$.

Suppose then that $k\ge 2$ and that the claim holds when we replace
$k$ by $k-1$. Let $f_i$ be the probability
that none of $A_1,\ldots,A_i$ holds. For any $k$ we have $\delta_k\le \delta_1=\eps/N$,
so $f_m\le \delta_k\le \eps/N$, and
\[
 i_1 = \min\{i: f_i\le  \eps/N\}
\]
is defined. Setting $I_1=\{1,2,\ldots,i_1\}$, note that the
event $\bigcup_{j\in I_1}A_i$ has probability at least $1-\eps/N$.

Consider the downsets $\bigcap_{1\le i\le i_1-1}A_i^\cc$ and $A_{i_1}^\cc$.
Applying Lemma~\ref{ourHarris2} to these events, seen as upsets in the
reversed poset, we have $f_{i_1}\ge F(f_{i_1-1},1-\Pr(A_{i_1}))$.
Since $F$ is increasing, using the definition of $i_1$
and our assumption on $\Pr(A_i)$, it follows that $f_{i_1}\ge F(\eps/N,\eps)=\eps'$.
Let $\cD_1=\bigcap_{1\le i\le i_1} A_i^\cc$ and $\cD_2=\bigcap_{i_1+1\le i\le m} A_i^\cc$.
Applying Lemma~\ref{ourHarris2} to $\cD_1$ and $\cD_2$, we have
\[
 \delta_k \ge \Pr(\cD_1\cap \cD_2) \ge F(\Pr(\cD_1),\Pr(\cD_2)) \ge F(f_{i_1},\Pr(\cD_2)).
\]
Since $F$ is strictly increasing, $f_{i_1}\ge \eps'$, and $\delta_k=F(\eps',\delta_{k-1})$,
it follows that
\[
 F(\eps',\delta_{k-1}) = \delta_k \ge F(\eps',\Pr(\cD_2)),
\]
so $\Pr(\cD_2)\le \delta_{k-1}$, and the union of $A_{i_1+1},\ldots,A_m$
has probability at least $1-\delta_{k-1}$. Applying the induction hypotheses
to this set of events gives us $I_2,\ldots,I_m$ with the required properties, completing
the proof.
\end{proof}

\section{Colourings, hypergraphs and crossings}\label{sec_cc}

Our next aim is to prove a form of rectangle-crossing
lemma loosely analogous to the Russo--Seymour--Welsh Lemma~\cite{Russo,SW}, but
applicable in the hyperlattice percolation context.
Naturally, this involves considering `open crossings
of rectangles'. 
As in~\cite{Voronoi,ourKesten2}, for example (see also~\cite{BRbook}), to make
this precise and clean we shall work instead with `black crossings' in a suitable
black/white colouring of the faces of a cubic map (i.e., 3-regular plane graph).
We assume throughout that our maps are \emph{well-behaved}, meaning that
the edges are drawn as piecewise-linear curves, every face is bounded, and any
bounded subset of the plane
contains only finitely many vertices and meets only finitely many edges.

Recall that a plane hyperlattice $\cH$ may be thought of as a 3-coloured cubic map, where
the faces are properly coloured black, white and grey, with a lattice
$\cL$ of translational symmetries; as usual we view
$\cL$ as a subset of $\RR^2$. From now on this is our
default viewpoint when considering any plane hyperlattice.
A \emph{colouring} $\cC$ of $\cH$ is a 2-coloured cubic map obtained as follows:
first subdivide each grey face of $\cH$ into one or more \emph{subfaces}, in such a way
that the resulting map is still cubic. Then colour each subface black or white,
as in Figure~\ref{fig_shading3}, for example.
\begin{figure}[htb]
 \centering
 \input{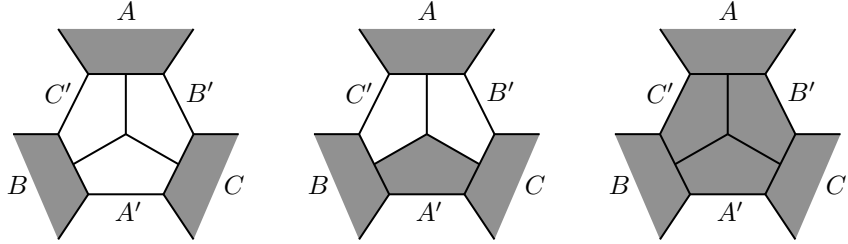}
 \caption{Colourings of a grey face (the central hexagon) corresponding to three of the five
possible partitions of the vertices (black faces) $A$, $B$, $C$ of the corresponding hyperedge,
namely the partitions $A|B|C$, $A|BC$ and $ABC$.
The outer black faces corresponding to two vertices are connected by a black
path inside the hexagon if and only
if the vertices are in the same part of the partition; the outer white
faces and the white connections between them correspond to the dual partition.}\label{fig_shading3}
\end{figure}
The resulting colouring $\cC$ is an (improper, of course) black/white colouring
of the faces of a (well-behaved) cubic map. It will be convenient to declare that points
in the boundary of a face have
the colour of that face, so some points are both black \emph{and} white.
A \emph{black path} in $\cC$ is then simply a (piecewise-linear) path in the plane every point of which is black;
such a path corresponds to a sequence of black faces in which consecutive faces share a point
and thus (since the map is cubic) an edge. White paths are defined similarly.

Recall that a grey face $F_e$ of $\cH$ corresponds to a hyperedge $e$.
Also, since $\cH$
is properly 3-coloured, $F_e$ is surrounded alternately by black and white faces,
corresponding to vertices and dual vertices.
In $\cC$, certain pairs of vertices incident with $e$ are connected by black
paths within $F_e$; this generates a partition $\pi$ of the vertices of $e$,
which is easily seen to be non-crossing. Thus $\cC$ corresponds to a configuration $\omega$ on $\cH$,
with open paths in $\omega$ corresponding to black paths in $\cC$ and vice versa. (Of course,
each open path is represented by many `nearby' black paths).
Crucially, \emph{white} paths within $F_{e^*}=F_e$ induce the dual partition $\pi^*$ of the dual vertices
incident with $e^*$, so the \emph{negative} of $\cC$, obtained by interchanging
black and white throughout, corresponds to the dual configuration $\omega^*$ on $\cH^*$.

By a \emph{black cluster} in $\cC$ we mean a maximal connected black subset of the plane.
A \emph{white cycle} is a white path that starts and ends at the same point.
Since $\cC$ is a black/white colouring of a cubic planar map, it is easy to see
that a black cluster
is surrounded by a white cycle if and only if it is finite.
This is a precise form of the duality property relating
$\omega$ and $\omega^*$ mentioned in Section~\ref{sec_mod}.

\begin{remark}\label{single}
It will be convenient later to assume that all partitions corresponding
to $e$ are realized by colourings of a single subdivision of $F_e$ into subfaces,
as in Figure~\ref{fig_shading3}.
This can be achieved for all \plp s of
hyperedges with any number of vertices, as illustrated
in Figure~\ref{fig_shading6}.
\end{remark}
\begin{figure}[htb]
 \centering
 \input{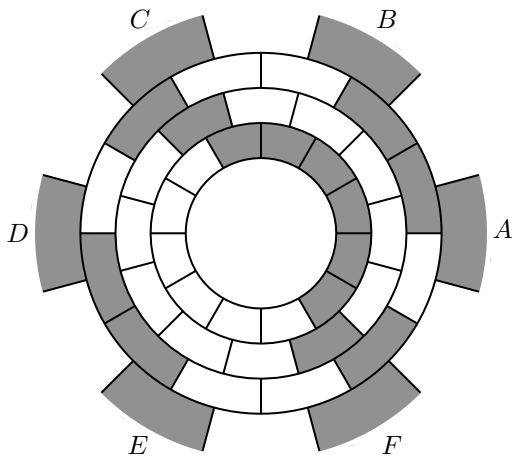}
 \caption{A subdivision of the $2n$-gon corresponding to a hyperedge $e$ with $n=6$ vertices,
and a colouring corresponding to the partition $AB|CF|DE$. In general, it suffices to take
$2k-1$ rings of subfaces of the type shown to realize any \plp, where $k\le n/2$
is the maximum `nesting depth' of a \plp\ of $n$ objects.}\label{fig_shading6}
\end{figure}

Note that there are many possible colourings $\cC$ corresponding to a given
configuration $\omega$ on $\cH$: even if we fix the division of each
grey face into subfaces (as we shall), there may be many colourings of the subfaces
giving the same partition of the vertices. For much of the rest of the paper,
we shall take the random colouring $\cC$ as the fundamental object of study,
rather than the random configuration $\omega$.

More formally, an \emph{independent lattice colouring} $\cC$, or
simply a \emph{colouring}, is a colouring obtained from a plane hyperlattice $(\cH,\cL)$
by subdividing each grey face in a deterministic manner,
and then colouring the resulting subfaces randomly black and white, in such a way that
the colourings inside different grey faces are independent, and 
translations through elements of the lattice $\cL$ preserve the distribution
of $\cC$.
From the remarks above, any plane lattice percolation model $\cH(\vecp)$
can be realized by an independent lattice colouring $\cC$ associated to $\cH$.

We assume throughout that our colourings $\cC$ are \emph{non-degenerate}, meaning
that within any grey face, the all-black and all-white colourings have positive
probability. Later on we shall have to impose some additional conditions for
faces corresponding to hyperedges with more than three vertices.
 
Note that the state space $\Omega$ underlying the random colouring $\cC$ may be viewed
as a product of one poset $\PP_F$ for each grey face $F$ of $\cH$:
in the partial order, we have $c_1\preccurlyeq c_2$ if every subface that is black
in $c_1$ is black in $c_2$.
Picking a finite set $F_1,\ldots,F_n$ of faces representing the orbits of $E(\cH)$ (the set
of grey faces) under the action of $\cL$, from lattice invariance
we may regard $\Omega$ as a power of the poset $\PP=\PP_{F_1}\times\cdots\times \PP_{F_n}$.
From independence, the probability measure associated
to $\cC$ is then a product probability measure on $\Omega$.
The non-degeneracy condition implies that $\PP$ has a greatest
element (all subfaces of each $F_i$ black) and a least element (all white),
and that each has positive probability.

The event that a given path is black, or that a black path exists
with certain properties,
is an upset in $\Omega$, in the sense of Section~\ref{sec_Harris}.
Thus Lemma~\ref{ourHarris2} applies to two such events.
Similarly, considering the reverse poset, Lemma~\ref{ourHarris2}
applies to two events each defined by the existence of a white path with certain properties.
This is the reason for the non-degeneracy assumption.

Note that we are always considering two coloured maps: $\cH$ (which is deterministic) and $\cC$.
To avoid ambiguity, we say that a point or face is \emph{$\cH$-black} if it is black
in $\cH$, and \emph{$\cC$-black} if it is black in $\cC$, and similarly for other colours.
By default, \emph{black} or \emph{white} refers to $\cC$, while \emph{grey} necessarily refers to $\cH$.

The lattice structure ensures that the faces of the hyperlattice $\cH$ (seen as a map, as usual)
cannot be too wild.
\begin{lemma}\label{reg}
Let $(\cH,\cL)$ be a plane hyperlattice, viewed as a cubic map.
There are finitely many faces $F_1,\ldots,F_N$ of $\cH$ 
such that for any face $F$, there is an element of $\cL$ such that
the corresponding translation of $\RR^2$
induces an isomorphism of $\cH$ mapping $F$ to one of $F_1,\ldots,F_N$.
Furthermore, there is a constant $d_0$ such that (i) every face has diameter at most 
$d_0$ and (ii) every point of $\RR^2$ is within distance $d_0$
of an element of $\cL$, 
and for each $r>0$ there is a constant $N_r$ such that any disk of radius
$r$ meets at most $N_r$ faces.
\end{lemma}
\begin{proof}
Let $D$ be a fundamental domain of $\cL$.
Since $D$ is bounded, by the definition of a plane hypergraph
$D$ contains finitely many vertices of $\cH$,
and meets finitely many hyperedges.
It follows that $D$ meets only finitely many faces of $\cH$, viewed
now as a cubic map. Hence there is a finite
set $F_1,\ldots,F_N$ of faces, all meeting $D$, containing
one representative of each orbit of the action of $\cL$ on
the faces of $\cH$.
The remaining statements follow
easily, taking $d_0$ to be the larger of $\max_i\diam(F_i)$ and $\diam(D)$.
\end{proof}

The parameter $d_0=d_0(\cH)$ appearing in Lemma~\ref{reg}
will be used throughout this and the next section.
For example, we say that a rectangle is \emph{large} if all its sides have length
at least $100d_0$. In what follows,
to avoid trivialities such as a rectangle having a black crossing
with probability $1$, we only ever consider large rectangles.

\begin{remark}
Let us remark briefly on the numerical constants appearing in this paper.
In many places, rather than argue that some constants exist with certain
properties, we simply give numerical values that work, such as $100$ (here)
or the less natural constants $0.1$, $1.1$, $8$, etc appearing later. Of course
the precise values are not important.
\end{remark}

Given an angle $\theta$, by a \emph{$\theta$-aligned rectangle}
we mean a rectangle $R\subset \RR^2$ such that one pair of sides
makes an angle $\theta$ to the $x$-axis, measured in the positive sense
from the $x$-axis. We refer to these sides as \emph{horizontal}
and the other sides as \emph{vertical}. Thus, after rotating $R$ clockwise through an angle
$\theta$, the horizontal sides become horizontal in the usual sense.
Whenever we speak of a rectangle $R$, we have an angle $\theta$ in mind and assume
that $R$ is $\theta$-aligned. Note that the same geometric rectangle is $\theta$-aligned
for two values of $\theta$ differing by $\pi/2$.

By the \emph{\ww} and \emph{\hh} of a rectangle, we mean the length of the horizontal
and vertical sides, respectively; which is which depends on whether
we view $R$ as $\theta$-aligned or $(\theta+\pi/2)$-aligned.

We always assume that our rectangles $R$ 
are in \emph{general position} with respect to our colouring $\cC$,
meaning that each vertex of $R$ lies in the interior of a face of the colouring,
no vertices of the colouring are on the boundary of $R$, and the edges of the colouring
can only cross the edges of $R$ transversely.

Given a rectangle $R$, by a \emph{black horizontal crossing} of $R$
we mean a $\cC$-black path within $R$ starting at some point on one vertical
side of $R$ and ending at some point on the other vertical side.
White vertical crossings are defined similarly, and so on. We write $\Hb(R)$ for the event
that $R$ has a black horizontal crossing (in the random colouring $\cC$), and
$\Vw(R)$ for the event that it has a white vertical crossing, and so on.
For a proof of the following `obvious' lemma concerning (well-behaved) 2-coloured maps
see~\cite[Ch. 8, Lemma 12]{BRbook}.

\begin{lemma}\label{oneway}
Given any well-behaved colouring $\cC$ 
of the plane and any rectangle $R$ in general position
with respect to $\cC$,
exactly one of the events $\Hb(R)$ and $\Vw(R)$ holds.\noproof
\end{lemma}

This lemma, together with self-duality, will be the starting point for our
Russo--Seymour--Welsh-type argument. This argument will be rather involved.
There are various technical complications arising from
the generality of plane hyperlattices; we deal with most of these
in the rest of this section. In the next section we turn
to the core of the argument, where the complications
are mostly due to the lack of symmetry.

For the rest of the section we consider a given non-degenerate independent
lattice colouring $\cC$, associated to a plane hyperlattice $\cH$.

\subsection{How crossing probabilities vary}

Let $h(R)=\Pr(\Hb(R))$ be the probability that $R$ has a black horizontal
crossing, and let $v(R)=\Pr(\Vb(R))$. Note that if we switch from
viewing a given geometric rectangle $R$ as $\theta$-aligned to
viewing it as $(\theta+\pi/2)$-aligned, then $h(R)$ and $v(R)$ swap.
Our next aim is to show that $h(R)$ and $v(R)$ do not change too much
if we move the edges of $R$ slightly.
This is not very surprising, but giving full
details in the present generality requires a little work.
We start with a technical lemma.

We say that a path $P$ in the plane is \emph{potentially black} with respect
to a hyperlattice $\cH$ if no point of $P$ is $\cH$-white,
so $P$ corresponds to a sequence of black and grey faces of $\cH$.
In other words, $P$ is potentially black if and only if there is a positive
probability that $P$ is actually black in the random colouring $\cC$. We
write $\dH$ for the Hausdorff distance between subsets of $\RR^2$.
Let $d_0=d_0(\cH)$ be the constant given by Lemma~\ref{reg},
so every face of $\cH$ has diameter at most $d_0$.

\begin{lemma}\label{ba}
Let $\cH$ be a hyperlattice
and $P$ a piecewise-linear path. Then there is a potentially black path $P'$
with $\dH(P,P')\le 2d_0$.
\end{lemma}

\begin{proof}
Recall that in the 3-coloured map $\cH$, no two white faces
are adjacent.
Whenever $P$ passes through a white face, simply take a detour around (or
just outside) this face. Similarly, if $P$ starts or ends in a white face, modify
$P$ to start/end just outside this face.
\end{proof}

\begin{lemma}\label{cont1}
Let $\cC$ be a non-degenerate independent lattice colouring. Then there exists a constant $c>0$,
depending only on (the distribution of) $\cC$,
such that, for any large rectangle $R$, if $R_+$ is
a rectangle formed by moving one vertical side of $R$ outwards by a distance of at most $1$,
then $h(R)\ge h(R_+)\ge c\, h(R)$.
\end{lemma}
\begin{figure}[htb]
 \centering
 \input{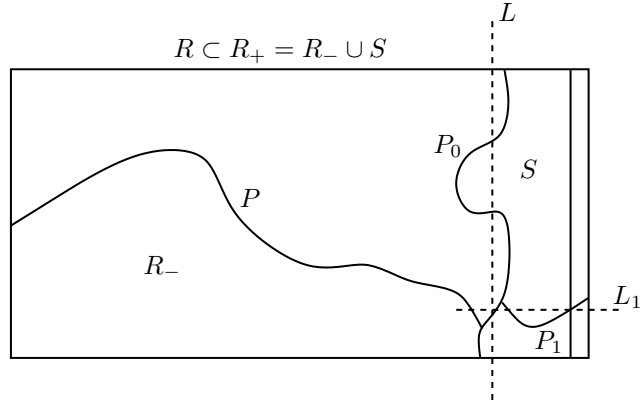}
 \caption{A rectangle $R_+$ slightly extending a rectangle $R$, divided
into $R_-$ and $S$ by a potentially black path $P_0$ that crosses
$R$ from top to bottom near its right-hand side. The path $P$ starts on the left-hand
side of $R$ and ends on $P_0$. Finally, $P_1$ starts on $P_0$ and ends on the right of $R_+$.
All paths lie inside $R_+$.}\label{fig_cont1}
\end{figure}
\begin{proof}
Note first that any black horizontal crossing of $R_+$ contains a black horizontal
crossing of $R$, so $h(R)\ge h(R_+)$. Also, we may assume without loss of generality
that we move a side of $R$ outwards by distance \emph{exactly} 1 to obtain $R_+$.

In the arguments that follow, various constants appear that depend on $\cC$. However,
they will depend only on (a) the quantity $d_0=d_0(\cH)$, where
$\cH$ is the hyperlattice underlying $\cC$, and (b) the minimum probability of the all-black state
in a grey face. These are invariant under rotation and translation, so,
rotating and translating $R$ \emph{and $\cC$}, without loss of generality
we may assume that $\theta=0$, so horizontal means horizontal in the usual sense,
and that
$R=[0,a]\times [0,b]$ and $R_+=[0,a+1]\times [0,b]$, even though $h(R)$ itself
varies as $R$ is rotated and/or translated with $\cC$ fixed.

We may assume that $a,b\ge 100d_0$. Applying Lemma~\ref{ba} to the line segment
$L$ from $(a-3d_0,-3d_0)$
to $(a-3d_0,b+3d_0)$, and truncating the resulting path $P'$ when it last hits
the bottom side of $R$ and first hits the top side,
we find a potentially black path $P_0$ crossing $R$ from top to bottom, where all
points have $x$-coordinate between $a-5d_0$ and $a-d_0$. Let $S$ be the `strip' consisting
of those points of $R_+$ to the right of $P_0$, and let $R_-$ denote the rest
of $R_+$, so $R_-\subset R$ is a `distorted rectangle' whose right-hand side is potentially black.
Let $E_0$ be the event that $R_-$ has a horizontal black crossing,
noting that if $\Hb(R)$ holds then so does $E_0$, so $\Pr(E_0)\ge h(R)$.

Let $\cF_-$ be the set of grey faces meeting $R_-$ but not $S$, and let $\cF_S$
be the set of grey faces meeting $S$, so $E_0$ depends on the colourings
of the faces in $\cF_-\cup \cF_S$. Let $E_0'$ be the event that $E_0$ would
hold after recolouring (in $\cC$) all faces in $\cF_S$ to black; thus $E_0'$ is the event
that there is a path $P$ crossing $R_-$ from left to right, every point of which
is either $\cC$-black or in a face in $\cF_S$.
If $E_0$ holds then so does $E_0'$, so $\Pr(E_0')\ge\Pr(E_0)\ge h(R)$.

Now $E_0'$ depends only on the states of the faces in $\cF_-$. Let us condition
on these states, assuming that $E_0'$ holds. 
Our aim is to show that the conditional probability that $\Hb(R_+)$ holds is not too small.
To do this we shall first modify $P$ in a certain way (if needed), and then extend $P$,
obtaining a path $P_+$ crossing $R_+$ from left to right in which every point is either
$\cC$-black or in a face in $\cF_S$, in such a way that the set of faces
in $\cF_S$ that $P_+$ meets has size $O(1)$.
Then we reveal the (as yet unexamined) states of these faces.
Since there are $O(1)$ of them, with probability bounded
away from zero they are all entirely black in $\cC$.

We start with the modification of $P$. Let us call a face
\emph{bad} if it meets both $R_-$ and $S$, but does not lie entirely in $R_+=R_-\cup S$,
and \emph{good} otherwise.
Since all faces have diameter at most $d_0$, any bad face must contain a point
within distance $10d_0$ of either $(a,0)$ or $(a,b)$. In particular, there
are $O(1)$ bad faces.
Suppose $P$ meets one or more good faces in $\cF_S$. Tracing $P$ from the left, stop the
first time it meets such a face $F$. Since $F$ is connected and meets $S$,
we can continue within this face to a point of $S$. Since $F$ is contained
in $R_+$, in doing so we do not go outside $R_+$, but we do
leave $R_-$, so we cross the right-hand side of $R_-$. Stop when this happens.

After this modification, $P$ has the properties above (all points $\cC$-black
or in faces in $\cF_S$), and it meets at most one good face in $\cF_S$, and
thus (since there are $O(1)$ bad faces in total) 
$O(1)$ faces in $\cF_S$.

Let $(x,y)$ be the right-hand end of $P$, so $(x,y)$ lies on $P_0$.
Pick $y'\in [3d_0,b-3d_0]$ with $|y-y'|\le 3d_0$.
Let $L_1$ be the line-segment from $(a-7d_0,y')$ to $(a+1+2d_0,y')$.
Apply Lemma~\ref{ba} to $L_1$ to obtain a potentially black path $P_1'$.
Then $P_1'$ starts inside $R_-$, ends outside $R_+$, and cannot cross
the lines $y=0$ and $y=b$, so it contains a sub-path $P_1$
within $S$ crossing $S$ from left to right. The left-hand end of $P_1$
is within distance $10d_0$ of $(x,y)$.
To construct our final path $P_+$, trace $P$ from left to right, run along $P_0$
from $(x,y)$ to the first end of $P_1$, and then trace $P_1$.
This path crosses $R_+$ from left to right. Furthermore, any point
of $P_+\setminus P$ is within distance $20d_0$ of $(x,y)$, so $P_+\setminus P$
meets $O(1)$ grey faces. Now $P_+$ is potentially black. Any point
of $P_+$ in a grey face in $\cF_-$ is necessarily a point of $P$,
and so is $\cC$-black by the properties of $P$.
Finally, $P_+$ meets
$O(1)$ grey faces in $\cF_S$. With (conditional) probability bounded away
from 0 the latter faces are all entirely $\cC$-black, and then $P_+$ is a black path,
so $\Hb(R_+)$ holds.
\end{proof}

When $h(R)$ is very close to 1, Lemma~\ref{cont1} is not very informative; it does
not rule out $h(R)$ dropping from $1$ to $1/100$, say, as $R$ is extended a tiny bit
horizontally. In this case the probability of a white vertical crossing
would increase from $0$ to $99/100$.
Note that this vertical crossing probability \emph{can} change by a large ratio:
if the white colouring is subcritical,
with $a$ constant and $b$ large, the white vertical crossing probability
is approximately of the form $e^{-c_a b}$, where $c_a$ is a positive constant
depending on $a$. Increasing $a$ by $1$ decreases $c_a$ to a new constant
value, which can change $e^{-c_a b}$ by an arbitrarily large ratio.

However, it is still true that the white vertical crossing probability cannot jump
from very small to fairly large. We phrase the result in terms of black
horizontal crossings as above.

\begin{lemma}\label{cont2}
Let $\cC$ be a non-degenerate independent lattice colouring. Given $\eps>0$ there is a $\delta>0$
such that for any large rectangle $R$
and any rectangle $R_+$ formed by moving one vertical side of $R$ outwards
by a distance of at most $1$, if $h(R)\ge 1-\delta$ then $h(R_+)\ge 1-\eps$.
\end{lemma}

\begin{proof}
The proof is an extension of that of Lemma~\ref{cont1}; we define the path $P_0$
splitting $R_+$ into $R_-$ and $S$ as before, and let $\cF_-$ be the set
of grey faces meeting $R_-$ but not $S$, and $\cF_S$ the set of grey faces
meeting $S$. As before, let $E_0'$ be the event that $R_-$ has a horizontal
crossing every point of which is $\cC$-black or in a grey face in $\cF_S$,
recalling that $\Pr(E_0')\ge h(R)$.

Let $f_1,\ldots,f_N$ list all grey or $\cH$-black faces of $\cH$ meeting $S$.
Let $E_i$ be the event that $R_-$ contains a $\cC$-black
path $P$ starting on the left-hand
side of $R_-$, ending at a boundary point of $f_i$, and meeting no other $f_j$,
nor the interior of $f_i$.
Note that $E_i$ depends only on the states of grey faces in $\cF_-$, not those in $\cF_S$.
If $E_0'$ holds then, truncating a path $P$ witnessing this event
the first time $P$ meets
any $f_i$, we see that one of the $E_i$ must hold.
Hence,
\begin{equation}\label{eibig}
 \Pr(\bigcup E_i) \ge \Pr(E_0') \ge h(R).
\end{equation}
We claim that any given face $f$ of $\cH$ is surrounded by a
`ring' of white and grey faces of $\cH$ with bounded size
such that if all grey faces in the ring happen to be coloured white in $\cC$,
then no $\cC$-black path starting outside the ring can end at a point of $f$.
Here the bound depends only on $\cH$,
not the face chosen. Indeed, if $f$ is $\cH$-black, we simply take
the faces neighbouring $f$ to form the ring. If $f$ is white or grey, we may
simply take all white or grey faces of $\cH$ within distance $2d_0$ of $f$ as our ring:
since no black face of $\cH$ touches any other black face, any $\cC$-black path
to $f$ from distance more than $2d_0$ must meet a grey face in our ring.

If every grey face in the ring just described about $f_i$ happens
to be coloured entirely white in $\cC$, then no black path ends at any point of $f_i$,
and $E_i$ does not hold. It follows that for some constant $\eps_1>0$ we have
$\Pr(E_i)\le 1-\eps_1$ for all $i$.

Let $N_0$ be the maximum number of faces meeting any disk of radius $100d_0$,
and let $N_1$ be the number of \emph{bad} grey or $\cH$-black faces, i.e., grey or
black faces meeting $S$, $R_-$, and the exterior of $R_+$. Note that
$N_0$ and $N_1$ are bounded by constants as before.
Let $M$ be a huge constant to be chosen in a moment.

By Lemma~\ref{manyhold} (applied with $\min\{\eps/2,\eps_1\}$
in place of $\eps$), our assumption $h(R)\ge 1-\delta$, and \eqref{eibig},
if we choose $\delta$ small enough, then with probability at least $1-\eps/2$
at least $K=N_1+MN_0$ of the events $E_1,\ldots,E_N$ hold.
Let us condition on the states of all faces in $\cF_-$,
assuming that at least $K$ of the $E_i$ hold. It suffices
to show that the conditional probability 
that $\Hb(R_+)$ holds is then at least $1-\eps/2$.
As before, we use the fact that we have not yet
looked at the faces in $\cF_S$.

Since there are at most $N_1$ bad faces among the $f_i$, there is
a set $I_0$ of size at least $MN_0$ such that for every $i\in I_0$
the event $E_i$ holds and $f_i$ is good.
Using the greedy algorithm, we may pick a subset $I\subset I_0$ of size at least $M$
such that for distinct $i,j\in I_0$ the faces
$f_i$ and $f_j$ are at distance at least $30d_0$.
For $i\in I_0$, let $P_i$ be a path witnessing $E_i$. Our aim is to complete the proof as
before, but now showing that each $P_i$ has a not-too-small chance
of being extendable to cross $R_+$, and that these events
(that the particular extensions we look for are present) are independent,
so with high probability at least one hold.
The details are essentially as before: since $f_i$ is good, we may extend
$P_i$ within the face $f_i$ to meet our right-hand side $P_0$.
Then we find an extension $P_i^+$ of $P_i$ as before, remaining within
distance $10d_0$ of the end of $P_i$. The extensions meet disjoint sets of faces,
so we are done.
\end{proof}

Together, Lemmas~\ref{cont1},~\ref{cont2} and~\ref{oneway} show that
no crossing probability changes `too much' when a rectangle is moved slightly.
This statement needs a little interpretation: we could in principle
obtain explicit bounds in Lemmas~\ref{cont1} and~\ref{cont2}. However,
these turn out to be irrelevant. In the end, all we care about is whether
certain probabilities tend to 0 or tend to 1 as some parameter (the area
of the rectangles we consider) tends to infinity. It will thus be
convenient to `re-scale' all probabilities
by an increasing function $\phi:(0,1)\to \RR$ with $\phi(x)\to-\infty$
as $x\to 0$ and $\phi(x)\to\infty$ as $x\to 1$, in such a way
that the maximum change in a probability $p$ `allowed' by our lemmas
corresponds to a change in $\phi(p)$ of at most $2$, say.

To make this precise,
let $c_\blk$ be the constant given by Lemma~\ref{cont1}; recall that this does
not depend on the orientation of $R$. Let $c_\w$ be the corresponding constant
with black and white exchanged, and let $c_0=\min\{c_\blk,c_\w\}$.
Similarly, given $\eps>0$, let $\delta_0(\eps)=\min\{\delta_\blk,\delta_\w\}$,
where $\delta_\blk=\delta_\blk(\eps)$ is given by Lemma~\ref{cont2}, and $\delta_\w$
by Lemma~\ref{cont2} with black and white exchanged.

Define a sequence $(\eps_n)_{n\ge 0}$ inductively
by setting $\eps_0=1/2$
and $\eps_{n+1}=\min\{c_0\eps_n,\delta_0(\eps_n)\}$.
Set $\pi_n=1-\eps_n$ for $n\ge 0$ and $\pi_n=\eps_{-n}$ for $n\le 0$.
Consider the \emph{scaling function} $\phi:(0,1)\to\RR$ defined as follows:
set $\phi(\pi_n)=n$ for all $n\in \Z$, and interpolate linearly between these points.
Note that $\phi(1-p)=-\phi(p)$. This function (or rather its inverse) is illustrated
in Figure~\ref{fig_scaling}.
\begin{figure}[htb]
 \centering
 \input{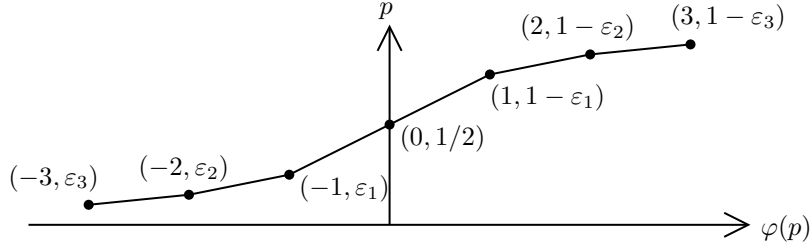}
 \caption{The inverse of the scaling function $p \mapsto \phi(p)$.}\label{fig_scaling}
\end{figure}

Recall that we call a rectangle \emph{large} if
both sides have length at least $100d_0$.

\begin{lemma}\label{cont}
Let $R_1$ and $R_2$ be two large rectangles
such that one of $R_1$ and $R_2$ is obtained from the other by moving one side
outwards by a distance between $0$ and $1$.
Let $f(R)$ be any of the four functions $\Pr(\Hb(R))$, $\Pr(\Hw(R))$,
$\Pr(\Vb(R))$ and $\Pr(\Vw(R))$. Then $|\phi(f(R_1))-\phi(f(R_2))|\le 2$.
\end{lemma}
\begin{proof}
Without loss of generality, $R_2$ extends $R_1$ horizontally.
Using Lemma~\ref{oneway} and the fact that $|\phi(1-p_1)-\phi(1-p_2)|=|-\phi(p_1)+\phi(p_2)|=
|\phi(p_1)-\phi(p_2)|$, we may assume that we are considering
horizontal crossings. Exchanging colours, we may assume that $f(R)=\Pr(\Hb(R))=h(R)$.
Let $p_j=f(R_j)=h(R_j)$, noting that $p_1\ge p_2$.
Let $i=\floor{\phi(p_1)}$, so $\pi_i\le p_1< \pi_{i+1}$.
If $i\le 0$ then by definition of $\pi_i$ we have $\pi_{i-1} = \eps_{-i+1}\le c_0 \eps_{-i} = c_0\pi_i$.
By Lemma~\ref{cont1} we have $p_2\ge c_0 p_1$, so $p_2\ge \pi_{i-1}$ and $\phi(p_2)\ge i-1$.

On the other hand, if $i\ge 1$ then $\pi_i=1-\eps_i\ge 1-\delta_0(\eps_{i-1})$.
Since $p_1\ge \pi_i$, Lemma~\ref{cont2} gives $p_2\ge 1-\eps_{i-1}=\pi_{i-1}$.
In either case we have $\pi_{i-1}\le p_2\le p_1< \pi_{i+1}$, so $i-1\le \phi(p_2)\le \phi(p_1)<i+1$
and the result follows.
\end{proof}

Lemma~\ref{cont} shows that if we measure probabilities in the right way, they don't change too much when 
we move a rectangle slightly. Since we have a lattice of translational symmetries,
this has the following consequence. (Recall that a rectangle is `large' if its
sides have length at least $100d_0$, where $d_0=d_0(\cH)$ is the constant
from Lemma~\ref{reg}.)

\begin{corollary}\label{c_trans}
There is a constant $C$ such that if $R'$ is a translate of a large
rectangle $R$, then $|\phi(h(R))-\phi(h(R'))|\le C$.
\end{corollary}
\begin{proof}
Pick a fundamental domain $D$ of the lattice $\cL$ of symmetries.
Since our colouring $\cC$ is invariant under translations corresponding
to elements of $\cL$,
we may assume that $R$ and $R'$ are related by translation by a vector in $D$.
Since $D$ is bounded, the result follows by applying Lemma~\ref{cont} a bounded number of times.
\end{proof}

Define $\psi:(0,1)\to (0,1)$ by $\psi(p)=\phi^{-1}(\phi(p)-C)$,
where $C$ is the constant given by Corollary~\ref{c_trans}.
Then $\psi$ is increasing. In fact, although we shall not use this, $\psi$
is strictly increasing, and $\psi(p)$ tends to $0$ as $p\to 0$ and to $1$ as $p\to 1$.
We may rewrite Corollary~\ref{c_trans} is the following more convenient form.

\begin{corollary}\label{c_trans_lower}
Suppose that $R$ and $R'$ are large rectangles with the same orientation, width and height.
Then $h(R')\ge \psi(h(R))$ and $v(R')\ge \psi(v(R))$.\noproof
\end{corollary}

We may also rotate a rectangle slightly without changing the crossing probabilities much.
\begin{corollary}\label{c_rot}
Let $R$ be any rectangle of \ww\ $a$ and \hh\ $b$, with $a,b\ge 100d_0$ (so $R$ is large),
and let $R'$ be obtained from $R$ by rotating it about its centre through an 
angle $\theta\le 1/(10\max\{a,b\})$.
If $f(\cdot)$ denotes any of the four crossing probability functions
considered in Lemma~\ref{cont},
then $|\phi(f(R))-\phi(f(R'))|\le 8$.
\end{corollary}
\begin{proof}
Without loss of generality we may assume that $f$ is the function
$f(\cdot)=h(\cdot)$ giving the probability of a black horizontal crossing.
Rotating and translating the rectangles and $\cC$ together as before
(or simply changing coordinates)
we may assume that $R=[-a,a]\times [-b,b]$. Let $R''=[-a-1,a+1]\times [-b+1,b-1]$.
Then any horizontal crossing of $R''$ crosses $R'$ horizontally, so $h(R')\ge h(R'')$.
Hence, by Lemma~\ref{cont}, $\phi(h(R'))\ge \phi(h(R''))\ge \phi(h(R))-8$.

A similar argument interchanging $R$ and $R'$ gives $\phi(h(R))\ge \phi(h(R'))-8$, so the result
follows.
\end{proof}

A key consequence of the lemma above is that for any given orientation,
we may find a large rectangle $R$ with $h(R)$ not too close to $0$ or $1$,
and that when we rotate, we can assume that the dimensions of $R$ vary `smoothly'.
For now we formalize only the first of these statements.
Given an angle $\theta$, we write
$h_\theta(m,n)$ and $v_\theta(m,n)$ for $\Pr(\Hb(R))$ and $\Pr(\Vb(R))$,
where $R$ is an $\xx$-by-$\yy$ $\theta$-aligned rectangle centred on the origin.
(Thus $v_\theta(m,n)=h_{\theta+\pi/2}(n,m)$.)

\begin{lemma}\label{1mod}
Let $\cC$ be a non-degenerate independent lattice colouring associated
to a hyperlattice $\cH$.
Given $L\ge 100d_0$, there is a constant $A_0$ that for any $A\ge A_0$ and any angle $\theta$,
there are $\xx,\yy\ge L$ with $\xx\yy=A$
such that $-4 \le \phi(h_\theta(\xx,\yy)) \le 4$.
\end{lemma}
\begin{proof}
By Lemma~\ref{reg} there is a constant $N=N(L)$ such that any disk of radius $2L$
meets at most $N$ faces of $\cH$. 
By Lemma~\ref{ba}, if $R$ is a rectangle of any orientation
with \ww\ $\xx\ge L$ and \hh\ $\yy=L$, then we can find $\floor{\xx/(10d_0)}\ge \xx/(20d_0)$
potentially white paths $P_i$ crossing $R$ from top to bottom, with these
paths separated by distances 
of at least $d_0$. Since the events that these paths are white are independent,
and each path meets at most $N$ grey faces, the probability that no $P_i$ is white is 
at most $\exp(-\alpha \xx)$ for some constant $\alpha>0$ that
does not depend on $\theta$.
By Lemma~\ref{oneway}, if any $P_i$ is white, then $\Hb(R)$ does not hold.
Taking $A$ large enough (i.e., $A\ge L^2$ and $A \ge L\alpha^{-1}|\log(\phi^{-1}(-4))|$), we thus have
$\phi(h_{\theta}(A/L,L))\le -4$;
similarly, if $A$ is large enough then
$\phi(h_{\theta}(L,A/L))\ge 4$.

Now consider a rectangle $R$ with area $A$ that varies smoothly between
these two extremes, centred always on the origin.
As $m$ varies, Lemma~\ref{cont} implies that $\phi(h_{\theta}(m,A/m))$
cannot jump by more than 8 at any point, and the result follows.
\end{proof}

\section{A rectangle-crossing lemma}\label{sec_RSW}

In the original context of independent bond percolation on the square
lattice, there are now several different proofs of the key lemma
of Russo~\cite{Russo} and Seymour and Welsh~\cite{SW};
see, for example,~\cite{ourKesten,ourKesten2,Voronoi}.
The various proofs extend (with differing degrees of additional complication)
to various more general classes of percolation model.
However, as far as we are aware, none of the published proofs
can be made to work in the context of general self-dual hyperlattices
-- in addition to various technical problems,
they all require symmetry assumptions that may not hold here.

In this section we shall prove an analogue of the Russo--Seymour--Welsh
Lemma for self-dual hyperlattice percolation. Since this proof is a little
involved, we first illustrate the key ideas by writing out the
argument for bond percolation on $\Z^2$. This amounts to reproving the original
RSW Lemma in a more complicated way than necessary.
Even among proofs using our new strategy, we do not aim to present the simplest,
but rather one that extends easily to hyperlattices.

\subsection{Bond percolation on $\Z^2$}\label{ss_z2}

Throughout this subsection we consider independent
bond percolation on $\Z^2$ with $p=1/2$. Thus a \emph{configuration} $\omega$
is an assignment of a \emph{state} (\emph{open} or \emph{closed})
to each edge $e$ of $\Z^2$, and $\Pr$ is the probability
measure on the set $\Omega=2^{E(\Z^2)}$
of configurations in which the states $\omega(e)$
of different bonds are independent and each bond is open with probability $1/2$.

All rectangles $R$ we consider will be aligned with the axes and have corners
with integer coordinates; a rectangle includes its boundary.
An \emph{open horizontal crossing} of $R$
is a path of open bonds in $R$ joining a vertex on the left to one on the right;
we write $H(R)$ for the event that $R$ has such a crossing. Similarly,
$V(R)$ is the event that $R$ has an open vertical crossing, defined analogously.

As inputs to the argument we shall present, we need two simple lemmas,
which \emph{do} make use of the symmetries of $\Z^2$; the main argument
will then use only translational symmetries.
The first lemma is a standard fact which is an easy consequence
of the self-duality of $\Z^2$. It is well known to
hold with $c_1=1/2$; see~\cite{ourKesten}, for example.
We write $c_1$ rather than $1/2$ since the main argument
below does not depend on the particular value of $c_1$, and in the case
of hyperlattices, the value of $c_1$ in the analogous statement will be different.

\begin{lemma}\label{Z1/2}
There is a constant $c_1>0$ such that if $S$ is any square in $\Z^2$ then
$\Pr(H(S))=\Pr(V(S))\ge c_1$.\noproof
\end{lemma}

Our second `input lemma' is the following consequence of Lemma~\ref{Z1/2},
whose proof also requires the use of symmetry. Here we can take $c_2=c_1^2/16$,
but again the value of $c_2$ is irrelevant later.

\begin{lemma}\label{Zdown}
Given an $n$-by-$n$ square $S$, let $E=E(S)$ be the event that there is an open vertical
crossing $P$ of $S$ such that the $x$-coordinates of the endpoints of $P$ differ
by at most $3n/5$. There is a constant $c_2>0$ such that $\Pr(E(S))\ge c_2$
for all squares $S$.
\end{lemma}
\begin{proof}
Let $c_1>0$ be as in Lemma~\ref{Z1/2} and consider $S=[0,n]^2$.
We may suppose that $\Pr(E)<c_1/2$.
Let $F_1$ be the event that $S$ has an open vertical crossing from some point
$(x,0)$ to some point $(x',n)$
with $x'-x>3n/5$, and $F_2$ the horizontal mirror image of this event.
Now $V(S)=E\cup F_1\cup F_2$, so we must have $\Pr(F_i)\ge c_1/4$ for some $i$.
Since $\Pr(F_1)=\Pr(F_2)$ by symmetry, it follows that $\Pr(F_1)=\Pr(F_2)\ge c_1/4$.
But then, by Harris's Lemma, $\Pr(F_1\cap F_2)\ge c_1^2/16$.
\begin{figure}[htb]
 \centering
 \input{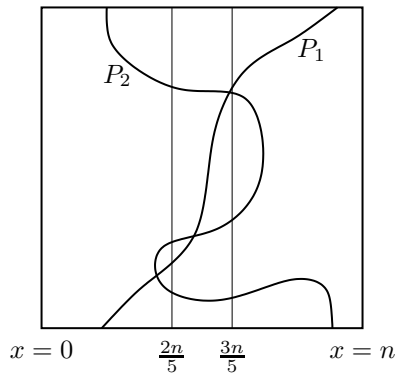}
\caption{Two open paths crossing a square $S$ vertically, $P_1$ from bottom-left to top-right,
and $P_2$ from bottom-right to top-left. Their union contains a path $P$ from bottom-left to top-left.}
\label{fig_F1F2}
\end{figure}
Suppose that $F_1$ and $F_2$ hold, and let $P_1$ and $P_2$ be open paths witnessing
these events, with $P_i$ joining $(x_i,0)$ to $(x_i',n)$; see Figure~\ref{fig_F1F2}.
Since $x_1'-x_1\ge 3n/5$, we have $x_1\le 2n/5$ and $x_1'\ge 3n/5$,
and similarly $x_2\ge 3n/5$ and $x_2'\le 2n/5$.
It follows that $P_1$ and $P_2$ cross. Hence there is an open path $P$ within $S$
joining $(x_1,0)$ to $(x_2',n)$. Since $0\le x_1,x_2'\le 2n/5$, this shows
that $E$ holds.
In conclusion, if $\Pr(E)<c_1/2$, then $\Pr(E)\ge c_1^2/16$, so $\Pr(E)\ge c_1^2/16>0$.
\end{proof}

Let us write $h(\xx,\yy)=\Pr(H(R))$ for the probability that
a rectangle $R$ of \ww\ $\xx$ and \hh\ $\yy$ has an open horizontal crossing,
and $v(\xx,\yy)=\Pr(V(R))$ for the probability that it has an open vertical crossing.
Our aim is to prove the following form of the RSW Lemma.
\begin{theorem}\label{th_RSWZ}
There is a constant $c>0$ such that $v(n,100n)\ge c$ for all $n$.
\end{theorem}
Of course, Theorem~\ref{th_RSWZ} is a well known result of Russo~\cite{Russo}
and Seymour and Welsh~\cite{SW}. As noted above, our aim in this subsection
is to present a (complicated) way of deducing Theorem~\ref{th_RSWZ} from
Lemmas~\ref{Z1/2} and~\ref{Zdown} using minimal properties
of the measure $\Pr$. In particular, we shall restrict
ourselves to properties that extend to general hyperlattice
percolation, so the argument will adapt to prove
Theorem~\ref{th_RSW1} below. For example, we shall use translational symmetry but
no other symmetry. We shall make repeated use of Harris's Lemma, and of the
geometric fact that open horizontal and vertical crossings of the same
rectangle must meet.

We also use one more very important property of the measure $\Pr$:
if a rectangle $R$
has an open horizontal crossing, then it has an \emph{uppermost
open horizontal crossing} $UH(R)$ with the property that
the event $UH(R)=P_0$ is independent of the states of all bonds
below $P_0$. Indeed, $UH(R)$ may be found by `exploring $R$ from above';
see~\cite{ourKesten}, for example. Similarly, if $H(R)$ holds
then $R$ has a \emph{lowest open horizontal crossing} $LH(R)$, defined
analogously, and found by exploring from below.

The proof of Theorem~\ref{th_RSWZ} that we shall present, although 
simpler than that of Theorem~\ref{th_RSW1}, is still somewhat lengthy.
We shall start with three lemmas, the first two of which are
standard observations.

In what follows, we shall often implicitly
assume that $n$ is `large enough', meaning larger than a suitable
constant $n_0$ depending on the parameters, e.g., $\eps$, that we choose.
To avoid clutter, we ignore the rounding of plane coordinates
to integers; it is easy to see that
the effect of rounding can be handled by adjusting the various constants
suitably. (Recall that in this subsection we are not proving new results, merely
rehearsing the arguments that we shall use in the next subsection; formally,
nothing outside this subsection depends on anything inside it. So we do
not feel the need to dot all i's and cross all t's.)

\begin{lemma}\label{l1e}
For any $\eps>0$ and $c'>0$ there is a $c>0$ such that for any $n$,
if $v(n,(1+\eps)n)>c'$ then $v(n,100n)>c$.
\end{lemma}
\begin{proof}
Given $\gamma\ge 1$, let $R_1$ and $R_3$ be rectangles of width $n$ and heights $\gamma n$
and $(1+\eps)n$, respectively, overlapping in a square $R_2$ of side $n$.
If $V(R_1)$, $H(R_2)$ and $V(R_3)$ all hold, then so does $V(R_1\cup R_3)$;
see Figure~\ref{fig_R123} for an illustration of this in a slightly different
context.
By the original form of Harris's Lemma~\cite{Harris}, it follows that
\begin{equation}\label{c1z}
 v(n,(\gamma+\eps)n)\ge v(n,\gamma n)h(n,n) v(n,(1+\eps)n).
\end{equation}
Since $v(n,n)\ge c_1$ and $h(n,n)\ge c_1$ by Lemma~\ref{Z1/2},
applying \eqref{c1z} inductively $\ceil{99/\eps}$ times
gives the result.
\end{proof}

Given two overlapping  $n$-by-$n$ squares $S_1$ and $S_2$ such that $S_2$ is
obtained from $S_1$ by translating it
upwards through a distance of at most $n$, let $\strip=\strip(S_1,S_2)$ denote the infinite strip 
bounded by the vertical lines containing the vertical sides of $S_1$ and $S_2$.
Let $J(S_1,S_2)$ be the event that $S_1$ and $S_2$ have open horizontal crossings
$P_1$ and $P_2$ that are \emph{joined} within $\strip$, meaning that there is an open path $P$
within $\strip$ joining some point of $P_1$ to some point of $P_2$; this
includes the case where $P_1$ and $P_2$ meet; see Figure~\ref{fig_J1J2}. Note
that when $J(S_1,S_2)$ holds, a minimal $P$ lies entirely between $P_1$ and $P_2$
in the strip, so we may assume that $P$ is contained in $S_1\cup S_2$. 
For later, note also that if there is an open
horizontal crossing $P_1$ of $S_1$ above
an open horizontal crossing $P_2$ of $S_2$, then $P_1$ is in fact contained
in $S_2$, and the crossings $P_1$ of $S_1$ and $P_1$ of $S_2$ meet, so
$J(S_1,S_2)$ holds.

\begin{lemma}\label{lJ}
For any $\eps>0$ and $c'>0$ there is a $c>0$ such that for any $n$,
if there exist $n$-by-$n$ squares $S_1$ and $S_2$
with $S_2$ obtained by translating $S_1$ upwards by a distance of $\eps n$
such that $\Pr(J(S_1,S_2))\ge c'$, then $v(n,100n)\ge c$.
\end{lemma}
\begin{figure}[htb]
 \centering
 \input{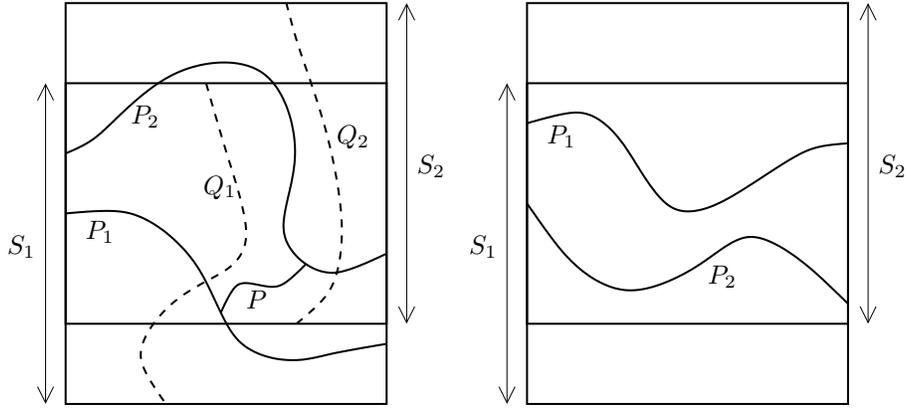}
 \caption{In the left figure, the solid paths illustrate the event $J(S_1,S_2)$. The dashed paths $Q_i$ are vertical crossings of the squares $S_i$. Since $P_i$ and $Q_i$
meet, the event $V(S_1\cup S_2)$ holds. The figure on the right
shows that if any horizontal crossing of $S_1$ is above any horizontal crossing
of $S_2$, then $J(S_1,S_2)$ holds - either crossing crosses both squares.}\label{fig_J1J2}
\end{figure}
\begin{proof}
The event $J=J(S_1,S_2)$ is increasing. Hence, by Harris's Lemma and Lemma~\ref{Z1/2},
the event $E=J\cap V(S_1)\cap V(S_2)$ has probability at least $c_1^2c'$.

Using the fact that horizontal and vertical crossings of the same square must
meet, it is easy to see that whenever $E$ holds, so does $V(S_1\cup S_2)$;
see Figure~\ref{fig_J1J2}. Hence
$v(n,(1+\eps)n)=\Pr(V(S_1\cup S_2))\ge \Pr(E) \ge c_1^2c'$, and
the result follows by applying Lemma~\ref{l1e} with $c_1^2c'$ in place of $c'$.
\end{proof}

Our next lemma is less run-of-the-mill.
Given $\eps>0$ and two $n$-by-$n$ squares $S_1$ and $S_2$ with $S_2$ obtained by translating
$S_1$ upwards by a distance of $\eps n/10$, define the strip $\strip=\strip(S_1,S_2)$
as above, and let $G_\eps(S_1,S_2)$ be the event
that $H(S_1)$ and $H(S_2)$ hold, the path $P_1=LH(S_1)$ is below $P_2=UH(S_2)$ in $\strip$,
and the area of $\strip$ between $P_1$ and $P_2$ is at most $\eps n^2$.
In other words, the lowest open horizontal crossing of the lower square and the highest
of the higher square do not meet, but they are `close together', in the
sense that the area between them is small.

\begin{lemma}\label{smallgap}
For any $0<\eps<1/10$ there are constants $c_3>0$ and $c>0$ such that for any $n$,
either there exist $n$-by-$n$ squares $S_1$ and $S_2$
as above with $\Pr(G_\eps(S_1,S_2))\ge c_3$, or $v(n,100n)>c$.
\end{lemma}
\begin{proof}
Set $N=2\ceil{2/\eps}$,
and, for $i=0,1,\ldots,N$, let $S_i=[0,n]\times [i\eps n/10,n+i\eps n/10]$,
so $S_{i+1}$ is obtained by translating $S_i$ upwards through a distance $\eps n/10$.

Let $H=H(S_0)\cap\ldots\cap H(S_N)$. Since each $H(S_i)$ has probability at least $c_1$,
by Harris's Lemma there is some $c'>0$ such that $\Pr(H)\ge c'$.
(We may take $c'=c_1^{N+1}$, but we prefer to be less specific, with an eye to the hyperlattice case.)
Set $c_3=c'/(2N)$. 
If for some $i$ the event $J(S_i,S_{i+1})$ has probability at least $c_3$,
then we are done by Lemma~\ref{lJ}. Let $J=\bigcup_{i=0}^{N-1} J(S_i,S_{i+1})$.
Then we may assume that $\Pr(J)\le Nc_3\le c'/2$.
Hence $\Pr(H\setminus J)\ge c'/2$.

We claim that if $H\setminus J$ holds, then so does $F=\bigcup_{i=0}^{N-1}G_\eps(S_i,S_{i+1})$.
Assuming the claim, the result follows, since for some $i$ we have
\[
 \Pr(G_\eps(S_i,S_{i+1})) \ge N^{-1} \Pr(H\setminus J) \ge c'/(2N) = c_3.
\]

Suppose then that $H\setminus J$ holds, and let $P_i^-$ and $P_i^+$ be the lowest
and highest open horizontal crossings of $S_i$.
Note that $P_i^-$ and $P_i^+$ may meet, but $P_i^-$ lies (weakly) below $P_i^+$.
Since $J$ does not hold, for $i=0,\ldots,N-1$, the path
$P_i^+$ is strictly below $P_{i+1}^-$; it follows that $P_i^-$ is
strictly below $P_{i+1}^-$.
For $i=0,\ldots,N-2$, let $A_i$ be the region in $\strip$ between $P_i^-$ and $P_{i+2}^-$.
Then the $A_i$ are disjoint. Since there are $\floor{N/2}\ge 2/\eps$ such regions
$A_i$, and their union is contained in a rectangle
of width $n$ and height $n+N\eps n/10 \le n+(6/\eps)\eps n/10 < 2n$, it follows
that some $A_i$ has area at most $\eps n^2$.
Since the region between $P_i^-$ and $P_{i+1}^+$ is contained
in $A_i$, it follows that $G_\eps(S_i,S_{i+1})$ holds, as required.
\end{proof}

Clearly, in the present context the events $G_\eps(S_i,S_{i+1})$ considered above
all have the same probability due to translational symmetry; with an eye
to the general case, we avoided using this
fact.

We now turn to the key idea, which is a rather involved way of generating a configuration.
Given a configuration $\omega$ and a vector $v\in \Z^2$,
let $\omega^v$ be obtained by translating $\omega$ through the vector $v$.
Thus the state of a bond $e$ in $\omega^v$ is the state of its translate $e-v$ in $\omega$.
Let us say that a random configuration $\omega$ has the \emph{standard distribution}
if it has the distribution corresponding to $\Pr$, so bonds are open independently
in $\omega$ and each is open with probability $1/2$.
From translation invariance, if $\omega$ is random with the standard distribution
and $v$ is constant, then $\omega^v$ has the standard distribution.
This conclusion also holds if $v$ is random, as long as $\omega$ and $v$ are independent.

Suppose we have some algorithm $\ALG$ whose input is a configuration $\omega$,
and that $\ALG$ examines the states of bonds one by one, with the next bond to be examined
depending on the results of previous examinations, but not on the states of any
other bonds. We assume that $\ALG$ terminates after a finite number of steps,
and write ${\cal S}={\cal S}_{\ALG}(\omega)$ for the
set of bonds examined by $\ALG$ when $\ALG$ is
run on the configuration $\omega$. Let $\omega_1$ and $\omega_2$ be independent standard
configurations. Define a new configuration $\omega$ by running $\ALG$
on $\omega_1$, 
setting $\omega(e)=\omega_1(e)$ if $e\in {\cal S}_{\ALG}(\omega_1)$ and $\omega(e)=\omega_2(e)$
otherwise. Then it is easy to check that $\omega$ has the standard
distribution: we can think of tossing coins corresponding to $\omega_1$
to determine the states of all bonds, looking at certain bonds (corresponding to $\cal S$),
and then retossing the coins we have not yet looked at.

Given an algorithm $\ALG$ as above, and a corresponding map ${\cal S}_{\ALG}$ from
the set $\Omega$ of all configurations to
the set of subsets of $E(\Z^2)$, define a map $f_\ALG$ from $\Omega\times\Omega\times\Z^2$
to $\Omega$ by
\begin{equation}\label{fA}
 \bb{f_\ALG(\omega_1,\omega_2,v)}(e)= \left\{
   \begin{array}{ll}
    \omega_1(e) & \hbox{if }e\in {\cal S}_{\ALG}(\omega_1), \\
    \omega_2(e-v) & \hbox{if }e\notin {\cal S}_{\ALG}(\omega_1). \\
   \end{array}
  \right.
\end{equation}
Combining the two observations above, we see that if $\omega_1$, $\omega_2$ and $v$
are independent and $\omega_1$ and $\omega_2$ have the standard distribution,
then $\omega=f_\ALG(\omega_1,\omega_2,v)$ does too.

\begin{proof}[Proof of Theorem~\ref{th_RSWZ}]
Recall that our task is to show that the probability $v(n,100n)$ that
an $n$-by-$100n$ rectangle has an open vertical crossing is bounded away
from zero, using Lemmas~\ref{Z1/2} and~\ref{Zdown} as `inputs', and
otherwise making no use of reflectional or rotational symmetry.

With an eye to later generalizations, 
suppose that $\alpha$, $\beta_1$, $\beta_2$ and $\eta$ are positive
constants satisfying
\begin{equation}\label{consts1}
  \alpha,\beta_1\le 1/3, \quad \beta_2\le 2, \quad\hbox{and }\eta>3\beta_1.
\end{equation}
Let $\Ed=\Ed(n)$ be the event that the rectangle $[0,\beta_1 n]\times [0,\beta_2 n]$
contains an open path from some point $(x,y)$ to some point $(x',y')$ with
$y'\ge y+\alpha n$ and $|x-x'|\le (1-\eta)|y-y'|$. (Here `v' stands for `vertical':
the overall orientation of the path is significantly closer to vertical
than to horizontal.)
Taking $\alpha=\beta_1=\beta_2=1/100$ and $\eta=1/10$, Lemma~\ref{Zdown} tells us that
for all (large enough) $n$, we have
$\Pr(\Ed(n))\ge c_2>0$. In the rest of the proof we assume only that
our various constants satisfy \eqref{consts1} and that, for these constants, $\Pr(\Ed)$ is
bounded away from $0$.

Pick $\gamma>0$ such that
\begin{equation}\label{consts2}
  \gamma\le 1/3\quad\hbox{and }\eta\ge 3\beta_1+3\gamma,
\end{equation}
and choose $\eps>0$ such that $\eps<\gamma$ and
\begin{equation}\label{e2}
 \eps < \gamma^2\alpha/10.
\end{equation}

Let $c_3$ and $c$ be the constants given by Lemma~\ref{smallgap}.
For any $n$, by Lemma~\ref{smallgap}, either $v(n,100n)\ge c$, in which case we are done,
or there are squares $S_1$ and $S_2$ with $S_2$ obtained by 
translating $S_1$ upwards by 
a distance of $\eps n/10$ such that
\begin{equation}\label{Fe}
 \Pr(G_\eps(S_1,S_2))\ge c_3.
\end{equation}
We may assume that the second case holds. By translational symmetry,
we may assume that $S_1=[0,n]^2$ and $S_2=[0,n]\times[\eps n/10,(1+\eps/10)n]$.

Recall that $J=J(S_1,S_2)$ is the event that there are open horizontal crossings of $S_1$ and $S_2$
that meet, or are connected by an open path lying within the strip
\[
 \strip = \{(x,y): 0\le x\le n\} \subset \RR^2
\]
generated  by $S_1\cup S_2$. Also, $G_\eps=G_\eps(S_1,S_2)$
is the event
that $H(S_1)$ and $H(S_2)$ hold, the path $P_1=LH(S_1)$ is below $P_2=UH(S_2)$ in $\strip$,
and the area of $\strip$ between $P_1$ and $P_2$ is at most $\eps n^2$.
We shall show that
\begin{equation}\label{newaim}
 \Pr(J\cap G_\eps) \ge c'
\end{equation}
for some constant $c'$ depending on the various constants we have chosen so far,
but not on $n$.
Then $\Pr(J)\ge c'$ and so, applying Lemma~\ref{lJ} and reducing $c$ if necessary,
we have $v(n,100n)\ge c$, as required.

Consider the following algorithm $\ALG$ for testing whether $G_\eps$ holds:
explore $S_1$ from below (as in~\cite{ourKesten}) to find its
lowest horizontal crossing $P_1$, if $H(S_1)$ holds.
Similarly, explore $S_2$ from above to find its uppermost horizontal crossing $P_2$,
if $H(S_2)$ holds; then,
from the positions of $P_1$ and $P_2$, decide whether $G_\eps$ holds. Define $f_{\ALG}$ as 
in \eqref{fA} above.

Let $\omega_1$, $\omega_2$ and $\rX$ be independent,
where the $\omega_i$ are random
configurations (with the standard distribution) and $\rX$ is uniformly random
on $[-5n,5n-1]^2=\{-5n,-5n-1,\ldots,5n-1\}^2\subset \Z^2$,
and let $\omega=f_{\ALG}(\omega_1,\omega_2,\rX)$,
so $\omega$ has the standard distribution. The reason for the slightly
incongruous notation is that the core of our argument will involve
conditioning on $\omega_1$ and $\omega_2$, but keeping $\rX$ random.

To establish \eqref{newaim}, we first examine $\omega$ to check whether $\omega\in G_\eps$.
This depends only on the states of bonds examined by the algorithm $\ALG$ above;
by the definition of $f_{\ALG}$, the state of such bonds in $\omega$ is the same
as in $\omega_1$. Thus $\omega\in G_\eps$ if and only if $\omega_1\in G_\eps$.
When this event holds,
let $P_1$ and $P_2$ be the paths defined above, and write $A$ for the part of the strip $\strip$
on or below $P_1$, $B$ for the part of $\strip$ on or above $P_2$, and $G$ for the `gap' between $P_1$
and $P_2$, i.e., the rest of $\strip$; see Figure~\ref{fig_P*}.
Note that all bonds whose interiors lie within $G$ have their states in $\omega=f_\ALG(\omega_1,\omega_2,\rX)$ given
by $\omega_2^{\rX}$.
Suppose that $\omega_2$ contains an open path $P$ from a point $u$ to a point $v$.
Then $\omega_2^{\rX}$ contains the open path $P+\rX$, the translate
of $P$ through $X$, joining $u+\rX$ to $v+\rX$. If $u+\rX\in A$, $v+\rX\in B$,
and $P+\rX$ remains within the strip $0\le x\le n$,
then the minimal subpath $P'$ of $P+\rX$ meeting $A$ and $B$ contains bonds only in $G$, so this path $P'$
is present in $\omega$ and joins $P_1$ and $P_2$, and $\omega\in J$.

Let $E$ be the event that $\omega_1\in G_\eps$ and $\omega_2\in \Ed$.
Since $\omega_1$ and $\omega_2$ are independent, we have $\Pr(E)=\Pr(G_\eps)\Pr(\Ed)>c_3c_2$.
For the rest of the proof we condition on $\omega_1$ and $\omega_2$,
so the only remaining randomness is
in the choice of $\rX$. We assume that $E$ holds; 
we shall show that for any $\omega_1$ and $\omega_2$ such that $E$ holds,
the conditional probability that $J$
holds satisfies
\begin{equation}\label{aim}
 \Pr(J\mid \omega_1,\omega_2) \ge c_4 = \gamma^2\alpha/400.
\end{equation}
Then we have $\Pr(J)\ge c_4\Pr(E)\ge c_4c_3c_2>0$, establishing \eqref{newaim}.
It remains only to prove \eqref{aim}.

Let us choose an open path $P$ in the configuration $\omega_2$ witnessing
$\omega_2\in \Ed$; thus $P$
lies within $[0,\beta_1 n]\times [0,\beta_2 n]$ and joins some point $(x,y)$
to some $(x',y')$ with $y'-y\ge \alpha n$ and $|x'-x|\le (1-\eta)|y-y'|$.
Let $v$ be the vector $(x'-x,y'-y)$.
Let $I=\ceil{(1+2\gamma)n/(y'-y)}$, and note that
\begin{equation}\label{Imax}
 I\le \ceil{(1+2\gamma)/\alpha} \le 2/\alpha,
\end{equation}
since $y'-y\ge \alpha n$ and $\alpha,\gamma\le 1/3$.
For $0\le i\le I$, set $\off_i=iv$; we think of $\off_i$ as an \emph{offset},
for reasons that will hopefully become clear.

Let $P^*$ be the (`virtual', in the sense that it is not known to be open
in any configuration we are considering) path formed by starting
at the origin and concatenating $I$ copies of $P$. Thus $P^*$ is the union
of the paths $P_1^*,\ldots,P_I^*$, where each $P_i^*$ is the translate of $P$
joining the point $\off_{i-1}$ to $\off_i$; see Figure~\ref{fig_P*}.

\begin{figure}[htb]
 \centering
 \input{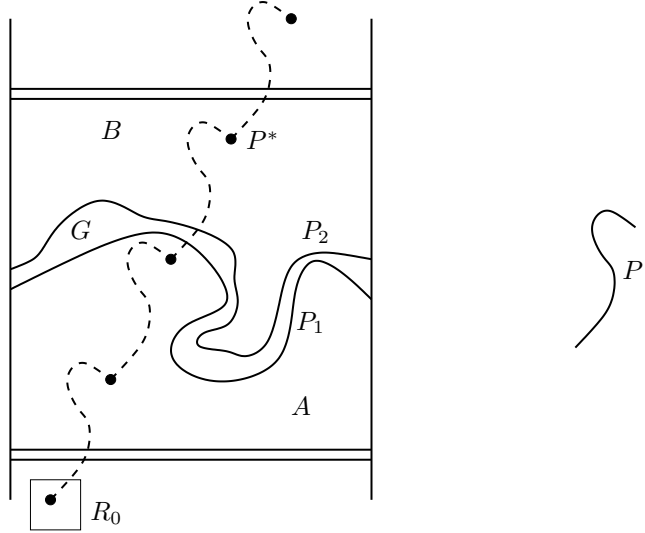}
\caption{The main part of the figure shows the lowest crossing $P_1$ of $S_1$ and
highest crossing $P_2$ of $S_2$, dividing the strip $T$ into the regions $A$, $G$
and $B$. In the final configuration $\omega$, bonds in $A\cup B$
have their states given by $\omega_1$; those in $G$ are given by a random translate of $\omega_2$.
On the right, an open path $P$ in
$\omega_2$ is shown; on the left we see a `virtual' chain $P^*$
of translates of $P$, starting somewhere in the square $R_0$. Since $G$ is small,
for most starting points in $R_0$, at least one copy of $P$ in $P^*$ crosses
from $A$ to $B$. It follows that with not-too-small probability,
the \emph{random} translate of $P$ present in $\omega_2^{\rX}$ crosses
from $A$ to $B$, giving an open path in $\omega$ joining $P_1$ to $P_2$.}
\label{fig_P*}
\end{figure}

If $x'\ge x$, let $v_0$ be the point $((\beta_1+\gamma/2)n,-\gamma n/2)$;
otherwise, set $v_0=(n-(\beta_1+\gamma/2)n,-\gamma n/2)$.
We claim that if a translate $\widetilde{P}^*$ of $P^*$ starts at a point within the square
$R_0$ of side $\gamma n$ centred at $v_0$, then this translate
lies entirely within the strip $\strip$.
Indeed, in the case $x'\ge x$,
the $x$-coordinate of any point of $\widetilde{P}^*$ is within $(\beta_1+\gamma/2)n$
of the $x$-coordinate of some point $v_0+\off_i$, and hence is
at least $0$ and at most
\begin{eqnarray*}
 (2\beta_1+\gamma)n+I(x'-x) &\le& (3\beta_1+\gamma) n+(I-1)(x'-x) \\
 &\le& (3\beta_1+\gamma) n+(I-1)(1-\eta)(y'-y)  \\
 &\le& (3\beta_1+\gamma+(1-\eta)(1+2\gamma))n \le n,
\end{eqnarray*}
using the assumption that $\eta\ge 3\beta_1+3\gamma$ in the final step.
In the case $x'<x$ the argument is similar, subtracting all $x$-coordinates from $n$.

Note that $R_0$ lies entirely below $S_1$, so, recalling
that $A$ and $B$ are the portions of the strip $\strip$
that lie below $P_1$ and above $P_2$, respectively, we have $R_0\subset A$.

If a translate of $P^*$ starts at a point in $R_0$,
then it ends at a point with $y$-coordinate at least
$-\gamma n+I(y'-y)\ge -\gamma n+(1+2\gamma) n =(1+\gamma)n \ge (1+\eps)n$,
so its upper endpoint is above $S_2$, and hence lies in $B$.

For any point $v\in R_0$, we have seen that $v+\off_0=v\in A$,
while $v+\off_I\in B$. Also, $v+\off_i\in \strip=A\cup G\cup B$ for $0\le i\le I$.
Hence, either some point $v+\off_i$, $1\le i<I$, lies in the `gap' $G$,
or there is some $i$ such that $v+\off_{i-1}\in A$ and $v+\off_i\in B$.
Let us colour the points of $R_0$ with $I+1$ colours, assigning colour 0 in the first
case, and colour $i$ in the second (choosing the minimal $i$ if there are several).
Let $C_i$ denote the set of points in $R_0$ assigned colour $i$.

Now $C_0$ is a subset of the union of $I-1$ translates
of $G$. Since $G$ has area at most $\eps n^2$, while $I\le 2/\alpha$,
the area of $C_0$ is thus at most $2\eps n^2/\alpha\le \gamma^2n^2/5$,
recalling \eqref{e2}. Since $R_0$ has area $\gamma^2 n^2$, it follows
that there is some $i>0$ for which $C_i$ has area at least
$I^{-1}\gamma^2n^2/2\ge \gamma^2\alpha n^2/4$.

Recall that in constructing our random configuration $\omega$ as
$\omega=f_{\ALG}(\omega_1,\omega_2,\rX)$, we shift the configuration $\omega_2$
by a random vector $\rX$
uniformly distributed on (the integer points in) $[-5n,5n-1]^2$. Recall also that the open path $P$
in $\omega_2$ starts at $(x,y)\in [0,\beta_1 n]\times[0,\beta_2 n]\subset[0,2n]^2$.
Consider the set $C_i'=C_i+\off_{i-1}-(x,y)$.
Then $C_i'$ has the same area as $C_i$, and certainly lies within $[-5n,5n-1]^2$.
Hence the probability that $\rX$ falls in $C_i'$ is at least
$(\gamma^2\alpha n^2/4)/(100n^2)=\gamma^2\alpha/400$.
But when this happens, the translate $P+X$ of $P$ starts
at a point of $C_i+\off_{i-1}$. Hence we may think of $P+X$ as the path $P_i^*=P_1^*+\off_{i-1}$
in a translate of $P^*$ starting at a point of $C_i$.
From the comments above and the definition of $C_i$ it follows that
$P+X$ lies entirely within $T$, starts in $A$, and ends in $B$.
As noted earlier, the presence of such a path in $\omega_2^{\rX}$
guarantees that the configuration $\omega=f_\ALG(\omega_1,\omega_2,\rX)$ has
the property $J$. Hence the conditional probability that
$\omega$ belongs to $J$ given $\omega_1$ and $\omega_2$ is at least
$\gamma^2\alpha/400$, establishing \eqref{aim} and completing the proof.
\end{proof}

\begin{remark}\label{Rchain}
The key step of the proof above involved selecting an open path $P$ in $\omega_2$,
and then chaining together translates of $P$ to
form a path $P^*$ with the following properties: $P^*$ stays well within the strip
$\strip$, starts well below $S_1$, and ends well above $S_2$, where `well within' means
at least a distance $\gamma n/2$ away from the boundary.
This elbow room ensures that we can translate
$P^*$ through distances of up to $\gamma n/2$ while retaining the properties
of starting below $S_1$, ending above $S_2$, and remaining within $\strip$.
Then we randomly
shifted the whole path $P^*$, and randomly chose one of the links in the chain to focus on,
thinking of this link as being the final random shift $P'$ of the path $P$ in $\omega_2$.
Since it is unlikely that any randomly shifted link starts or ends in the gap $G$, there
is always at least one link that crosses from $A$ to $B$, so the probability
that $P'$ does so is bounded away from zero. In order to construct
a suitable $P^*$, we required that $P$ remain within a fairly small region,
and that its endpoints be significantly further apart vertically than horizontally.

It is just as easy to start from more than one path in $\omega_2$.
Indeed, suppose that there are constants $c>0$ and $C$
such that with probability at least $c$ the configuration
$\omega_2$ contains a set of open paths $P_j$
such that we can chain together at most $C$ paths $P_i^*$, each
of which is a translate of some $P_j$, to form a path $P^*$ with the properties above.
Note that
this will hold (for example) whenever $\omega_2$ contains paths $P_1$ and $P_2$
each of which stays within some not-too-large region, such that the overall
directions of $P_1$ and $P_2$ are significantly different, and each is not too short,
in the sense that the vector from the start to the end is not too short.
Then one can always chain $O(1)$ copies together to produce an approximation to a vertical line.
Taking $\off_{i-1}$ to be the start of the $i$th path in the chain as above,
the proof goes through essentially unaltered, except that at the very end
we set $C_i'=C_i+\off_{i-1}-(x_j,y_j)$ if the $i$th path in our chain is a copy
of $P_j$, where $(x_j,y_j)$ is the starting point of $P_j$.
\end{remark}

\subsection{A rectangle-crossing lemma for hyperlattices}\label{ss_RSWh}

Our aim in this subsection is to prove an analogue of Theorem~\ref{th_RSWZ}
in the context of percolation on plane hyperlattices; we have already
illustrated the main ideas in a simpler context in the previous subsection.

Recall that a plane hyperlattice $(\cH,\cL)$,
originally defined as an embedding of a hypergraph, may also be
defined simply as a cubic map whose faces are properly
coloured black, white and grey, in a way that is invariant under translations
through elements of the lattice $\cL$; see Section~\ref{sec_mod}.
Recall from Section~\ref{sec_cc} that an \emph{independent lattice colouring} $\cC$
associated to $\cH$ is a random black/white-coloured map
obtained from $\cH$ as follows: First subdivide each grey face
into {\em subfaces} in a deterministic manner, keeping the resulting map cubic.
Then recolour these subfaces randomly black and white,
with the colourings inside different grey faces independent, such
that if one grey face is obtained by translating another through an
element of $\cL$, then their colourings have the same distribution.
Such a random colouring $\cC$ gives rise to a hyperlattice percolation model $\cH(\vecp)$:
simply take the \plp\ of a hyperedge $e$ to be the partition of the black
faces around the corresponding grey face $F_e$ induced by the (black part of) the colouring
of $F_e$. We say that $\cC$ \emph{realizes} the model $\cH(\vecp)$.

It turns out that, at one point in the coming argument, we may need to modify
the colouring within certain faces. To enable this, we need
our random colouring to satisfy a certain technical assumption.
Given a colouring $\chi$ of a grey face $F_e$, the \emph{colour components}
of $\chi$ are the maximal connected monochromatic subsets
of $F_e$.
\begin{definition}\label{maldef}
An independent lattice colouring $\cC$ is \emph{malleable} if
two conditions hold. First, within each grey face $F_e$, the all-white colouring
has positive probability. Second, if $\chi$ is a colouring of $F_e$
with positive probability, and $\chi'$ is obtained from $\chi$ by recolouring
a white component black, then $\chi'$ has positive probability.
\end{definition}
Note that, in a malleable colouring, the all-black colouring of a grey
face necessarily has positive probability, so a malleable colouring is non-degenerate
(meaning, as before, that
within each grey face, the all-black and all-white colourings have positive probability).

In general, recolouring as above may change the partition $\pi$ corresponding
to the colouring $\chi$ in many different ways. However, if we restrict the colourings
suitably, the situation becomes simpler. Let us call a part of a partition $\pi$
\emph{non-trivial} if it is not a singleton. A colouring $\chi$ of a grey face
$F_e$ is \emph{minimal} if its black components are in one-to-one correspondence
with the non-trivial parts of $\pi$, and its white components with those
of the dual partition $\pi^*$. Let us say that a part $P$ of $\pi$
is \emph{adjacent} to a part $P'$ of $\pi^*$ if, in the original polygon formulation
of non-crossing partitions, $P$ contains a vertex $v$ and $P'$ an edge $e$
incident to $v$. This corresponds to $P$ containing a black vertex
adjacent to a white dual vertex in $P'$.
If $\chi$ is a minimal colouring associated to a partition $\pi$, then recolouring a white component black
has the effect of uniting all parts of $\pi$ adjacent to some non-trivial part of $\pi^*$.
We call such an operation on a partition a \emph{joining}. In the dual, the operation
is simpler: simply split a non-trivial part into singletons.

Given a hyperedge $e$, by the \emph{top} partition of its vertices we mean
the partition into a single part. The \emph{bottom} partition is that
into singletons. Recall that a probability vector $\vecp$ associated
to a hyperlattice $(\cH,\cL)$ assigns a probability $p_{i,\pi}$ to each
non-crossing partition $\pi$ of the vertices of a hyperedge $e$,
where $i$ encodes which orbit of the action of $\cL$
the hyperedge $e$ belongs to.

A hyperlattice percolation model $\cH(\vecp)$ is \emph{non-degenerate} if $p_{i,\pi}>0$
whenever $\pi$ is a top or bottom partition.

\begin{definition}\label{maldef2}
A hyperlattice percolation model
$\cH(\vecp)$ is \emph{malleable} if it is non-degenerate and, whenever $p_{i,\pi}>0$
and $\pi'$ is obtained from $\pi$ by a joining operation as defined above,
then $p_{i,\pi'}>0$.
\end{definition}
Note that if $\cH$ is 3-uniform, then any non-degenerate model $\cH(\vecp)$
is automatically malleable: any joining operation results in the top
partition. Also, any $\vecp$ assigning positive probability to all top
and bottom partitions but to no other partitions gives a malleable model:
the unique joining operation that may be performed on a bottom partition
yields the corresponding top partition. Thus malleability
holds automatically in the `site percolation' models
considered in the discussion surrounding Theorem~\ref{th2}.

If $\cH(\vecp)$ is self-dual, then malleability
is equivalent to its dual formulation, that if $p_{i,\pi}>0$
and $\pi'$ is obtained from $\pi$ by splitting a part into
singletons, then $p_{i,\pi'}>0$.

The next lemma captures the connection between the notions of malleability
for probability vectors and for lattice colourings.

\begin{lemma}\label{regpos}
Let $\cH(\vecp)$ be a malleable hyperlattice percolation model.
Then $\cH(\vecp)$ may be realized by a malleable independent lattice colouring.
\end{lemma}
\begin{proof}
Regard $\cH$ as a 3-coloured cubic map, as usual.
For each grey face $F_e$ corresponding
to hyperedge $e$ with at least 3 vertices,
subdivide it into subfaces as in Figure~\ref{fig_shading6}. (When $|e|=3$ one can
also use the simpler subdivision shown in Figure~\ref{fig_shading3}.)
It is not hard to check that for any non-crossing partition $\pi$ of the vertices
of $e$, there is at least one black/white colouring of the subfaces
of $F_e$ that gives a \emph{minimal} colouring realizing the partition $\pi$.
If there are $N_{\pi}$ such colourings, assign each probability $p_{i,\pi}/N_{\pi}$,
where $i$ is the probability vector entry corresponding to $e$.
Since $\cH(\vecp)$ is non-degenerate, the all-black and all-white colourings (which
are minimal) receive positive probability.
Also, if $\chi$ is any colouring of $F_e$ receiving positive probability,
then $p_{i,\pi}>0$ for the corresponding $\pi$. If $\chi'$ is obtained
from $\chi$ by recolouring a white component to black, then $\chi'$ 
is minimal, and corresponds to a partition $\pi'$ obtained from $\pi$
by a joining operation. Since $\cH(\vecp)$ is malleable, we have $p_{i,\pi'}>0$,
so $\chi'$ has positive probability.

For a hyperedge $e$ with $|e|=2$, there is no need to subdivide $F_e$ at all; simply
colour $F_e$ black or white with the appropriate (positive) probabilities.
Finally, if $|e|=1$ then the colouring of $F_e$ is irrelevant, so we may colour
$F_e$ black with probability $1/2$ and white otherwise.
\end{proof}

Of course, a similar but simpler argument shows that any non-degenerate hyperlattice percolation
model may be realized by a non-degenerate independent lattice colouring.

As noted in Section~\ref{sec_cc}, a non-degenerate independent lattice colouring $\cC$
corresponds to a product probability measure on a power of a certain
poset in a natural way; furthermore, the non-degeneracy condition
ensures that Lemma~\ref{ourHarris2} applies both to this poset and to its
reverse. Events defined by the existence of black paths with certain
properties are upsets; events defined by the existence of white
paths are downsets. 

For the rest of this section we fix a malleable
independent lattice colouring $\cC$.
Recall that a black horizontal crossing of a rectangle $R$ is a piecewise-linear
path $P$ in the plane joining a point on the left-hand side of $R$ to a point on the right
and otherwise lying in the interior of $R$, such that every point of $P$
is black. White vertical crossings are defined similarly, and so on.
In our product probability space, events such as $\Hb(R)$ are increasing.

We shall use the following lemma, which applies to all
black/white colourings of cubic planar maps, i.e., involves no randomness.

\begin{lemma}\label{oneway_I}                                                                       
Let $R$ be a rectangle in $\RR^2$ in general position with respect
to a colouring $\cC$.
Then precisely one of the events $\Hb(R)$ and $\Vw(R)$ holds.  
\end{lemma}                                                                                      
\begin{proof}                                                                             
As in~\cite[Ch. 8, Lemma 12]{BRbook}, from where Figure~\ref{fig_ow} is adapted,
re-colour the points outside $R$ as              
in Figure~\ref{fig_ow}, and consider the interfaces between black and white regions.
\begin{figure}[htb]                                                                              
\[  \epsfig{file=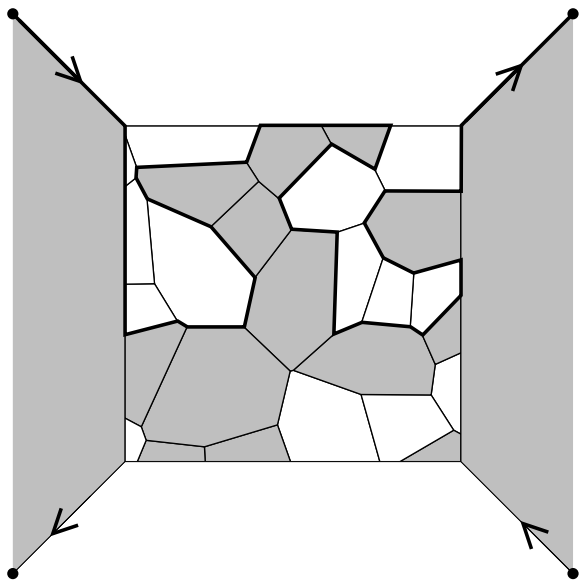,width=3in}  \]                                             
 \caption{The shading inside the (square) rectangle $R$ is from $\cC$. The event $\Hb(R)$
 holds if and only if the outer black regions are joined by a black path,
 and $\Vw(R)$ holds if and only if the outer white regions are joined by a white path.
Tracing the interface between black and white regions shows that one of these events must hold.}
 \label{fig_ow}
\end{figure} 
\end{proof}     

We shall need the equivalent of highest and lowest open crossings; these
are the highest and lowest \emph{black} crossings of a rectangle, defined
below. The definitions require a little care, as the independence properties
are not quite what one would like.

Given a rectangle $R$ and our colouring $\cC$,
shade the outside of the rectangle
as in Figure~\ref{fig_ow}. Let $I^+$ be the \emph{interface} shown by the thick black line,
starting at the top left corner. Thus if $\Hb(R)$ holds, then $I^+$ leaves
the rectangle $R$ from the top
right corner. Let $I^+_0$ be the minimal subpath of $I^+$ meeting both vertical sides of $R$.
In the figure, points just to the right of $I^+$ are black. Hence, points just to the right of $I^+_0$
give a black path within $R$ joining the left to the right. We call this path the \emph{highest
black horizontal crossing} of $R$, and denote it $UH(R)$.
Similarly, again assuming $\Hb(R)$ holds, the interface $I^-$ starting at the bottom right corner
leaves at the bottom left; we define $I^-_0$ to be the minimal subpath joining the vertical
sides of $R$; the points to the right of $I^-_0$ form the 
\emph{lowest black horizontal crossing} of $R$, written $LH(R)$. In fact,
we can usually work directly with the interfaces $I_0^-$ and $I_0^+$.

As before, a rectangle $R$
is \emph{large} if all its sides have length at least $100d_0$,
where $d_0$ is the constant from Lemma~\ref{reg}.

Given an angle $\theta$, a $\theta$-aligned rectangle $R'$,
and `length scales' $\xx$ and $\yy$, let $\Ed(R',\xx,\yy,\alpha,\eta)$
be the event that $R'$ contains a black path $P$ such that,
after rotating so that $R'$ is horizontal,
the endpoints $(x,y)$ and $(x',y')$
of $P$ satisfy $(y'-y)/\yy \ge \alpha$ and $|x'-x|/\xx \le (1-\eta) (y'-y)/\yy$.

Our aim now is to adapt the proof of Theorem~\ref{th_RSWZ} to prove the following result.

\begin{theorem}\label{th_RSW1}
Let $\cC$ be a malleable independent lattice colouring. Given constants
$\alpha,\beta_1\le 1/3$, $\beta_2\le 2$, $\eta>3\beta_1$, and $c_1$,
$c_2> 0$, there exists a constant $c>0$ such that the following holds.
Suppose that $R$ is a large rectangle with \ww\ $\xx$ and \hh\ $\yy$
and any orientation, and $R'$ is a large rectangle
with the same orientation, \ww\ $\beta_1\xx$, and \hh\ $\beta_2\yy$.
If $\Pr(\Hb(R))\ge c_1$, $\Pr(\Vb(R))\ge c_1$, and $\Pr(\Ed(R',\xx,\yy,\alpha,\eta))\ge c_2$,
then $\Pr(\Vb(R''))\ge c$
for any rectangle $R''$ with the same orientation as $R$, \ww\ $\xx$ and \hh\ $100\yy$.
\end{theorem}

Our proof of Theorem~\ref{th_RSW1} will follow that of Theorem~\ref{th_RSWZ} in
the previous subsection,
rescaling to map the square $S_1$ considered there onto the rectangle $R$.
Unfortunately, there are various additional complications; for example, we cannot assume
that congruent rectangles with the same orientation have the same crossing
probabilities. These complications can be dealt with using Corollary~\ref{c_trans_lower}.
There will also be some other difficulties.

\begin{remark}\label{Rsym}
As noted in Section~\ref{sec_cc}, although the probabilities of
events such as the existence of various crossings of a rectangle
$R$ will depend very much on the orientation of $R$, and to
a lesser extent on its position, all the lower bounds we shall prove
will depend on $\cC$ only via three quantities: the quantity
$d_0$ appearing in Lemma~\ref{reg} (which provides
an upper bound on the diameter of a face), the maximum
number $N$ of faces meeting any disk of radius $1$, and
the minimum probability $p_0$ of a configuration within a face.
These three quantities are preserved by rotations and translations,
so whenever we consider a single rectangle $R$, we may rotate
and translate $R$ \emph{and $\cC$} together so
that $R$ has the form $[0,\xx]\times [0,\yy]$.
\end{remark}

Let us write $h(R)$ for $\Pr(\Hb(R))$ and $v(R)$ for $\Pr(\Vb(R))$.
Also, we write $h(\xx,\yy)$ for the probability that $[0,\xx]\times [0,\yy]$ has a black horizontal
crossing, and $v(\xx,\yy)$ for the probability that it has a black vertical crossing.
Recall that by Corollary~\ref{c_trans_lower} there is an increasing function
$\psi:(0,1)\to(0,1)$ such that if $R$ and $R'$ are large rectangles with the same orientation,
width and height, then $h(R')\ge \psi(h(R))$ and $v(R')\ge \psi(v(R))$.

Recall that the colouring $\cC$ we are considering is malleable and hence non-degenerate,
meaning that within any grey face, the all-black and all-white colourings have positive
probabilities. As noted earlier, non-degeneracy
allows us to apply our Harris-type lemma, Lemma~\ref{ourHarris2}, to two black-increasing
events, or two white-increasing events.
Throughout the proof of Theorem~\ref{th_RSW1} we write $F$ for
the function whose existence is guaranteed
by Lemma~\ref{ourHarris2}, so for any two black-increasing events $A$ and $B$
we have
\begin{equation}\label{lH}
 \Pr(A\cap B)\ge F(\Pr(A),\Pr(B)).
\end{equation}

The first step in our proof of Theorem~\ref{th_RSW1} is the analogue of Lemma~\ref{l1e};
there is an additional assumption (that $h(\xx,\yy)$ and $v(\xx,\yy)$ are at least $c_1$)
since we do \emph{not} have the analogue of Lemma~\ref{Z1/2} in this context.
\begin{lemma}\label{Cextend}
Let $\cC$ be a non-degenerate independent lattice colouring, and
let $\eps>0$, $c_1>0$ and $c'>0$. There is a $c>0$ such that
for any $\xx,\yy\ge 100d_0$, if $h(\xx,\yy)\ge c_1$, $v(\xx,\yy)\ge c_1$, and
$v(\xx,(1+\eps)\yy)\ge c'$, then $v(\xx,100\yy)\ge c$.
\end{lemma}
\begin{proof}
Fix $\eps$, $c_1$ and $c'>0$.
Let $\alpha_0=c_1$; for $i\ge 0$ let
\[
 \alpha_{i+1}=F\bb{ F(\alpha_i,\psi(c_1))\ ,\psi(c') },
\]
where $\psi$ is the function appearing in Corollary~\ref{c_trans_lower}, 
and set $c=\alpha_{\ceil{99/\eps}}$. Note that $c>0$.

Suppose that $\xx,\yy\ge 100d_0$, $h(\xx,\yy)\ge c_1$, $v(\xx,\yy)\ge c_1$, and $v(\xx,(1+\eps)\yy)\ge c'$.
We claim that for every $i\ge 0$ we have
\begin{equation}\label{cl1}
 v(\xx,(1+i\eps)\yy)\ge \alpha_i.
\end{equation}
Setting $i=\ceil{99/\eps}$, we then have $1+i\eps\ge 100$,
so $v(\xx,100\yy)\ge v(\xx,(1+i\eps)\yy)\ge c$,
and the result follows.

We prove \eqref{cl1} by induction. For $i=0$ it is true by assumption. 
Turning to the induction step, suppose that
\eqref{cl1} holds for some $i$, and consider
the rectangles $R_1=[0,\xx]\times[0,(1+i\eps)\yy]$,
$R_2=[0,\xx]\times [i\eps\yy,(1+i\eps)\yy]$, and $R_3=[0,\xx]\times[i\eps\yy,1+(i+1)\eps\yy]$,
as in Figure~\ref{fig_R123}.
Let $R_2'=[0,\xx]\times[0,\yy]$ and $R_3'=[0,\xx]\times [0,(1+\eps)\yy]$,
so $R_2'$ and $R_3'$ are images of $R_2$ and $R_3$ under appropriate translations.
\begin{figure}[htb]
 \centering
 \input{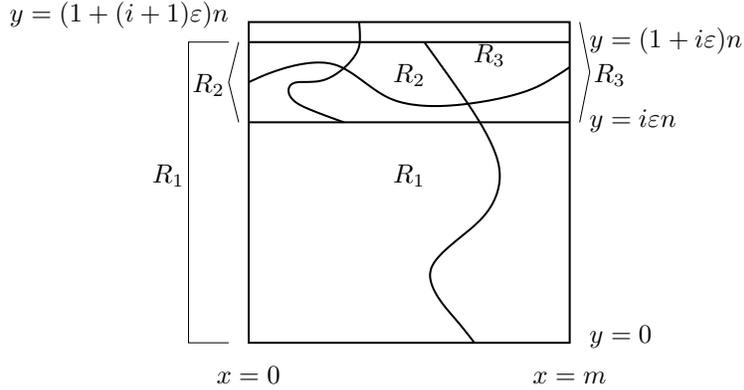}
 \caption{The rectangles $R_1$ and $R_3$, with their intersection $R_2$.
Whenever $R_1$ and $R_3$ have black vertical crossings and $R_2$ has a black
horizontal crossing, then these crossings can be combined to form
a black vertical crossing of $R_1\cup R_3$.}\label{fig_R123}
\end{figure}

By the induction hypothesis we have $\Pr(\Vb(R_1))=v(\xx,(1+i\eps)\yy)\ge \alpha_i$
and by assumption we have $\Pr(\Hb(R_2'))=h(\xx,\yy)\ge c_1$. Hence, by Corollary~\ref{c_trans_lower},
we have $\Pr(\Hb(R_2))\ge \psi(c_1)$.
Similarly, by assumption $\Pr(\Vb(R_3'))=v(\xx,(1+\eps)\yy)\ge c'$, so
$\Pr(\Vb(R_3))\ge \psi(c')$.

Since $\Vb(R_1)$ and $\Hb(R_2)$ are increasing events,
from \eqref{lH} it follows that $\Pr(\Vb(R_1)\cap \Hb(R_2))\ge F(\alpha_i,\psi(c_1))$.
Applying \eqref{lH} to the increasing events $\Vb(R_1)\cap \Hb(R_2)$ and $\Vb(R_3)$,
it follows that with probability at least $\alpha_{i+1}$,
the events $\Vb(R_1)$, $\Hb(R_2)$ and $\Vb(R_3)$ all hold. Choosing black paths $P_1$, $P_2$
and $P_3$ witnessing these events, $P_2$ meets both $P_1$ and $P_3$, and
it follows that $\Vb(R_1\cup R_3)$ holds.
Hence
\[
 v(\xx,(1+(i+1)\eps)\yy) = \Pr(\Vb(R_1\cup R_3)) \ge \alpha_{i+1},
\]
proving the induction step and so
completing the proof of \eqref{cl1}.
\end{proof}

The next step is the analogue of Lemma~\ref{lJ}; this concerns
crossings of nearby squares (now rectangles) that are joined.
As in the previous lemma,
we work with rectangles, not squares, and add the assumptions that
$h(\xx,\yy)$ and $v(\xx,\yy)$ are at least some
constant $c_1$ rather than appeal to Lemma~\ref{Z1/2}.
The definition of the event $J$ is as before, \emph{mutatis mutandis}: specifically,
wherever $n$ appears in the definition, it is replaced by $\xx$ or $\yy$ 
depending on whether we are considering an $x$-coordinate or a $y$-coordinate.

\begin{lemma}\label{lJ2}
Let $\cC$ be a non-degenerate independent lattice colouring.
For any $\eps>0$, $c_1>0$ and $c'>0$, there is a $c>0$ such that
for any $\xx,\yy\ge 100d_0$ with $h(\xx,\yy)$, $v(\xx,\yy)\ge c_1$,
if there exist $\xx$-by-$\yy$ rectangles $R_1$ and $R_2$, with $R_2$
obtained by translating $R_1$ upwards by a distance of $\eps\yy$,
for which $\Pr(J(R_1,R_2))\ge c'$, then $v(\xx,100\yy)\ge c$.
\end{lemma}
\begin{proof}
Modify the proof of Lemma~\ref{lJ} as above: replace each application of Harris's Lemma
by an appeal to \eqref{lH}, and use the fact that, under our assumptions,
any $\xx$-by-$\yy$ rectangle $R$ has $h(R)$, $v(R)\ge \psi(c_1)>0$
by Corollary~\ref{c_trans_lower}.
\end{proof}

Next comes the analogue of Lemma~\ref{smallgap}. For two 
$\xx$-by-$\yy$ rectangles $R_1$ and $R_2$, with $R_2$ obtained by translating
$R_1$ upwards by a distance of $\eps\yy/10$, define the strip $\strip=\strip(R_1,R_2)$
to be the region between the vertical lines containing
the vertical sides of $R_1$ and $R_2$, and let $G_\eps(R_1,R_2)$ be the event
that $\Hb(R_1)$ and $\Hb(R_2)$ hold, the path $P_1=LH(R_1)$ is below $P_2=UH(R_2)$ in $\strip$,
and the area of $\strip$ between $P_1$ and $P_2$ is at most $\eps\xx\yy$.

\begin{lemma}\label{smallgap2}
Let $\cC$ be a non-degenerate independent lattice colouring.
For any $\eps>0$ and $c_1>0$ there are constants $c_3>0$ and $c>0$ such that
for any $\xx,\yy\ge 100d_0$ with $h(\xx,\yy)$, $v(\xx,\yy)\ge c_1$,
either there exist $\xx$-by-$\yy$ rectangles $R_1$ and $R_2$
as above with $\Pr(G_\eps(R_1,R_2))\ge c_3$, or $v(\xx,100\yy)>c$.
\end{lemma}
\begin{proof}
Imitate the proof of Lemma~\ref{smallgap}, \emph{mutatis mutandis}.
\end{proof}

In adapting the main part of the proof of Theorem~\ref{th_RSWZ} there is only
one genuine additional complication: we shall have to work
to join up paths in our two configurations $\omega_1$ and $\omega_2^X$. The other
changes are mostly in notation.

Recall that our independent lattice colouring $\cC$ is obtained
by randomly colouring the subdivided grey faces of the 3-coloured
planar map $\cH$.
As before, let $\Omega$ denote the set of
all configurations, i.e., assignments of states to the grey faces of $\cH$.
Let $\ALG$ be an algorithm that
examines the states of grey faces in its input configuration one by
one, with the next face to be examined determined by the states of the
faces examined so far. Assume that $\ALG$ terminates and write ${\cal S}_{\ALG}(\omega)$
for the set of grey faces examined by $\ALG$ when run on the configuration $\omega$.
Recalling that $\cH$ is invariant under translations
through elements of the lattice $\cL$,
define a function $f_\ALG$ from $\Omega^2\times\cL$ to $\Omega$
by $(\omega_1,\omega_2,\nrl)\mapsto \omega$, where $\omega=f_\ALG(\omega_1,\omega_2,\nrl)$
is the configuration given by
\begin{equation}\label{fA2}
 \omega(g)= f_\ALG(\omega_1,\omega_2,\nrl)(g) = \left\{
   \begin{array}{ll}
    \omega_1(g) & \hbox{if }g\in {\cal S}_\ALG(\omega_1), \\
    \omega_2(g-\nrl) & \hbox{if }g\notin {\cal S}_\ALG(\omega_1). \\
   \end{array}
  \right.
\end{equation}
Here $g$ denotes an arbitrary grey face, and $g-\nrl$ the grey face obtained
by translating $g$ through the vector $-\nrl$.
In other words, as before, the state of a grey face in $\omega=f_\ALG(\omega_1,\omega_2,\nrl)$
is given either by its state in $\omega_1$ or by its state in the translate $\omega_2^\nrl$
of $\omega_2$, according
to whether or not the algorithm $\ALG$ examines $g$ when run on the configuration $\omega_1$.

Let $\cLl$ be the set of points of $\cL$ in $[-5m,5m]\times [-5n,5n]$,
and interpret $\Omega^2\times \cLl$ as the product probability space
in which the two configurations have the distribution associated to $\cC$ and 
are independent, and the random vector $\rX\in\cLl$ is chosen uniformly from $\cLl$.
As before, $\omega=f_\ALG(\omega_1,\omega_2,\rX)$ has the distribution
appropriate for $\cC$.

\begin{proof}[Proof of Theorem~\ref{th_RSW1}]
We follow the proof of Theorem~\ref{th_RSWZ}, concentrating on the differences.
In the light of Remark~\ref{Rsym}, rotating the rectangles
under consideration together with $\cC$, we may assume
that $R$ and $R'$ are aligned with the coordinate axes, i.e.,
are $0$-aligned.

First, note that \eqref{consts1} now holds by assumption.
As before, choose $\gamma>0$ such that \eqref{consts2} holds,
and then choose $\eps>0$
such that $\eps<\gamma$ and \eqref{e2} holds.
Let $\xx$ and $\yy$ satisfy the assumptions of the theorem.
Then, since $h(\xx,\yy)$, $v(\xx,\yy)\ge c_1$,
we can apply Lemma~\ref{smallgap2}. Let $c_3$ and $c$ be the constants given by Lemma~\ref{smallgap2},
so either $v(\xx,100\yy)\ge c$, in which case we are done,
or there are $\xx$-by-$\yy$ $0$-aligned rectangles
$S_1$ and $S_2$ with $S_2$ obtained by translating $S_1$ upwards by 
a distance of $\eps \yy/10$ such that
\begin{equation*}
 \Pr(G_\eps(S_1,S_2))\ge c_3.
\end{equation*}
We may assume that the second case holds.
Translating the rectangles
under consideration together with $\cC$, we may assume
that
$S_1=[0,\xx]\times[0,\yy]$ and $S_2=[0,\xx]\times[\eps\yy/10,(1+\eps/10)\yy]$.

As before, we explore $S_1$ from below to find its lowest black horizontal crossing
$P_1$, if it exists, and $S_2$ from above to find its highest black horizontal crossing
$P_2$. More precisely, we let $I_1$ be the interface $I_0^-$ in $S_1$ described
earlier in the subsection, and $I_2$ the interface $I_0^+$ in $S_2$.
\begin{figure}[htb]
 \centering
 \input{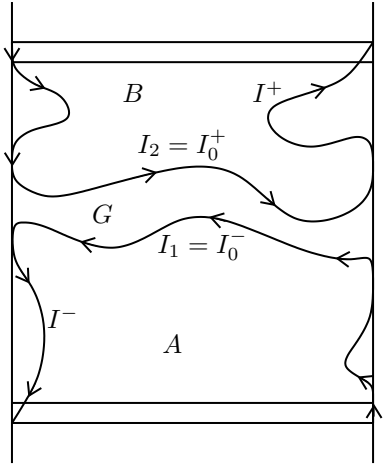}
 \caption{The overlapping congruent rectangles $S_1$ (below) and $S_2$ (above),
with the interfaces $I^-$ in $S_1$ and $I^+$ in $S_2$ defined as in Figure~\ref{fig_ow}.
$I^-_0$ and $I^+_0$ are the minimal subpaths of $I^-$ and $I^+$ crossing
the strip $S_1\cup S_2$. Note that every point just to the right of $I^\pm$ is black or
outside $S_1\cup S_2$ horizontally; points just to the right of $I_0^\pm$
are always inside $S_1\cup S_2$ and hence black. Points
to the left of $I^\pm$ are white or outside $S_1\cup S_2$ vertically. The `gap' $G$ is the region
between $I_1=I^-_0$ and $I_2=I^+_0$.}\label{fig_interfaces}
\end{figure}
Note that orienting these interfaces as in Figure~\ref{fig_interfaces}, the points just to the
right of each interface are black while those to the left
are either white, or inside $T$ but outside the relevant rectangle $S_i$ vertically, if the interface
runs along the top/bottom of $S_i$. It follows that $I_1$ and $I_2$ cannot meet.
Our definitions allow us to view $I_i$ as
a black path itself, although it is perhaps clearer to think of $P_i$ as running
next to $I_i$.
In the present context, the \emph{gap} $G$ is the region between $I_1$ and $I_2$;
we define $A$ and $B$ to be the regions of our strip $[0,\xx]\times\RR$
below $I_1$ and above $I_2$, respectively.

As before, we first test whether $\omega_1\in G_\eps$, using the algorithm $\ALG$
implicitly defined above. More precisely, we implement $\ALG$ by following
the full interfaces $I^-$ and $I^+$ in the relevant rectangles. Note that this only involves
`testing' the state of grey faces $f$ that meet one of these interfaces, where
$f$ \emph{meets} $I$ if $I$ passes through $f$, or along one of the sides of $f$.
It is easy to check that all of $I^-$ lies below $I_0^-=I_1$,
and all of $I^+$ above $I_2$.
Let $Z(I_i)$ denote the union of $I_i$ and all grey faces that meet $I_i$,
which we think of as the  \emph{zone of influence} of $I_i$.
Then when $G_\eps$ holds, the algorithm $\ALG$ establishes
this by looking only at grey faces in $A\cup B\cup Z(I_1)\cup Z(I_2)$.
Defining $\omega=f_\ALG(\omega_1,\omega_2,\rX)$ as before,
it follows that any grey face contained in $G\setminus (Z(I_1)\cup Z(I_2))$
has its state in $\omega$ given by its state in $\omega_2^\rX$,
the configuration $\omega_2$ translated through the random vector $\rX$.

As before, we condition on $\omega_1$, assuming that $G_\eps$ holds, which it does
with probability bounded away from zero. Furthermore, we condition on $\omega_2$, assuming
the existence of a path $P$ with the property described in the definition of
$\Ed(R',\xx,\yy,\alpha,\eta)$. Again, the probability of this event is bounded away from zero,
this time by assumption. As before, the only remaining randomness is
in the choice of the random translation $\rX$.

Recalling that we may take $m$ and $n$ large, apart from one very minor
technical issue that we postpone to the end of the proof, trivial modifications
to our previous arguments show that with probability
bounded away from zero, the random translate $P+\rX$ of $P$ crosses
from $A$ to $B$, while remaining within the strip $\strip$. In fact, adjusting the constants
slightly if necessary, we can assume that it does not come within distance $d_0$ of
the edges of $\strip$.

Let $E\subset \Omega^2\times\cLl$ denote the set of triples $(\omega_1,\omega_2,\nrl)$
such that $\omega_1\in G_\eps$, and $\omega_2$ contains a path $P$ as above
whose translate $P+\nrl$ meets the interfaces $I_1$ and $I_2$ in $\omega_1$.
We have shown that $\Pr(E)$ is bounded away from zero.
Unfortunately, unlike in the $\Z^2$ setting,
it is not true that if $(\omega_1,\omega_2,\nrl)\in E$,
then $f_\ALG(\omega_1,\omega_2,\nrl)\in J$; the problem is illustrated in Figure~\ref{fig_problem}.
\begin{figure}[htb]
 \centering
 \input{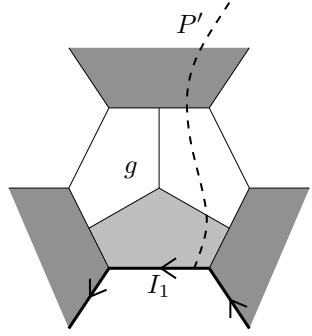}
 \caption{A hexagonal grey face $g$ surrounded by black and white faces of $\cH$. In
the random colouring associated to $\omega_1$, the internal colouring of $g$ is
indicated by the lighter shading.
In this case $I_1$ touches $g$. The dashed line shows a black path $P'$ in the
configuration $\omega_2^\rX$; since the state of $g$ is read from $\omega_1$ (even
though $g$ is in the `gap'), in the combined configuration $\omega$,
the path $P'$ fails to join up with $I_1$.}\label{fig_problem}
\end{figure}

To overcome this problem, we adjust the colourings of at most two faces.
Let $(\omega_1,\omega_2,\nrl)\in E$, and
let $P'$ denote a minimal part of the path $P+\nrl$ joining $I_1$ to $I_2$.
Then $P'$ necessarily lies in the gap $G$. Let $P''$ denote a minimal
subpath of $P'$ joining $Z(I_1)$ to $Z(I_2$), and let $v_1$ and $v_2$ denote
its endpoints (which may coincide, if $Z(I_1)$ and $Z(I_2)$ meet);
see Figure~\ref{fig_gap}.
Note that any interior point of $P''$ lies in $G\setminus (Z(I_1)\cup Z(I_2))$,
so its colour in $\omega=f_\ALG(\omega_1,\omega_2,\nrl)$ is its colour in $\omega_2^\nrl$;
since $P''\subset P+\nrl$, such points are therefore black.
\begin{figure}[htb]
 \centering
 \input{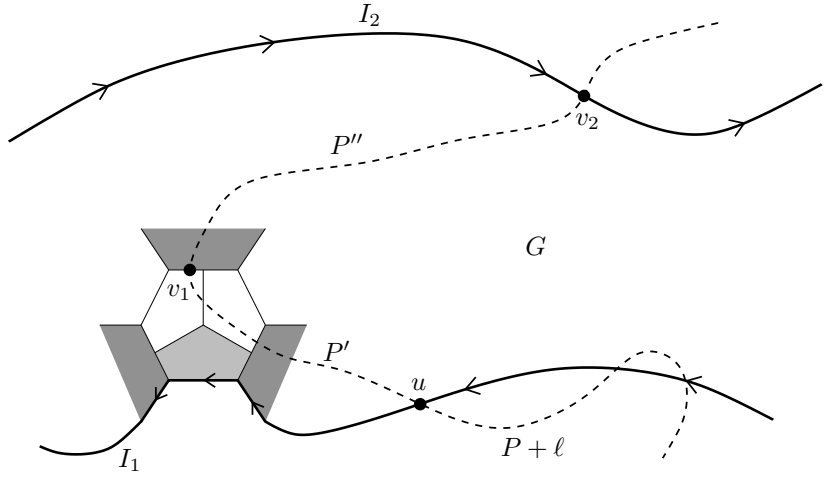}
 \caption{The path $P+\ell$ (dashed line) crossing the gap $G$, from below $I_1$
to above $I_2$. The subpath $P'$ of $P+\ell$, here from $u$ to $v_2$, is minimal subject
to crossing $G$, and so lies within $G$. In turn, $P''$, from $v_1$ to $v_2$,
is minimal subject to joining $Z(I_1)$ to $Z(I_2)$. In this example, $v_2$ is on $I_2$,
while $v_1$ is in a grey face $f_1$ (the hexagon in the figure) touching $I_1$;
the two white subfaces inside $f_1$ will be recoloured black.}\label{fig_gap}
\end{figure}

Now each $v_i$ is either on $I_i$, or belongs to a grey face $f_i$ which meets $I_i$.
Suppose the latter case holds for $i=1,2$. To handle the remaining cases
we simply recolour at most one face in the following argument, rather than two.
Recalling that $P+\nrl$ does not pass
within distance $d_0$ of the edges of $T$, 
note that each $f_i$ lies entirely with the strip $T$.
Let $\omega_1'$ be the configuration obtained from $\omega_1$ by recolouring
all points of $(f_1\cup f_2)\cap G$ black. As we shall show in a moment,
this is a legal configuration. Assuming this for the moment,
if we restrict our attention overall to a finite region of the plane, as we may,
the ratio $\Pr(\omega_1')/\Pr(\omega_1)$ of the probabilities of the individual
configurations is bounded below by $p_0^2$, where $p_0>0$ is the minimum probability
of any possible state of a grey face.

The key point is that recolouring a set of points within the gap $G$
black does not change the interfaces $I_1$ and $I_2$; these
interfaces are determined by the full interfaces
$I^\pm$ shown in Figure~\ref{fig_interfaces}, which are defined `locally',
and the side of $I_i$ on which the gap lies
is already `locally black'. Thus the algorithm $\ALG$ examines the states of the same
set of faces if run on $\omega_1'$ or on $\omega_1$.
Recalling that the states of $f_1$ and $f_2$ in $\omega$ are those in $\omega_1$,
it follows that
in the configuration $\omega'=f_\ALG(\omega_1',\omega_2,\nrl)$ there is a black path
joining $I_1$ to $I_2$, given by the union of $P''$ and two short
paths from $v_i$ to $I_i$ within $f_i\cap G$. Hence $\omega'\in J$.

Let $g:E\to\Omega^2\times\cLl$ denote the map $(\omega_1,\omega_2,\nrl)
\mapsto (\omega_1',\omega_2,\nrl)$, and let $E'=g(E)$. We have
shown that $E'\subset f_\ALG^{-1}(J)$. Recalling that $f_\ALG$ is measure preserving,
it follows that $\Pr(J)\ge \Pr(E')$.
Since our recolouring does not change
the interfaces $I_1$ and $I_2$,
given some $(\tilde\omega,\omega_2,\nrl)\in E'$ known to be the image
of some unknown $(\omega_1,\omega_2,\nrl)$ under $g$,
we can read off the interfaces $I_1$ and $I_2$ (defined in $\omega_1$) by looking
at $\tilde\omega$. We also know the path $P$ from $\omega_2$.
This allows us to determine $P'$ and $P''$ as defined above, and hence
$f_1$ and $f_2$. In other words, we know which two (or at most two) faces
were recoloured, though not how.
It follows that $g^{-1}(\{(\tilde\omega,\omega_2,\nrl)\})$ consists of a bounded
number of configurations, each of whose probabilities is at most $p_0^{-2}$
times that of $(\tilde\omega,\omega_2,\nrl)$. Hence
there is a constant $C$ such that $\Pr(g^{-1}(E')) \le C \Pr(E')$.
Since $g^{-1}(E')=g^{-1}(g(E))\supset E$, it follows that $\Pr(E')\ge \Pr(E)/C$.
Since $E$ is known to have probability bounded away from zero, the result follows.

\medskip
It remains to establish that the recolouring is permissible, i.e., to show that
if we recolour within a grey face $g$ (either $f_1$ or $f_2$ above),
the new colouring has positive probability; this is where
we use the assumption that $\cC$ is malleable (a condition
of Theorem~\ref{th_RSW1}). 
Now $G$ is bounded by interfaces running between black and white regions, as well
as the sides of the strip $\strip$; the latter do not meet $g$.
It follows that $g\cap G$ is the union of one or more colour components (maximal
connected monochromatic subsets) of $g$. Our recolouring thus recolours
one or more white components within $g$ to black; the definition of malleability
ensures that the resulting colouring has positive probability.

Finally, let us comment on the technical issue we overlooked, which is
that since the endpoints of $P$ may not differ by a lattice element, we cannot
exactly join up translates of $P$ though lattice elements
to form $P^*$. One way to handle this is to find
a short (length $O(1)$) black path $P'$ in $\omega_2$ joining appropriate
points within a fundamental domain of $\cL$, and form $P^*$ by chaining together
alternate copies of $P$ and $P'$, as described in Remark~\ref{Rchain}
at the end of Subsection~\ref{ss_z2}.
Alternatively, simply leave small (length at most $d_0=O(1)$) gaps between the translates of $P$
making up $P^*$: when colouring $R_0$, assign colour 0 to any point $v$
such that in the copy of $P^*$ starting at $v$, one or more of these small gaps meets $G$.
It remains the case that few points receive colour 0: the set of such
points is contained in the
union of $I=O(1)$ translates of the $d_0$-neighbourhood $G^{d_0}$ of $G$. Since $G$
is made up of faces whose size is bounded below, the area of $G^{d_0}$ is at most
a constant times that of $G$, so the area receiving colour 0 is still $O(\eps n^2)$;
choosing $\eps$ small enough, the rest of the argument is unchanged.
\end{proof}

\subsection{A stronger rectangle-crossing lemma}

Although technically we can do without it, we now present a more convenient
version of Theorem~\ref{th_RSW1}
giving the same conclusion under a weaker assumption. First we need a simple geometric lemma.

By the \emph{displacement} of a path $P$ we mean the Euclidean distance
from its start to its endpoint. By the \emph{direction} of $P$
we mean the direction from its start to its endpoint, considered
as an angle modulo $\pi$. The \emph{angle} between two paths
is simply the (unsigned) angle between their directions, taken
as a real number between $0$ and $\pi/2$.

\begin{lemma}\label{lchain}
Given $\delta>0$ there exists a constant $C=C(\delta)$ such that whenever
$P_1$ and $P_2$ are two paths in $[0,1/10]^2$ with displacement at
least $\delta$ such that the angle between $P_1$ and $P_2$ is at least $\delta$,
then we can chain together at most $C$ paths each of which is a translate
of $P_1$ or $P_2$ to form a path $P^*$ that lies within $[0.1,0.9]\times [-0.5,1.5]$,
starts below $y=-0.1$, and ends above $y=1.1$.
\end{lemma}
\begin{proof}
Let $v_i$ be the vector from the start of $P_i$ to its end,
and let $\cL$ be the lattice $\{av_1+bv_2:a,b\in \Z\}$.
Let $D$ be the fundamental domain of $\cL$ whose corners are the origin, $v_1$, $v_2$
and $v_1+v_2$. Note that the area of $D$ is bounded below
by a constant depending on $\delta$.

The idea is simply to approximate the line-segment $L$ joining the points
$(0.5,-0.3)$ and $(0.5,1.3)$ by a path $Q$ in the graph associated to $\cL$
where two lattice points are neighbours if they differ by $\pm v_i$
for some $i$; then replace each edge by an appropriate translate of $P_i$.
We can find such a path $Q$ so that every point of $Q$
is in the same domain $D+av_1+bv_2$ as some point of $L$,
and $Q$ starts and ends in the same domains
as the start and end of $L$; to see this, simply apply an affine transformation
mapping $\cL$ to the usual square grid, and approximate the image
of $L$ by an appropriate path.

Since any two points of $D$ have $x$-coordinates that differ by at most $2/10$
and $y$-coordinates that differ by at most $2/10$, the path $P^*$ obtained
from $Q$ has the required properties.
\end{proof}

\begin{theorem}\label{th_RSW2}
Let $\cC$ be a malleable
independent lattice colouring, and let $\delta$, $c_1$ and $c_2$ be positive constants.
Then there exists a constant $c>0$ such that the following holds.
Suppose that $R$ is a large rectangle with \ww\ $\xx$, \hh\ $\yy$
and any orientation, and $R'$ is a large rectangle
with the same orientation, \ww\ $\xx/10$, and \hh\ $\yy/10$.
Let $E$ be the event that $R'$ contains two black paths such that their
images under the affine transformation mapping $R'$ to $[0,1/10]^2$ satisfy
the conditions of Lemma~\ref{lchain}.
If $\Pr(\Hb(R))\ge c_1$, $\Pr(\Vb(R))\ge c_1$, and $\Pr(E)\ge c_2$,
then $\Pr(\Vb(R''))\ge c$
for any rectangle $R''$ with the same orientation as $R$, \ww\ $\xx$ and \hh\ $100\yy$.
\end{theorem}

\begin{proof}
The proof is the same as that of Theorem~\ref{th_RSW1}, except
that we construct the `virtual' path $P^*$ as the union of translates
of our black paths $P_1$ and $P_2$ in $\omega_2$ given by Lemma~\ref{lchain}.
As before, since we can only translate by elements of the lattice $\cL$
of symmetries of $\cC$, we may need additional short paths
to join up these translates. Each of these additional paths lies within
a fundamental domain of $\cL$, and we may take them to be translates
of appropriate short black paths $P_3,\ldots$ in $\omega_2$. (With extremely
high probability, such paths will exist; we only need one of the
many domains to have the property that every grey face meeting it
is coloured black.)

Arguing as before, after conditioning on $\omega_1$ and $\omega_2$,
we find that with probability bounded away from zero our random
translate of one of the $P_i$ joins $I_1$ to $I_2$, and the rest
of the proof is as before.
\end{proof}

The precise constants appearing in Lemma~\ref{lchain} are not important;
the key point is that we can chain together translates of our
paths $P_i$ to give a path $P^*$ with the properties discussed
in Remark~\ref{Rchain}. After appropriate rescaling, these properties
are that $P^*$ crosses a square from bottom to top, with some `elbow room', i.e.,
starting well below the bottom and ending well above the top, without
coming too close to the vertical sides.

\section{Self-duality and rectangle crossings}\label{sec_RSWapp}

Our aim in this section is to show that if $\cC$ is a malleable independent
lattice colouring associated to a self-dual plane hyperlattice $\cH$,
then an analogue of Theorem~\ref{th_RSWZ} holds for $\cC$.
It turns out that, due to the lack of symmetry, we cannot specify
in advance the orientation of the rectangles we work with. In fact,
we cannot even fix their aspect ratio.

Recall that $h_\theta(m,n)$ and $v_\theta(m,n)$ denote $\Pr(\Hb(R))$ and $\Pr(\Vb(R))$,
where $R$ is an $\xx$-by-$\yy$ $\theta$-aligned rectangle centred on the origin.

\begin{definition}\label{lrrdef}
Let $\rho>1$. We say that a lattice colouring $\cC$ has the `large rectangles'
property $LR_{\rho}$
if there exists a constant $c>0$ such that for all $L$
and for all large enough $A$ (depending on $L$)
there are $m,n\ge L$ with $mn=A$ and an angle $\theta$ such that
$h_\theta(\rho m,n)\ge c$ and $v_\theta(m,\rho n)\ge c$.
\end{definition}
In symbols, this large rectangles property with parameter $\rho$ may be stated as follows:
\begin{multline*}
 \exists c>0\ \forall L\ \exists A_0\ \forall A\ge A_0\ \exists m,n,\theta:\\
 m,n\ge L, mn=A, h_\theta(\rho m,n)\ge c\hbox{ and }v_\theta(m,\rho n)\ge c.
\end{multline*}
Roughly speaking, the idea is that we can find a rectangle $R$ so that when we extend it
by a factor of $\rho$ horizontally or vertically, the probability of a black horizontal or
vertical crossing, respectively, is not too small. We can take this rectangle to
have any given large area, and can assume that both sides are at least any given length.
However, the orientation and aspect ratio cannot be specified in advance.

Recall that $\cC$ is non-degenerate if, within each grey face, the all-black
and all-white colourings have positive probability. Under this assumption,
one can adapt the usual argument from Harris's Lemma to
show that for any $\rho_1$, $\rho_2>1$, the property
$LR_{\rho_1}$ implies $LR_{\rho_2}$; the argument is as for Lemma~\ref{Cextend}.
In the light of this, the following definition makes sense.

\begin{definition}\label{lrdef}
A non-degenerate
independent lattice colouring $\cC$ has the \emph{large rectangles property}
if it has the property $LR_{\rho}$ for some $\rho>1$, and hence for all $\rho>1$.
\end{definition}

Recall that $\cC$ is \emph{malleable} if it is non-degenerate and satisfies
certain technical extra conditions; for the full definition
see Subsection~\ref{ss_RSWh}. Our aim in this section is to prove the following result;
as we shall see in the next section, it is then easy to deduce Theorems~\ref{th1}
and~\ref{th2}.

\begin{theorem}\label{rswc}
Let $\cC$ be a malleable independent lattice colouring realizing an (approximately)
self-dual hyperlattice percolation model $\cH(\vecp)$.
Then $\cC$ has the large rectangles property.
\end{theorem}

The proof of Theorem~\ref{rswc} will require a little preparation. First,
let us restate the property slightly. Given an ellipse $E$ with centre $x_0$,
let $2E/3$ denote the ellipse $\{x_0+2(x-x_0)/3:x\in E\}$ obtained by shrinking
$E$ by a factor of 3/2, keeping the centre the same.
Let $E^0$ denote the `annulus' between $E$ and $2E/3$,
and let $\Sb(E)$ denote the event that there is a closed black path in $E^0$ surrounding
the central hole.

\begin{definition}\label{ledef}
A lattice colouring $\cC$ has the \emph{large ellipses property}
if there exists a constant $c>0$ such that for all $L$
and for all large enough $A$ (depending on $L$)
there is an ellipse $E$ centred on the origin with area $A$ and
with both axes having length at least $L$
such that $\Pr(\Sb(E))\ge c$.
\end{definition}

It is easy to see that the large rectangles property and the large ellipses property are
equivalent.
\begin{lemma}\label{r_e}
A non-degenerate independent lattice colouring $\cC$ has the large rectangles property
if and only if it has the large ellipses property.
\end{lemma}
\begin{figure}[htb]
 \centering
 \input{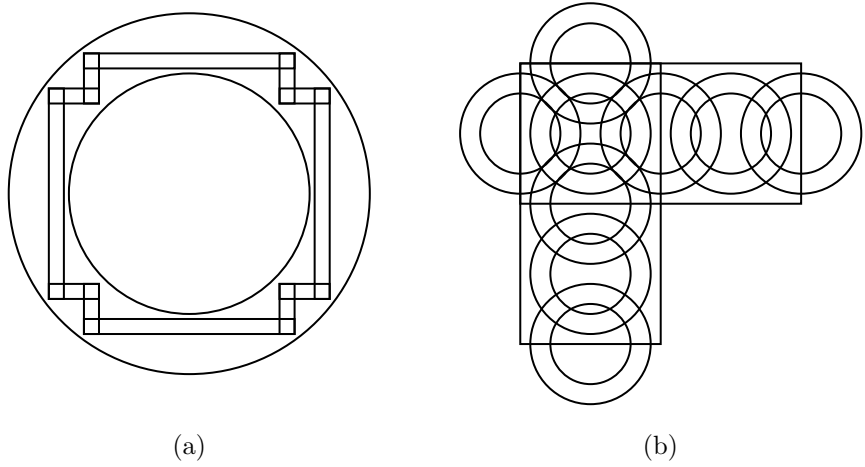}
 \caption{(a) Rectangles in an annulus, and (b) circles crossing two rectangles.}\label{fig_rinc}
\end{figure}

\begin{proof}
Suppose first that $\cC$ has the large rectangles property; in particular,
it has the property $LR_{20}$.

Consider the arrangement of overlapping rectangles
shown in Figure~\ref{fig_rinc}(a). Taking the shorter side
of each rectangle to have length $1$, the arrangement is such that the longer side
of each rectangle has length at most $20$. Also, if each rectangle
has a black crossing in the long direction, then $\Sb(C)$
holds, where $C$ is the outer circle. Let $a$
denote the area of $C$, so $a$ is an absolute constant.

We must show the existence of an ellipse $E$ with (large enough) area $A$
and both axes at least $L$ such that $\Pr(\Sb(E))\ge c_1$, for some
constant $c_1$.  Let $d_0=d_0(\cC)$ be the constant
in Lemma~\ref{reg}. The large rectangles property $LR_{20}$
gives us a $\theta$ and $m,n\ge \max\{L,100d_0\}$ with $mn=A/a$ such
that $h_{\theta}(20m,n)\ge c$ and $v_{\theta}(m,20n)\ge c$,
for some constant $c$ that is independent of $A$.
Consider the image of Figure~\ref{fig_rinc}(a) under a linear transformation
mapping each 1-by-1 square to a $\theta$-aligned $m$-by-$n$ rectangle.
Note that the resulting ellipse $E$ (the image of $C$) has both axes at least $L$ and
has area $A$. Using Corollary~\ref{c_trans_lower} to allow for translation,
each image rectangle has a black crossing in the relevant
direction with probability at least some constant $c'=\psi(c)>0$.
Using Lemma~\ref{ourHarris2} (in the form~\eqref{lH}), it follows that $\Pr(\Sb(E))$ is
bounded away from zero.

For the reverse implication, we assume the large ellipses
property and deduce the property $LR_2$. As shown
in Figure~\ref{fig_rinc}(b), one can arrange 9 circles $C_i$ of radius
$0.9$ to `cross' a $2$-by-$1$ rectangle $R_1$ and a $1$-by-$2$ rectangle
$R_2$ in such a way that if $\Sb(C_i)$ holds
for each $i$, then $\Hb(R_1)$ and $\Vb(R_2)$ hold.
Moreover, this remains true if each circle is translated by a small
distance (at most $0.01$, say). Given an ellipse $E$
with both axes at least $\max\{100d_0,L\}$ such that $\Pr(\Sb(E))\ge c$,
one can rotate the arrangement in Figure~\ref{fig_rinc}(b) and
then scale it along the directions of the axes
of $E$ so that each circle is mapped to a translate $E_i$
of $E$, and $R_1$ and $R_2$ are mapped to $\theta$-aligned
$2m$-by-$n$ and $m$-by-$2n$ rectangles $R_1'$ and $R_2'$, for some $\theta$, $m$ and $n$.
Note that $m$, $n\ge L$, and $mn$ is an absolute constant times the area of $E$.

Moving each $E_i$ by a distance of at most $d_0$ (which
corresponds to translating $C_i$ through a distance of
at most $0.01$), we may assume that $E_i=E+\ell_i$ for some $\ell_i\in \cL$.
Then $\Pr(\Sb(E_i))=\Pr(\Sb(E))\ge c$, and if all 9
upsets $\Sb(E_i)$ hold, then $\Hb(R_1')$ and $\Vb(R_2')$ hold.
Using Lemma~\ref{ourHarris2} thus
gives a constant lower bound on $h_{\theta}(2m,n)$ and $v_{\theta}(m,2n)$,
establishing the property $LR_2$.
\end{proof}

Most of the time, we work with the large rectangles
property; the large ellipses property will be convenient to use in 
Section~\ref{sec_crit}. The equivalence is also useful in that
it leads to a quick proof that the large rectangles property is affine-invariant.

\begin{lemma}\label{raf}
Let $T$ be an invertible linear map from $\RR^2$ to $\RR^2$ and
let $\cC$ be a non-degenerate independent lattice colouring.
If $\cC$ has the large rectangles property, then so does $T(\cC)$.
\end{lemma}
\begin{proof}
There is a constant $a=a(T)>0$ such that if $E$ is an ellipse
both of whose axes have length at least $L$ then $T(E)$ is an ellipse
both of whose axes have length at least $aL$. (For example,
note that $E$ contains a circle of radius $L$, and the image
of this circle contains a circle of radius $aL$ for some $a>0$.)
Hence the large ellipses property is invariant under $T$. Applying
Lemma~\ref{r_e} twice, we see that the large rectangles property is too.
\end{proof}

Our next lemma shows that after a linear transformation, 
a self-dual hyperlattice percolation model must be
related to its dual by one of a small number of linear transformations.
We consider both the self-dual case (for Theorem~\ref{th1}) and the
approximately self-dual case (for Theorem~\ref{th2}).

\begin{lemma}\label{lsymclass}
Let $\cH(\vecp)$ be an (approximately) self-dual hyperlattice percolation model.
Then there is a map $S:\RR^2\to\RR^2$
giving an isomorphism between $\cH(\vecp)$ and a hyperlattice
model equivalent to $\cH^*(\vecp^*)$
such that $S$ can be written in the form $S(x)=T(x)+\Delta(x)$,
where $T$ is linear and $|\Delta(x)|$ is bounded. Furthermore,
after a linear change of coordinates (if necessary),
we may assume that $T$ is either reflection in some line,
or rotation through one of the angles $0$, $\pi/2$, or $\pi$.
\end{lemma}
\begin{proof}
Suppose first that $\cH(\vecp)$ is self-dual.
By the definition of self-duality for hyperlattices, there is a homeomorphism
$S$ from $\RR^2$ to $\RR^2$ mapping $\cH$ to its dual, and preserving
the lattice structure. More precisely, there is a linear map
$T$ with $T(\cL)=\cL$ such that $S(x+\ell)=S(x)+T(\ell)$
whenever $\ell\in \cL$. Since $\Delta=S-T$ is continuous
and doubly periodic, it is bounded, giving the first statement.

Since $T(\cL)=\cL$, the map
$T$ preserves area, so
$T\in GL_2(\RR^2)$ with $\det(T)=\pm 1$.
From standard results,
$T$ is conjugate in $GL_2(\RR^2)$ to a map $T'$ that is either
a rotation, a shear with matrix 
$\left(\begin{smallmatrix}1&\la\\ 0&1\end{smallmatrix}\right)$,
or a stretch with matrix $\left(\begin{smallmatrix}\la&0\\ 0&\pm1/\la\end{smallmatrix}\right)$.
Changing coordinates (or applying a linear transformation to $\cH$ and its
dual simultaneously), we may assume that $T=T'$.

Now $S^2$ maps $\cH$ into itself, and $S^2(x)=T^2(x)+\Delta_2(x)$
where $\Delta_2$ is bounded. 
If $T$ is a shear with $\la\ne 0$ or a stretch with $|\la|\ne 1$,
then repeated application of $S^2$ shows the existence of arbitrarily long edges in $\cH$,
giving a contradiction. Thus $T$ is either a rotation or a reflection.
In the former case, the fact that $T$ maps $\cL$ into itself guarantees
that the angle of rotation $\theta$ is a multiple of either $\pi/3$ or $\pi/2$.
Replacing $S$ by the isomorphism $S^3$ from $\cH(\vecp)$ to its dual allows
us to reduce the cases $\theta=k\pi/3$, $k\in \Z$, to the cases $\theta=0$ or $\theta=\pi$,
and the case $\theta=3\pi/2$ to the case $\theta=\pi/2$.

The argument for the approximately self-dual case is similar, except that
from the definition of approximate self-duality we may simply assume that $T$
is a rotation or a reflection.
\end{proof}

For the rest of the section, in the light of Lemma~\ref{raf},
we assume as we may that $\cH(\vecp)$ and its dual
are related as described
in Lemma~\ref{lsymclass}. In the case where the map $T$ is a reflection,
we take it to be reflection in the $x$-axis, and call this the \emph{reflection} case.
The remaining cases are \emph{rotation} cases; we write $\td$ for the angle
of rotation passing from $\cH$ to $\cHd$, so $\td\in\{0,\pi/2,\pi\}$.

Let $\cC^\dual$ be the \emph{negative}
of the colouring $\cC$, defined simply by
interchanging black and white (both in $\cH$, and in the colours of the
subfaces of the grey faces of $\cH$).
Even in the self-dual case,
although $\cC$ realizes $\cH(\vecp)$, which is self-dual via the map $S:\RR^2\to\RR^2$,
we cannot assume that $\cC$ is self-dual in the natural sense.
For example, in the dual colouring to that shown in the centre of Figure~\ref{fig_shading3},
there are two black subfaces within the hexagon. This pattern may not occur in $\cC$.
However, writing $d_0$ for the constant given by Lemma~\ref{reg}, since $\cC$
realizes $\cH(\vecp)$, there is a natural coupling of $\cC$ and $\cH(\vecp)$
such that any black path $P$ in $\cC$ is within Hausdorff distance $d_0$
of an open path in $\cH(\vecp)$, and vice versa.
Applying this observation also to $\cC^\dual$, which realizes $\cHd(\vecp^\dual) = S(\cH(\vecp))$,
it follows that we can couple $\cC$ and $\cC^\dual$ so that for any black
path $P$ in $\cC$ there is a black path in $\cC^\dual$ within distance $O(1)$ of $S(P)$,
and vice versa.
Recalling that black in $\cC^\dual$ simply means white in $\cC$, and using Lemma~\ref{lsymclass},
this has the following consequence. Here, $\phi$ is the `probability scaling function'
defined before Lemma~\ref{cont}.

\begin{lemma}\label{wb}
Let $\cC$ be a non-degenerate
independent lattice colouring realizing an (approximately) self-dual
hyperlattice percolation model $\cH(\vecp)$, and let $T$ be the linear
map given by Lemma~\ref{lsymclass}.
Then there are constants $L$ and $C$ such that, for any angle $\theta$
and any $\theta$-aligned $m$-by-$n$ rectangle $R$ with $m$, $n\ge L$, we have
\[
 \phi\bb{\Pr(\Hb(R))} - \phi\bb{\Pr(\Hw(T(R)))} \in [-C,C]
\]
\end{lemma}
\begin{proof}
We write out only the self-dual case; since we in any case allow some
`elbow room' when passing from $\cH(\vecp)$ to its dual, there 
are no additional difficulties in the approximately
self-dual case.

Let $L=100d_0+10D$, where $D=\sup\{|S(x)-T(x)|\}$, which is finite
by Lemma~\ref{lsymclass}. Couple $\cC$ and $\cC_1=\cC^\dual$ as above.
Let $\cC_2$ be the colouring obtained from $\cC_1$ by interchanging
white and black, so $\cC_2$ has the same distribution as $\cC$.

Let $R$ be a rectangle as described, and let $R^+$ be obtained by moving
the vertical sides of $R$ outwards by a distance $D$ and the horizontal sides
inwards by the same distance. Suppose $R^+$ has a black horizontal crossing in $\cC$.
Then from the remarks before the lemma, there is a black path in $\cC_1$
close to $S(P)$ and hence to $T(P)$, and thus a white path in $\cC_2$
close to $T(P)$. But any such path crosses $T(R)$.
Hence
\[
 \Pr_{\cC}(\Hw(T(R))) = \Pr_{\cC_2}(\Hw(T(R))) \ge \Pr_{\cC}(\Hb(R^+)).
\]
Writing $\Pr$ for $\Pr_{\cC}$ as usual, and using Lemma~\ref{cont},
this gives $\phi(\Pr(\Hw(T(R))))\ge \phi(\Pr(\Hb(R))) -O(1)$.
The reverse inequality is proved similarly.
\end{proof}

In what follows we assume throughout that our `large' length $L$ is at least
$100d_0$, and is such that Lemma~\ref{wb} applies. Note that rectangles
with dimensions at least $L$ are `large' in the terminology of the previous
sections.

\begin{lemma}\label{2mod}
Let $\cC$ be a non-degenerate
independent lattice colouring realizing an (approximately) self-dual
hyperlattice percolation model $\cH(\vecp)$.
For any $L$
there are constants $A_0=A_0(\cC,L)$ and $C=C(\cC,L)$ such that for any $A\ge A_0$,
and any $\theta$ in the rotation case or
$\theta=0$ in the reflection case,
we may find $\xx$, $\yy\ge L$ with $\xx\yy=A$ such that
$\phi(h_\theta(\xx,\yy))$, $\phi(v_\theta(\xx,\yy))\in [-C,C]$.
\end{lemma}
In other words, both the horizontal and vertical crossing probabilities
for some rectangle of the given area and orientation are
bounded away from 0 and~1.
\begin{proof}
Let us write $h_\theta^*(\xx,\yy)$ and $v_\theta^*(\xx,\yy)$
for the probabilities that an $\xx$-by-$\yy$ $\theta$-aligned
rectangle has a white horizontal crossing or a white vertical
crossing, respectively.

Consider first the rotation case with $\td=\pi/2$.
In this case we simply set $A_0=L^2$, and
choose $\xx=\yy=\sqrt{A}$.
Let $R$ be the $\xx$-by-$\xx$ $\theta$-aligned square centred
on the origin. Note that $T(R)$ is the same rectangle $R$, but
viewed as $(\theta+\pi/2)$-aligned, so $v_\theta^*(\xx,\xx)=\Pr(\Hw(T(R)))$.
Hence Lemma~\ref{wb} gives $\phi(h_\theta(\xx,\xx))-\phi(v_\theta^*(\xx,\xx))=O(1)$.
But by Lemma~\ref{oneway}, we have $h_\theta(\xx,\xx)+v_\theta^*(\xx,\xx)=1$,
so $\phi(v_\theta^*(\xx,\xx))=-\phi(h_\theta(\xx,\xx))$. It follows
that $\phi(h_\theta(\xx,\xx))=O(1)$. Similarly, $\phi(v_\theta(\xx,\xx))=O(1)$.

For the remaining cases, for $A$ large enough,
Lemma~\ref{1mod} gives us $\xx,\yy\ge L$ with $\xx\yy=A$ and
$\phi(h_\theta(\xx,\yy))\in [-4,4]$.
This time (noting that $\theta=0$ in the reflection case), the map $T$
maps our $\theta$-aligned $\xx$-by-$\yy$ rectangle
into itself, and Lemma~\ref{wb} gives
$\phi(h_\theta(\xx,\yy))-\phi(h_\theta^*(\xx,\yy))=O(1)$.
Hence $\phi(v_\theta(\xx,\yy))=-\phi(h_\theta^*(\xx,\yy))=\phi(h_\theta(\xx,\yy))+O(1)=O(1)$.
\end{proof}

We are now ready to prove Theorem~\ref{rswc}.
\begin{proof}[Proof of Theorem~\ref{rswc}]
We assume as we may that $L\ge 100d_0$.
Throughout we fix an arbitrary $A\ge 10^6A_0(\cC,L)$,
where $A_0(\cC,L)$ is the constant in Lemma~\ref{2mod}. All constants $c$ or $c_i$
below will depend only on $\cC$, not on $A$ or $L$. We shall show that
for some $\theta$, $a$ and $b$ with $ab=10A$ we have
\begin{equation}\label{rswcaim}
 h_\theta(10a,b),\, v_\theta(a,10b)\ge c,
\end{equation}
where $c>0$ does not depend on $A$. This will establish
that $\cC$ has the large rectangles property $LR_{10}$.

Fix an orientation $\theta$, with $\theta=0$ in the reflection case, and $\theta$
arbitrary otherwise. Suppressing the dependence on $\theta$ in the notation,
by Lemma~\ref{2mod} there are $\xx$, $\yy\ge L$ with $\xx\yy=A$ such that
\begin{equation}\label{bw}
 h(\xx,\yy),\, v(\xx,\yy)\ge c_0,
\end{equation}
where $c_0>0$ is a constant depending
only on $\cC$.

Suppose for the moment that $h(100\xx,\yy)\ge c_1$, where $c_1$ is a positive
constant not depending on $A$. Set $\xx'=10\xx$ and $\yy'=\yy$, and consider
an $\xx'$-by-$\yy'$ rectangle $R$ and the $\xx'/10$-by-$\yy'$ rectangle $R'$
inside it. Note that $h(R)=h(10\xx,\yy)\ge h(100\xx,\yy)\ge c_1$.
Also, since $R'$ is $\xx$-by-$\yy$, we have $v(R')\ge c_0$.
We may thus apply Theorem~\ref{th_RSW1} with $\xx$ and $\yy$ replaced
by $\xx'$ and $\yy'$, and $\alpha=\beta_1=1/10$, $\beta_2=1$, and $\eta=1/2$.
Any black vertical crossing of $R'$ satisfies the conditions
for the event $\Ed$ considered in Theorem~\ref{th_RSW1} simply because,
seen with respect to the $\xx'$-by-$\yy'$ grid, $R'$ is much taller than wide.
Theorem~\ref{th_RSW1} thus gives $v(\xx',100\yy')\ge c$,
for some constant $c>0$. But then $h(10\xx',\yy')=h(100\xx,\yy)$
and $v(\xx',10\yy')\ge v(\xx',100\yy')$ are both at least $\min\{c_0,c\}$,
establishing \eqref{rswcaim} with $a=\xx'=10\xx$ and $b=\yy'=\yy$.

A similar argument (using a `rotated' version of Theorem~\ref{th_RSW1})
shows that if $v(\xx,100\yy)\ge c_1$, then \eqref{rswcaim} holds with $a=\xx$ and $b=10\yy$.
Thus, in what follows, it suffices to show that one of $h(100\xx,\yy)$
and $v(\xx,100\yy)$ is not too small.

Applying Lemma~\ref{2mod} again we find $\xx_1$, $\yy_1\ge L$ with $\xx_1\yy_1=\xx\yy/400=A/400$
such that
\[
 h(\xx_1,\yy_1),\,v(\xx_1,\yy_1)\ge c_0.
\]
Let $R_1$ be the $\xx_1$-by-$\yy_1$ rectangle centred on the
origin, so $\Pr(\Hb(R_1))\ge c_0$ and $\Pr(\Vb(R_1))\ge c_0$.

Fix a small constant $\delta>0$; for example, $\delta=1/1000$ will do.
Suppose first that $\yy_1/\yy\ge (1+\delta)\xx_1/\xx$, say.
Then on the $\xx$-by-$\yy$ scale, the rectangle $R_1$ is significantly
taller than wide, and has area $1/400$. After truncating $R_1$ vertically
if necessary (after which we still have $\Pr(\Vb(R_1))\ge c_0$),
we may apply Theorem~\ref{th_RSW1} with $R$ an $\xx$-by-$\yy$ rectangle,
with $\beta_1=\xx_1/\xx\le 1/20$, $\alpha=\beta_2=\min\{\yy_1/\yy,1/3\}\ge (1+\delta)\xx_1/\xx$,
and $\eta=\delta/2$.
A black vertical crossing of $R_1$ once again ensures that $\Ed$ holds,
and it follows that $v(\xx,100\yy)\ge c$ for some constant $c>0$. As noted
above, we are done in this case.

A similar argument applies if $\xx_1/\xx\ge (1+\delta)\yy_1/\yy$, so we may assume
that $(1-\delta)/20\le \xx_1/\xx,\,\yy_1/\yy \le (1+\delta)/20$.

As before, let $F$ be the function given by Lemma~\ref{ourHarris2} applied to the
product of partially ordered sets corresponding to the random partitions
induced by each edge $e$ of $\cH$.
Recalling that $h(\xx_1,\yy_1)$, $v(\xx_1,\yy_1)\ge c_0$, set $c_2=F(\psi(c_0),\psi(c_0))$,
where $\psi$ is the function appearing
in Corollary~\ref{c_trans_lower}.
For any rectangle $R$,
let $X(R)$ denote the event $\Hb(R)\cap \Vb(R)$ that $R$ has both horizontal
and vertical black crossings.
For any $\xx_1$-by-$\yy_1$ rectangle $R$, we have $\Pr(\Hb(R))$, $\Pr(\Vb(R))\ge \psi(c_0)$
by Corollary~\ref{c_trans_lower}, and thus
\begin{equation}\label{XR}
 \Pr(X(R))\ge c_2.
\end{equation}
Set $c_3=F(c_2/4,c_2/4) < c_2/4$, and
$c_4=F(c_3,c_3/2)< c_3/2$.  

Let us say that a rectangle $R$ is \emph{useful} if it contains
two black paths $P_1$ and $P_2$ such that, after scaling by dividing
all $x$-coordinates by $\xx$ and all $y$-coordinates by $\yy$, the
paths satisfy the assumptions of Lemma~\ref{lchain}. In other words,
for each $i$, the vector $v_i$ from the start of $P_i$ to the end
has (after rescaling) length at least $\delta$, and the angle
between $v_1$ and $v_2$ is at least $\delta$. We write
$U(R)$ for the event that $R$ is useful.
If there is any rectangle $R$ with width at most $\xx/10$ and height
at most $\yy/10$ for which $\Pr(U(R))\ge c_4$ then, recalling \eqref{bw},
Theorem~\ref{th_RSW2} gives $v(\xx,100\yy)\ge c$ for some $c>0$.
Hence we may assume that
\begin{equation}\label{noU}
 \Pr(U(R)) < c_4 < c_3/2 < c_2/8
\end{equation}
for any rectangle with these dimensions.
Note that $1.5\xx_1\le \xx/10$ and $1.5\yy_1\le \yy/10$, so
the rectangle $R_1$ defined earlier
satisfies the size restriction above with room to spare.

Suppose that some $\xx_1$-by-$\yy_1$ rectangle
$R$ is coloured in such a way that $X(R)\setminus U(R)$ holds.
Then $R$ has a black horizontal crossing $P_1$ and
a black vertical crossing $P_2$. The angle (after scaling) between
these crossings must be very close to 0, so it follows
that either both $P_1$ and $P_2$ cross $R$ from near the top left
to near the bottom right, or both cross from near the bottom left
to near the top right. Suppose the former holds. Then, in addition,
every black horizontal or vertical
crossing stays (after the usual scaling) within distance
$5\delta$ of the diagonal; otherwise, such a crossing can
be split into two parts ($AB$ and $BC$ in the figure)
with an angle of at least $\delta$ between
them, so $U(R)$ holds; see Figure~\ref{fig_U1U2}.
\begin{figure}[htb]
 \centering
 \input{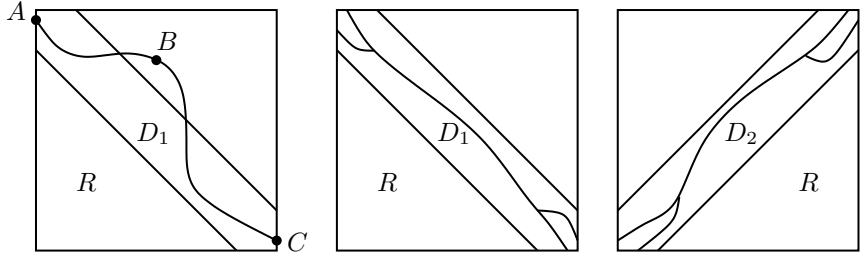}
 \caption{Various possible configurations in an $\xx_1$-by-$\yy_1$ rectangle $R$, rescaled by dividing
$x$-coordinates by $\xx$ and $y$-coordinates by $\yy$. Note that the width and height
of $R$ are then between $(1-\delta)/20$ and $(1+\delta)/20$. 
The strips $D_1$ and $D_2$ have width $10\delta$.
The second and third figures illustrate $X_1(R)$ and $X_2(R)$.}\label{fig_U1U2}
\end{figure}

Let $D_1=D_1(R)$ and $D_2=D_2(R)$ denote the strips of width (after rescaling
as above) $5\delta$ about the two diagonals of $R$.
Let $X_i(R)$ denote the event that $D_i(R)$ contains both horizontal and vertical
crossings of $R$;
then $X(R)\setminus U(R)\subset X_1(R)\cup X_2(R)$.
Hence, for any $\xx_1$-by-$\yy_1$ rectangle
$R$, for some $i$ we have $\Pr(X_i(R))\ge (\Pr(X(R))-\Pr(U(R)))/2 \ge c_2/4$,
using \eqref{XR} and \eqref{noU}. We say that $R$ is \emph{of type $i$}
if $\Pr(X_i(R))\ge c_2/4$, so any $R$ is of type 1 or type 2.
Furthermore, no $R$ can be of both types: otherwise,
using Lemma~\ref{ourHarris2} again, we have
$\Pr(U(R))\ge \Pr(X_1(R)\cap X_2(R))\ge F(c_2/4,c_2/4)=c_3$,
contradicting \eqref{noU}.

At this point we consider two separate cases.

1. \emph{Reflection case.}
Let $R_1$ and $R_2$ be two $\xx_1$-by-$\yy_1$ rectangles with $R_2$
obtained by translating $R_1$ vertically through a distance
$10\delta\yy\le \yy_1/10$.
Recalling that our orientation $\theta$ is the standard
orientation $\theta=0$ in this case, we may choose $R_1$ and $R_2$
so that the rectangle $R_1\cup R_2$ is centred at the origin;
in particular, the $x$-axis is an axis of symmetry of $R_1\cup R_2$.

Suppose first that $R_1$ and $R_2$ are of opposite types, say
with $R_i$ of type $i$.
Then $\Pr(U(R_1\cup R_2))\ge \Pr(X_1(R_1)\cap X_2(R_2))\ge F(c_2/4,c_2/4)=c_3$,
contradicting \eqref{noU}.
It follows that $R_1$ and $R_2$ must be of the same type, say type $1$.

Let $E$ be the event $X_1(R_1)\cap X_1(R_2)\setminus U(R_1\cup R_2)$,
so $\Pr(E)\ge F(c_2/4,c_2/4)-c_3/2=c_3/2$ by \eqref{noU}.  If $E$
holds, then $R_1$ contains a black horizontal crossing $P_1$ in the
strip $D_1=D_1(R_1)$, and $R_2$ contains a black horizontal crossing
$P_2$ in the strip $D_1'=D_1(R_2)$. It is easy to check that
if $P_1$ and $P_2$ are joined by a black path in $R_1\cup R_2$,
then $U(R_1\cup R_2)$ holds. (For example, since we also
have vertical crossings of $R_1$ and $R_2$, we obtain a vertical
crossing of $R_1\cup R_2$, whose direction is necessarily more
than an angle $\delta$ away from that of $P_1$ or $P_2$.)
Hence, whenever $E$ holds, so does the event $F$, that
$D_1\cup D_1'$ contains a \emph{white} horizontal crossing
of $R_1\cup R_2$. Thus, $\Pr(F)\ge c_3/2$.

We now apply self-duality as in the proof of Lemma~\ref{wb}, recalling
that now $T$ is reflection in the $x$-axis, an axis of symmetry of $R_1\cup R_2$.
Let $\tilde F$ be the image of $F$ under the symmetry transformation.
Then $\Pr(\tilde F)=\Pr(F)$. Also, if $\tilde F$ holds then there
is a black path in $R_1\cup R_2$ that almost crosses
$R_1\cup R_2$ horizontally, and lies within or close to the mirror image of $D_1\cup D_1'$.
Since $\tilde F$ is an upset, 
by Lemma~\ref{ourHarris2} we have $\Pr(X_1(R_1)\cap X_2(R_2)\cap \tilde F)
\ge F(c_3,c_3/2)= c_4$. But whenever this event holds, $R_1\cup R_2$
is clearly useful: the relevant crossings must in fact meet, but
we do not even need this, simply the observation that $R_1\cup R_2$
then contains two longish black paths that, after the usual rescaling,
are at almost 90 degrees to each other.
Thus $\Pr(U(R_1\cup R_2))\ge c_4$, contradicting \eqref{noU}.
This contradiction completes the proof in the reflection case.

2. \emph{Rotation case.} In this case, Lemma~\ref{2mod} applies
\emph{regardless of the orientation $\theta$}.
The idea is to observe that the argument above shows that for each
$\theta$, the $\theta$-aligned $\xx_1$-by-$\yy_1$ rectangle $R_1$
centred on the origin is either of type~1 or of type 2.
As we rotate, we can assume (as we will shortly show) that $R_1$ varies
continuously. Thus it should not jump from type 1 to type 2. But after
rotating by 90 degrees, we return to the same rectangle
viewed with a different orientation. 
Changing orientation in this way interchanges types
1 and 2.

To make this precise, first note that the argument above establishes
the following. Let $C$ be the constant
in Lemma~\ref{2mod}, let $\theta$ be any orientation, and let $\xx_1$ and $\yy_1$
satisfy $\xx_1\yy_1=A/400$ and $\xx_1$, $\yy_1\ge L$.
If $R$ is the $\theta$-aligned $\xx_1$-by-$\yy_1$ rectangle centred
on the origin, and $\phi(h_\theta(\xx_1,\yy_1))$, $\phi(v_\theta(\xx_1,\yy_1))\in [-C-8,C+8]$,
then $R$ is either of type 1 or of type 2 (and not both,
although we shall not use this). To see this, simply use
the rectangle $R$ under consideration as $R_1$ in the argument above:
so far we selected an arbitrary rectangle with the properties
described by Lemma~\ref{2mod}; now we choose a specific one
with the same properties, except that we have replaced $C$ by $C+8$, which makes no difference.

Pick an integer $k\ge A/d_0$, and let $\eps=(\pi/2)/k$.
We claim that we can construct a finite sequence $R_0,R_1,\ldots,R_N$
of rectangles, all centred on the origin,
where $R_i$ is $\theta_i$-aligned and satisfies the conditions
above, with $\theta_0=0$, $\theta_N=\pi/2$ and $R_0$ and $R_N$
the same geometric rectangle, such that for $i=0,1,\ldots,N-1$,
either $R_i$ and $R_{i+1}$ have the same dimensions and $\theta_{i+1}=\theta_i+\eps$,
or $R_i$ and $R_{i+1}$ have the same orientation, and their corresponding dimensions
differ by at most one.

To establish the claim, simply start with $\theta_0=0$ and $R_0$ of the dimensions
given by Lemma~\ref{2mod}. At each stage, if $R_i$
satisfies $\phi(h(R_i))$, $\phi(v(R_i))\in [-C,C]$, then rotate it through
an angle $\eps$ to obtain $R_{i+1}$.
Since the largest dimension of $R$ is at most $(A/400)/(100d_0)\le k/10$,
Corollary~\ref{c_rot} ensures that $R_{i+1}$ has the required properties.
Otherwise, by assumption we have $\phi(h(R_i))$, $\phi(v(R_i))\in [-C-8,C+8]$.
Applying Lemma~\ref{2mod}, pick a rectangle $R$ with the same orientation and area for which
$\phi(h(R))$, $\phi(v(R))\in [-C,C]$. Since $R_i$ and $R$ have the same
orientation, we may pass gradually from one to the other with $h(\cdot)$
increasing and $v(\cdot)$ decreasing or vice versa; the intermediate
rectangles thus all satisfy our requirements. Once we reach $R$, we are ready
for the next rotation step.

Since the type of $R_N$ is the opposite of that of $R_0$, there is some $i$
such that $R_i$ and $R_{i+1}$ have different types; suppose that $R_i$
is of type $1$ and $R_{i+1}$ of type 2. Then by Lemma~\ref{ourHarris2}
we have $\Pr(X_1(R_i)\cap X_2(R_{i+1}))\ge c_3$.
Although the dimensions
of our rectangles may change radically as we rotate them, they cannot do so in one
step. In particular, after the rescaling above applied to $R_i$,
the paths witnessing $X_1(R_i)$ and $X_2(R_{i+1})$ are close to orthogonal.
It follows that if $X_1(R_i)\cap X_2(R_{i+1})$ holds, then $R_i$ is useful,
so $\Pr(U(R_i))\ge c_3$, contradicting \eqref{noU} above.
\end{proof}

\section{From rectangle crossings to percolation}\label{sec_deduce}

In this section we shall deduce Theorems~\ref{th1} and~\ref{th2} from Theorem~\ref{rswc};
this turns out to be relatively straightforward, adapting the very
simple argument for bond percolation on $\Z^2$ described
in~\cite{ourKesten}. First, we get one technical detail out of the way.

Recall that we can represent any hyperlattice percolation model $\cH(\vecp)$
by an independent lattice colouring $\cC$; our results
in the previous section apply only to \emph{malleable} $\cC$.
In the bulk of this section we shall prove the following variant
of Theorem~\ref{th1}, differing only in the additional assumption of malleability.
\begin{theorem}\label{th1p}
Let $\cH(\vecp)$ be a malleable (approximately) self-dual hyperlattice percolation model.
Then for any $\vecq\pg\vecp$ the model $\cH(\vecq)$ percolates,
and for any $\vecq\pl\vecp$ the model $\cH(\vecq)$ exhibits exponential decay.
\end{theorem}
Before proving this result, we note that Theorem~\ref{th1} follows.
\begin{proof}[Proof of Theorem~\ref{th1}]
Let $\cH(\vecp)$ be self-dual,
and let $S$ be an isomorphism from $\cH$ to $\cHd$ witnessing
this.
For this $S$, the self-duality condition reduces to a set of equations equating
certain entries $p_{i,\pi}$ of $\vecp$. In particular, any entry $p_{i,\pi}$ where
$\pi$ is a partition into singletons
is equated with some $p_{j,\pi'}$, where $\pi'$ is a partition into one part.
Recall that we call entries of these two types \emph{bottom} and \emph{top} entries,
respectively.

If $\vecq\pg\vecp$, then by definition of our partial order, each bottom entry of $\vecp$
is non-zero. Hence by self-duality each top entry is non-zero. Since corresponding
top or bottom entries of $\vecp$ and $\vecq$ cannot be equal, it follows
that we can adjust $\vecp$ slightly to find some $\vecp'$ all of whose
entries are strictly positive such that $\cH(\vecp')$ is self-dual, with $\vecp'\pl\vecq$.
Since $\vecp'$ is malleable,
Theorem~\ref{th1p} implies that $\cH(\vecq)$ percolates, as required.

The argument that any $\vecq\pl\vecp$ exhibits exponential decay proceeds similarly.
\end{proof}

The argument above shows that in proving Theorem~\ref{th1}, we may impose
the condition of malleability (which we need in the proof) without loss
of generality, so we do not need to assume malleability in the statement
of the theorem. Unfortunately, there does not seem to be an obvious 
analogous argument
in the approximately self-dual case: it is not clear how to adjust
the probabilities slightly while preserving approximate self-duality.
For this reason we simply impose malleability as a condition in Theorem~\ref{th2},
so Theorem~\ref{th2} trivially follows from Theorem~\ref{th1p}.

To deduce Theorem~\ref{th1p} from Theorem~\ref{rswc}, we shall use an analogue
for posets of the well-known Friedgut--Kalai sharp-threshold result
for symmetric events, Theorem 2.1 of~\cite{FK},
which is itself a consequence of a result of Kahn, Kalai and Linial~\cite{KKL}
(see also~\cite{BKKKL}) concerning the influences of coordinates in a product space.
This sharp-threshold result has been applied in many contexts --
it was first used to prove criticality (for random Voronoi percolation)
in~\cite{Voronoi}; we shall use the same technique here (see also~\cite{ourKesten}).

Let $\PP$ be a finite poset. Given two probability measures $\Pr_0$ and $\Pr_1$
on $\PP$, recall that $\Pr_1$ \emph{strictly dominates} $\Pr_0$, written
$\Pr_1\succ\Pr_0$, if $\Pr_1(\cU)>\Pr_0(\cU)$ for every upset $\cU\subset \PP$,
except the trivial upsets $\cU=\emptyset$, $\PP$.

As usual, given a (po)set $\PP$ and a subset $\cA$ of $\PP^n$, a coordinate $i$
is \emph{pivotal} for $\cA$ in a configuration $\omega\in \PP^n$
if changing the $i$th coordinate of $\omega$ can affect whether $\omega\in \cA$.
Let $\cA_i(\omega)\subset \PP$ denote the set of values
that, when substituted for the $i$th coordinate of $\omega$, give
some $\omega'\in \cA$. Thus $i$ is pivotal for $\cA$ in $\omega$
if and only if $\emptyset\ne \cA_i(\omega)\ne \PP$.
If $\cA$ is an upset, then $\cA_i(\omega)$ is an upset.

Given $\Pr_0\prec\Pr_1$, for $0<h<1$ define $\Pr_h$ by linear interpolation:
$\Pr_h(x)=h\Pr_1(x)+(1-h)\Pr_0(x)$ for all $x\in \PP$.
Let $c_0=c_0(\Pr_0,\Pr_1)$ be the minimum of $\Pr_1(\cU)-\Pr_0(\cU)$ over
all non-trivial upsets in $\PP$, so $c_0>0$ by assumption.
Then, for any non-trivial upset $\cU$, we have $\ddh \Pr_h(\cU)\ge c_0$.
Considering partial derivatives in a product with different values
for $h$ in each coordinate, one obtains an analogue of the Margulis--Russo
formula~\cite{Margulis74,Russo81}: if $\cA\subset \PP^n$ is an upset,
then with $c_0=c(\Pr_0,\Pr_1)>0$ as above, we have
\begin{equation}\label{MR}
 \ddh \Pr_h^n(\cA)\ge c_0 \E_h N,
\end{equation}
where $\E_h$ denotes expectation
with respect to the product measure $\Pr_h^n$, and $N=N(\omega)$ is the number of pivotal
coordinates for $\cA$ in the random configuration $\omega$.

Bourgain, Kahn, Kalai, Katznelson and Linial~\cite{BKKKL} showed that
if $X$ is any probability space, and $\cA$ is a subset of $X^n$, then
then there is some coordinate $i$ such that the probability that $i$
is pivotal for $\cA$ is at least $ct(1-t)\log n/n$, where $c>0$ is an absolute
constant, and $t$ is the probability of $\cA$.
As usual, we say that $\cA$ is \emph{symmetric}
if there is a permutation group acting transitively
on $\{1,2,\ldots,n\}$ whose induced action on $X^n$ preserves $\cA$.
If $\cA$ is symmetric, each coordinate has the same probability
of being pivotal, so
the expected number of pivotal elements is at least
$ct(1-t)\log n$. Using \eqref{MR} in place of the usual Margulis--Russo formula,
one then obtains the following result; we omit the simple calculation,
noting that one may take $c_1(\Pr_0,\Pr_1)=c\, c_0(\Pr_0,\Pr_1)/2$.

\begin{theorem}\label{sharp}
Let $\Pr_0$ and $\Pr_1$ be probability measures on a poset $\PP$ with $\Pr_0\prec\Pr_1$.
There is a constant $c_1(\Pr_0,\Pr_1)>0$ with the following property.
Let $0<\eps<1/2$, and let $\cA$ be a symmetric, increasing event in a power $\PP^n$ of $\PP$
with $\Pr_0^n(\cA)>\eps$. If
\[
 c_1(\Pr_0,\Pr_1) \log n \ge \log(1/\eps),
\]
then $\Pr_1^n(\cA)>1-\eps$.\noproof
\end{theorem}

Using Theorem~\ref{sharp} in place of the Friedgut--Kalai result, it is very simple
to adapt (one of) the simple arguments given in~\cite{ourKesten} to
deduce Theorem~\ref{th1p} from Theorem~\ref{rswc}.

\begin{proof}[Proof of Theorem~\ref{th1p}]
Let $\cH(\vecp)$ be a malleable (approximately) self-dual hyperlattice percolation model, with
$\cL$ the corresponding lattice of translational symmetries, and 
let $\vecq\pg\vecp$. Note that since $\cH(\vecp)$ is non-degenerate,
so is $\cH(\vecq')$, where $\vecq'=(\vecp+\vecq)/2$.
If some bottom entries in $\vecq$ are zero, we
replace $\vecq$ by $\vecq'$ in what follows: since $\vecq\pg\vecq'\pg\vecp$,
it suffices to prove percolation in $\cH(\vecq')$. Thus
we may assume without loss of generality that $\cH(\vecq)$ is non-degenerate.

As usual, we wish to work with crossings of rectangles, so it is more convenient
to work with independent lattice colourings.

By Lemma~\ref{regpos} there is
a malleable independent lattice colouring $\cC_0$ realizing $\cH(\vecp)$.
As in Section~\ref{sec_cc} (before Lemma~\ref{reg}), we regard the state space
$\Omega$ underlying the random colouring $\cC_0$ 
as a product of one poset $\PP_F$ for each grey face $F$ of $\cH$:
in the partial order, we have $c_1\preccurlyeq c_2$ if every subface that is black
in $c_1$ is black in $c_2$.
Picking a finite set $F_1,\ldots,F_k$ of faces representing the orbits of $E(\cH)$ (the set
of grey faces) under the action of $\cL$, from lattice invariance
we may regard $\Omega$ as a countable power of the poset
$\PP=\PP_{F_1}\times\cdots\times \PP_{F_k}$. (As usual, the events we consider
in the following arguments will be defined in terms of finite regions
of the plane, and so can be viewed as events in a finite power of $\PP$.)

From independence, the probability measure associated
to $\cC_0$ is a power of a probability measure $\Pr_0$ on $\PP$.
Furthermore, since $\vecp\pl\vecq$, we may choose another measure $\Pr_1$ on $\PP$
with $\Pr_0\prec\Pr_1$ such that the corresponding independent lattice colouring
$\cC_1$ realizes $\cH(\vecq)$. We may and shall
assume that $\cC_1$ is non-degenerate.

Let $\delta(\PP,\Pr,k,\eps)$ be the function appearing in Corollary~\ref{csr}
(our version of the square-root trick), and set
\begin{equation}\label{cd}
 \delta = \delta(\PP,\Pr_1,100,0.01) >0.
\end{equation}
Let $c>0$ be the constant in the $LR_{10}$ property of $\cC_0$;
such a constant exists by Theorem~\ref{rswc}.
Let $\eps>0$ be the minimum of $c$
and $\delta$, and choose $N$ such that
\begin{equation}\label{cn}
 c_1(\Pr_0,\Pr_1) \log N \ge \log(1/\eps),
\end{equation}
where $c_1(\Pr_0,\Pr_1)$ is defined as in Theorem~\ref{sharp}.
Let $D=F_1\cup\cdots\cup F_k$, and
choose $L_0$ so that $L_0^2\ge N\area(D)$, so any region
of area at least $L_0^2$ meets at least $N$ translates of $D$
by elements of $\cL$.

Let $d_1$ be the maximum of the quantity $d_0$ appearing in Lemma~\ref{reg}
and $\diam(D)$, the diameter of $D$. Since $\cC_0$ has the large rectangles
property $LR_{10}$, we can find
an angle $\theta$ and $m,n\ge \max\{L_0,100d_1\}$ such that $h_\theta(10m,n)\ge c$
and $v_\theta(m,10n)\ge c$. In other words, there are orthogonal
vectors $v_1$ and $v_2$ (obtained by rotating $(m,0)$ and $(0,n)$ through
an angle $\theta$) with the following property:
the rectangles $R_1$ with corners $\pm 5v_1\pm v_2/2$
and $R_2$ with corners $\pm v_1/2 \pm 5v_2$ are
such that the probability that $R_i$ has a `long' (parallel to the
$10v_i$ side) black crossing in $\cC_0$ is at least $c$.

By Lemma~\ref{reg}, every point of $\RR^2$ is within distance $d_1$
of some point of $\cL$, so we may find $\ell_1$, $\ell_2\in \cL$ within
distance $d_1$ of $1.1v_1$ and $1.1v_2$, respectively. Let $\tR_1$
be the parallelogram with corners $0$, $8\ell_1$, $8\ell_1+\ell_2$ and $\ell_2$.
Since $v_1$ and $v_2$ have length at least $100d_1$, this parallelogram
is obtained from a translate of $R_1$ by first `distorting it very slightly',
and then
making it significantly shorter and thicker. It is easy to check that a
translate of $\tR_1$ through a suitable lattice element
has the property that any `long' crossing of $R_1$
includes a `long' crossing of $\tR_1$, so, in $\cC_0$, the probability
that $\tR_1$ has a `long' black crossing is at least $c$.
Define $\tR_2$ from $R_2$ similarly.
Since $\ell_i$ is close to $1.1v_i$ and the vectors $v_i$ are not too short
and are orthogonal, the area of $\tR_i$ is (crudely) at least $8|v_1||v_2|\ge 8mn\ge 8L_0^2\ge
8N\area(D)$.

Applying a linear transformation mapping $\ell_1$ to $(1,0)$ and $\ell_2$ to
$(0,1)$, we find that in the transformed model $\cC_0'$, with lattice
of symmetries $\cL'\supset \Z^2$,
the probability that the rectangle $R_1'=[0,8]\times [0,1]$ has a black horizontal
crossing is at least $c$, as is the probability that $R_2'=[0,1]\times [0,8]$
has a black vertical crossing. Note that the image $D'$ of our fundamental
domain $D$ has diameter at most $1/10$, say, since $|\ell_i|\ge 100d_1\ge 100\diam(D)$
and the $\ell_i$ are close to orthogonal. Also, the area
of $D'$ is $\area(D)\times\area(R_1')/\area(\tR_1)\le 1/N$.

Let $\TT$ be the torus obtained by taking the quotient of
$\RR^2$ by the lattice
$10\Z^2$ generated by $(10,0)$ and $(0,10)$. Since $10\Z^2\subset \cL'$,
we may choose $n=100/\area(D')$ translates of $D'$ by elements of $\cL'$
so that their images in $\TT$ cover $\TT$ exactly once. This allows
us to define a natural equivalent of $\cC_0'$ on $\TT$; the corresponding
probability measure may be seen as $\Pr_0^n$, where $n=100/\area(D')\ge 100N$. Moreover,
given a rectangle that does not come `close' to wrapping around the torus,
the events that it has a horizontal black crossing in the plane
or in the torus have the same probability. 

Let $E$ be the event that some translate of $[0,8]\times [0,1]$ in $\TT$
has a black horizontal crossing. Then $\Pr_0^n(E)\ge c$, and $E$
is a symmetric, increasing event in $\Pr_0^n$ in the sense of Theorem~\ref{sharp}.
Since $n\ge N$, from our choice \eqref{cn} of $N$ and Theorem~\ref{sharp},
we have $\Pr_1^n(E)\ge 1-\delta$.
As in~\cite{ourKesten}, let $R_1,\ldots,R_{100}$ be translates in $\TT$
of the rectangle $[0,6]\times [0,2]$ arranged so that any $8$-by-$1$
rectangle crosses some $R_i$ horizontally. Then we have $\Pr_1^n(\bigcup\Hb(R_i))\ge 1-\delta$,
so by Corollary~\ref{csr} and our choice \eqref{cd} of $\delta$
we have $\Pr_1^n(\Hb(R_i))\ge 0.99$ for some $i$, and thus for all~$i$.

Translating back to the plane, we see that in $\cC_1'$ (obtained from $\cC_1$
by the linear transformation mapping $\cC_0$ to $\cC_0'$),
any $6$-by-$2$ rectangle $R$ with corners at points of the lattice $\cL'$ has $\Pr(\Hb(R))\ge 0.99$.
The same argument shows that any $2$-by-$6$ rectangle $R'$
with lattice point corners has $\Pr(\Vb(R'))\ge 0.99$. From here it is
very easy to prove that percolation occurs, using any of several
standard methods; we shall give one example.

Note that if $S$ is a $2$-by-$2$ square then, with $\Pr$ denoting
the probability measure associated to $\cC_1'$, we have $\Pr(\Hb(S))\ge \Pr(\Hb(R))\ge 0.99$
and $\Pr(\Vb(S))\ge \Pr(\Vb(R'))\ge 0.99$.
As in~\cite{ourKesten} (the third version of the proof of Theorem 10 there),
let $G(R)$ be the event that $\Hb(R)$ holds and each of the two
$2$-by-$2$ `end squares' of $R$ has a black vertical crossing,
and define $G(R')$ similarly. Then $\Pr(G(R))$, $\Pr(G(R'))\ge 1-3(1-0.99)=0.97$.
Of course, the bound $0.97$ here can be replaced by any constant less
than $1$, although, as we shall see, $0.97$ is more than good enough.

Considering a square grid of $6$-by-$2$ and $2$-by-$6$ rectangles
overlapping in $2$-by-$2$ squares as in~\cite{ourKesten}, and taking a bond of $\Z^2$
to be open if $G(R)$ holds for the corresponding rectangle,
one obtains a dependent bond percolation measure on $\Z^2$.
Given sets $S$ and $T$ of bonds of $\Z^2$ separated by a distance
(in the graph $\Z^2$) of at least 1, the corresponding unions
of rectangles are disjoint, and are separated in the plane by a distance
of at least $2\ge 1/10$. It follows that the states of the bonds in $S$ are
independent of the sates of the bonds in $T$, i.e.,
the bond percolation measure is \emph{$1$-independent}. It is rather
easy to see that any such measure in which each bond is open
with high enough probability has an infinite open cluster with probability $1$;
see, for example, the general domination result of Liggett, Schonmann and Stacey \cite{LSS}.
The best current bound on what `high enough' means is due to 
Balister, Bollob\'as and Walters~\cite{BBWsquare}, who showed that all bond probabilities
at least $0.8639$ will do. Since $\Pr(G(R))$, $\Pr(G(R'))\ge 0.97$,
we see that with probability $1$ there is an infinite open cluster
in $\Z^2$. Translating back, the definition of $G(R)$ ensures
that we find a corresponding infinite black cluster
in $\cC_1'$. Since $\cC_1'$ is simply a linear image of $\cC_1$,
it follows that $\cC_1$ contains an infinite black cluster
with probability $1$; hence $\cH(\vecq)$ percolates,
as required.

It remains to establish exponential decay of the volume in $\cH(\vecq)$
for $\vecq\pl\vecp$. But first note that with $\vecq\pg\vecp$ as above,
and with the arbitrary constant $0.99$ replaced by a suitable constant $a<1$,
the argument in~\cite[Section 3]{ourKesten2} (again using
locally-dependent percolation) shows that the dual of the model
$\cC'_1$ exhibits exponential decay (of the volume). It follows that 
the dual $\cHd(\vecq^*)$ of $\cH(\vecq)$ exhibits exponential decay.

Given a self-dual model $\cH(\vecp)$ and $\vecq\pl\vecp$,
the model $\cHd(\vecp^*)$ is self-dual (it is isomorphic to $\cH(\vecp)$),
and $\vecq^*\pg\vecp^*$. Applying the result above to $\cHd(\vecp^*)$ and $\cHd(\vecq^*)$,
we see that the dual $\cH(\vecq)$ of $\cHd(\vecq^*)$ exhibits exponential decay,
as required.

Suppose instead that $\cH(\vecp)$ is approximately self-dual. Then by definition
there is a model $\cH'(\vecp')$ that is isomorphic to $\cH(\vecp)$, such
that $\cHd(\vecp^*)$ and $\cH'(\vecp')$ are equivalent, in the sense
that they may be coupled so that for any open path in either model there
is a nearby open path in the other model. Recall also
that we may take the isomorphism to be given by an isometry of the plane (plus
a small `distortion', if needed).
Taking the colouring viewpoint,
open paths are simply black paths. Now the notion of equivalence
is not obviously preserved under taking duals, i.e., inverting the colouring.
However, the condition for approximate self-duality is exactly that for
every {\em white} path in $\cH(\vecp)$ (corresponding to a black path in $\cHd(\vecp^*)$),
there is a nearby {\em black} path in $\cH'(\vecp')$, and vice versa.
Since $\cH(\vecp)$ and $\cH'(\vecp')$ are isomorphic, $\cH'(\vecp')$ also
satisfies this condition, so we may couple $\cH'(\vecp')$ and $\cH(\vecp)$
so that for every white path in $\cH'(\vecp')$ there is a nearby black
path in $\cH(\vecp)$, and vice versa. In other words,
approximate self-duality holds after interchanging black and white, i.e.,
$\cHd(\vecp^*)$ is approximately self-dual. From this point
the argument for exponential decay is as in the self-dual case.
\end{proof}

\section{On the critical surface}\label{sec_crit}

In the bulk of this paper we have shown that any self-dual hyperlattice
percolation model $\cH(\vecp)$ is `critical' in the sense
that if $\vecq\pg\vecp$ then $\cH(\vecq)$ percolates,
while if $\vecq\pl\vecp$ then $\cH(\vecq)$ exhibits exponential decay.

As noted earlier, the model $\cH(\vecp)$ itself may or may not percolate.
Here we show that, except for degenerate cases, it does not.
Furthermore, we show that one has power-law decay of the radius, as
expected.
Let $v_0$ be any fixed vertex of $\cH$;
we write $v_0\to r$ for the event that there is an open
path from $v_0$ to a vertex at distance at least $r$ from $v_0$.

\begin{theorem}\label{crit}
Let $\cH(\vecp)$ be a malleable self-dual hyperlattice percolation model.
Then there are constants $0<a_1<a_2$ and $r_0$ such that
$r^{-a_2}\le \Pr(v_0\to r)\le r^{-a_1}$ for all $r\ge r_0$.
\end{theorem}

Note that since we argue directly about properties of the
self-dual case, we need to impose the technical condition of malleability
defined in Definition~\ref{maldef2}. It seems
likely that this can be weakened to non-degeneracy; the latter
condition is used throughout the proofs in the previous sections, whereas
malleability is only used at one point, where the need for it could
perhaps be circumvented.

\begin{proof}
By Lemma~\ref{regpos}, $\cH(\vecp)$ can be realized by a malleable
independent lattice colouring $\cC$.
By Theorem~\ref{rswc}, $\cC$ has the large rectangles property and hence,
by Lemma~\ref{r_e}, the large ellipses property.

Let $S$ and $T$ be the maps exhibiting self-duality, as described
in Lemma~\ref{lsymclass}.
Let $L=100\max\{d_0,d_1\}$, where $d_0$ is the constant in Lemma~\ref{reg} and $d_1$ is the
bound on $|S(x)-T(x)|$ from Lemma~\ref{lsymclass}.

The large ellipses property tells us that
there are constants $c>0$ and $A_0$ such that for every
$A\ge A_0$ there is an ellipse $E=E(A)$ centred on the origin with area $A$
and with both axes having length at least $L$ such that $\Pr(\Sb(E))\ge c$.
Let $N$ be an integer such that $(1-c)^N<c$.

For $i\ge 0$ let $E_i=E((10N)^iA_0)$ be an ellipse as above with area $(10N)^iA_0$.
We claim that for each $i$, $T(E_i)$ fits inside a copy of $E_{i+1}$
scaled by a factor $1/2$. To see this, rotate and scale so that the ellipse $T(E_i)$
becomes a circle with diameter $1$, and $E_{i+1}$ has horizontal major axis
with length $a$ and vertical minor axis with length $b$.
Our claim is exactly that $b\ge 2$. But if not,
then $b\le 2$ so, since $\area(E_{i+1})=10N \area(E_i)=10N\area(T(E_i))$, we have
$a\ge 5N$. Writing $E^0$ for the annulus between an ellipse $E$
and the concentric ellipse $2E/3$,
it is not hard to see that when $b\le 2$ and $a\ge 5N$ one can arrange $N$ disjoint copies
of $T(E_i)$ to `cross' the annulus $E_{i+1}^0$ as in Figure~\ref{fig_ellipse}.
\begin{figure}[htb]
\[
 \epsfig{file=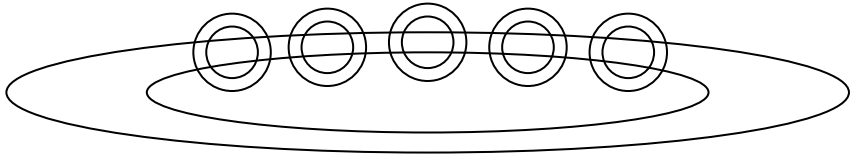}
\]                                             
 \caption{Circles `crossing' an elliptical annulus. Depending on the height
of the ellipse, the circles may or may not intersect the bottom half
of the ellipse; this is irrelevant for the argument.}\label{fig_ellipse}
\end{figure}
Moreover, in the rescaled arrangement one can easily ensure that the circles are separated
by a distance of at least $1/10$, say, and each `crosses' $E_{i+1}^0$ even after any transformation
moving points by a distance of at most $1/10$. Since lengths with original scale $\max\{d_0,d_1\}$
have transformed scale at most $1/100$, it follows that each image $S(E_i)$ crosses $E_{i+1}^0$,
and the $S(E_i)$ are separated by a distance of at least $d_0$.
But by self-duality each $S(E_i)$ contains a \emph{white} path surrounding
its centre with probability at least $c$, and these events are independent.
Whenever one of these white paths is present, $\Sb(E_{i+1})$ cannot hold.
Hence $\Pr(\Sb(E_{i+1})) \le (1-c)^N <c$, a contradiction.

Recalling from Lemma~\ref{lsymclass} that $T^2$ is either the identity
or reflection in the origin, we see that
$E_{i+2}^0$ surrounds $T(E_{i+1}^0)$ which surrounds
$T(T(E_i^0))=E_i^0$. Hence the annuli $E_{2i}^0$ are disjoint.
Moreover, since $1/2<2/3$ and all axes of all $E_i$ have length at least $L$,
the annuli $E_{2i}^0$ 
are separated by distances of at least $d_0$, and so meet disjoint sets of faces of $\cH$.
Now the shorter axis of $E_i$ has length at least $L\ge 1$, so its longer axis
has length at most $(10N)^i A_0$. It follows that for large $r$, any point at distance
$r$ from the origin is outside $E_k$, where $k=\Theta(\log r)$.
There can only be a \emph{white} path starting at the origin and ending at least
distance $r$ away if none of the events $\Sb(E_{2i})$, $2i<k$, holds.
Since these events are independent, this has probability at most $(1-c)^{\floor{k/2}} =\exp(-\Theta(\log r))$.
Passing to the dual, this proves the upper bound on $\Pr(v_0\to r)$. 

The lower bound is essentially immediate, with $1$ as the exponent. Indeed,
Lemma~\ref{oneway} tells us that if $R$ is an $r$-by-$r$ square, then
either $\Pr(\Hb(R))\ge 1/2$ or $\Pr(\Vw(R))\ge 1/2$.
Suppose first that $\Pr(\Hb(R))\ge 1/2$. Then considering one of the $O(r)$ translates of our
reference vertex $v_0$ within distance $O(1)$ of the left-hand side of $R$,
by the union bound there is a translate of $v_0$ such that the probability
that there is an open path starting at $v_0$ with length at least $r-O(1)$
is at least $O(1/r)$.

If $\Pr(\Vw(R))\ge 1/2$ then we apply the same
argument in the dual, which is isomorphic to the original percolation model.
\end{proof}

It is easy to see that the argument above extends to the approximately self-dual case.
The key point is that there is enough `elbow room' for small distortions of the paths
considered not to matter.

It seems very probable that the general conjecture of 
Aizenman and Langlands, Pouliot and Saint-Aubin~\cite{Langlands_confinvar}
concerning conformal invariance of the scaling limit of critical
plane percolation will hold for all non-degenerate self-dual hyperlattice percolation
models. However, this is likely to be very hard to prove.
This conjecture asserts, among other things, that if $R$ is any
rectangle, $\lambda R$ denotes its image under a dilation with scale-factor
$\lambda$, and $H(R)$ denotes the event
that $R$ has an open (or here, black) horizontal crossing,
then for any fixed $R$, the limit
$\lim_{\lambda\to\infty} \Pr(H(\lambda R))$ exists
and lies strictly between $0$ and $1$. Moreover, this limit
should be given by Cardy's formula~\cite{Cardy92}, after first applying a suitable
linear transformation to the model.

This conjecture has been proved by Smirnov~\cite{Smirnov}
for site percolation on the triangular lattice; this is essentially
the only case known.
For many other models, such as bond percolation on the square lattice, 
RSW-type theorems give the much weaker result that
\begin{equation}\label{rbounded}
 0<\liminf_{\la\to\infty} \Pr(H(\lambda R)) \le
 \limsup_{\la\to\infty} \Pr(H(\lambda R)) <1.
\end{equation}
(This applies just as well to shapes other than rectangles.)
Unfortunately, Theorem~\ref{rswc}, while strong enough to establish criticality,
lacks the uniformity needed to prove \eqref{rbounded},
so we leave this as a conjecture.

\begin{conjecture}
Let $\cH(\vecp)$ be a non-degenerate (approximately) self-dual hyperlattice percolation model.
Then for any fixed rectangle $R$, the bounds \eqref{rbounded} hold.
\end{conjecture}

As far as we are aware, this conjecture is open even for the simple
special case of inhomogeneous bond percolation on the square lattice,
where each horizontal bond is open with probability $p$ and each
vertical one with probability $1-p$, with the states of all bonds independent.

\bigskip\noindent
{\bf Acknowledgements.}
We are grateful to Robert Ziff for bringing the results of Scullard and himself
to our attention: this paper started from an attempt to show that the self-duality
they established in certain (quite general) cases does imply criticality.
The last section of this paper was added in response to a question asked
by Marek Biskup.

\end{document}